\documentclass[reqno]{amsart}%[a4paper,10pt,reqno]{amsart}

\usepackage{mathrsfs, amsfonts, amsmath, amssymb, amscd}
\usepackage{amsbsy,bm, amsthm, enumerate}

\usepackage[english]{babel}
\usepackage{amsmath,amssymb,amsthm}
\usepackage{hyperref}
\usepackage[centering]{geometry}
\usepackage{xcolor}

\renewcommand{\epsilon}{\varepsilon}            %%  changes epsilon
\newtheorem{theorem}{Theorem}[section]   %% definition of theorem environment
\newtheorem*{theorem*}{Theorem}          %% a theorem environment without numbering
\newtheorem{lemma}[theorem]{Lemma}
\newtheorem{proposition}[theorem]{Proposition}

\theoremstyle{definition}
\newtheorem{definition}[theorem]{Definition}
\newtheorem{corollary}[theorem]{Corollary}
\newtheorem{example}[theorem]{Example}

\newtheorem{remark}{Remark}[section]

\numberwithin{equation}{section}

\title[On Hopf hypersurfaces of the complex quadric]
{On Hopf hypersurfaces of the complex quadric with constant principal curvatures}

\author{Haizhong Li, Hiroshi Tamaru \and Zeke Yao}



\date{}
\begin{document}
\begin{abstract}
In this paper, we classify the Hopf hypersurfaces of the complex quadric $Q^m=SO_{m+2}/(SO_2SO_m)$ ($m\geq3$) with at most five distinct constant principal curvatures. 
We also classify the Hopf hypersurfaces of $Q^m$ ($m=3,4,5$) with constant principal curvatures.  
All these real hypersurfaces are open parts of homogeneous examples.
\end{abstract}

\maketitle

%========================================
\section{Introduction}\label{sect:1}

Let $\mathbb{C}P^{m+1}$ be the complex projective space that is endowed with
the Fubini-Study metric $g$ such that its holomorphic sectional curvature is $4$. The complex
quadric is a compact Einstein complex hypersurface of $\mathbb{C}P^{m+1}$ which is defined
by $Q^m:=\{[(z_1,z_2,\ldots,z_{m+2})]\in\mathbb{C}P^{m+1}:\,z_1^2+z_2^2+\cdots+z_{m+2}^2=0\}$.
Moreover, $Q^m=SO_{m+2}/(SO_2SO_m)$ can be identified with the Grassmannian manifold of oriented
$2$-dimensional linear subspaces of $\mathbb{R}^{m+2}$; associated with the induced metric
(denoted still by) $g$, it is a compact Hermitian symmetric space of rank $2$ (see Reckziegel \cite{R},
Smyth \cite{Sm} and Berndt-Suh \cite{BS}). In order to describe the Riemannian curvature
tensor of $Q^m$, it is necessary to notice that, besides the K\"ahler structure $J$ that is
induced from that of $\mathbb{C}P^{m+1}$, $Q^m$
carries a special structure $A$, called the {\it almost product structure} or the {\it complex
conjugation structure}, and they are related by $AJ=-JA$ (see \cite{BS,BS-2022,LMVVW,R}). As the complex quadric $Q^m$ is an important Riemannian manifold, the study of its
submanifolds is significant and has attracted many geometers.
%Here, focusing on the study of the hypersurfaces of $Q^m$, we refer to
%\cite{BS-2012,BS,BS-2015,BS-2022,K,R,Suh-2015,Ur} and the many later references citing them.

%
Recall that for an almost Hermitian manifold $\overline{M}$ with almost complex structure $J$,
a connected orientable real hypersurface $M$ of $\overline{M}$ is associated with an important
notion the {\it Reeb vector field} defined by $\xi:=-JN$, where $N$ is the unit normal
vector field of $M$. If the integral curves of $\xi$ are geodesics, then $M$ is called a {\it
Hopf hypersurface}. In particular, a real hypersurface of $Q^m$ being Hopf is equivalent to that
its Reeb vector field $\xi$ is a principal vector field corresponding to
Reeb function $\alpha$ (cf. \cite{BBW}). During the last
four decades, Hopf hypersurfaces of the non-flat complex space forms and several other almost Hermitian
manifolds have been extensively and deeply investigated. % for details we refer to, e.g.,
%\cite{B,CR,KM,M,NR} and \cite{BBW,BS,HJZ,HY,HYZ,Suh2} and the references therein.
%
For example, in complex projective space $\mathbb{C}P^m$ ($m\geq2$), Kimura \cite{KM} classified
the Hopf hypersurfaces with constant principal curvatures, in particular, the number
of distinct principal curvatures has only three possibilities: $2$, $3$, and $5$. In
complex hyperbolic space $\mathbb{C}H^m$ ($m\geq2$), Berndt \cite{B} classified the
Hopf hypersurfaces with constant principal curvatures, and in this case the number of
distinct principal curvatures has exactly two possibilities: $2$ and $3$.
Their results are very important in the study of real hypersurfaces of $\mathbb{C}P^m$ and $\mathbb{C}H^m$.
%The same problem is still open in $Q^m$.
Now, we focus on the study of Hopf hypersurfaces in the complex quadric $Q^m$.
It is well known that the complex quadric $Q^2$
with Einstein constant $1$ is holomorphically isometric to the K\"ahler surface
$\mathbb{S}^2\times\mathbb{S}^2$. Urbano \cite{Ur} classified the homogeneous hypersurfaces, the isoparametric hypersurfaces and the hypersurfaces of $\mathbb{S}^2\times\mathbb{S}^2$ with at most two distinct constant principal curvatures, and
give some partial classification results about hypersurfaces of $\mathbb{S}^2\times\mathbb{S}^2$ with three distinct constant principal curvatures. The classification of
Hopf hypersurfaces of $\mathbb{S}^2\times\mathbb{S}^2$ with constant principal curvatures
can be obtained by Theorem 1.3 of \cite{ZGHY}.
Berndt and Suh \cite{BS} classified the real hypersurfaces of $Q^m$ ($m\geq3$) with isometric Reeb flow. They also classified the contact hypersurfaces of $Q^m$ ($m\geq3$) in \cite{BS-2022}.
For more studies on real hypersurfaces of $Q^m$, we refer to, e.g.,
\cite{BS-2012, BS, BS-2015, BS-2022, LVWY, Loo11, Suh1} and the references therein. Nevertheless, many fundamental problems remain open.
For example, we have the following natural and interesting problem:

\vskip 2mm

\noindent {\bf Problem} Classify all Hopf hypersurfaces of $Q^m$ ($m\geq3$) with constant principal curvatures.

\vskip 2mm

Recall that, the inverse image $\pi^{-1}(M)$ of a Hopf hypersurface $M$ of $\mathbb{C}P^m$ with constant principal curvatures under the Hopf map $\pi$ is an isoparametric hypersurface of $\mathbb{S}^{2m+1}$, then using the known results of isoparametric hypersurfaces of the unit sphere, one can determine the number and the multiplicities of distinct constant principal curvatures
of $M$ (cf. \cite{KM}). On the other hand, by establishing the Cartan's formula on
a Hopf hypersurface $M$ of $\mathbb{C}H^m$ with constant principal curvatures, and
through careful analysis, one can also determine the number and the multiplicities of distinct constant principal curvatures (cf. \cite{B}).
%In the following, let $M$ be a Hopf hypersurface of $Q^m$ ($m\geq3$) with constant principal curvatures.
The situation of the Hopf hypersurface of $Q^m$ ($m\geq3$) with constant principal curvatures
is much more complicated.
According to Lemma \ref{lemma:2.5}, a Hopf hypersurface in $Q^m$ $(m\ge3)$ with constant
principal curvatures has either $\mathfrak{A}$-principal unit normal vector field $N$ or $\mathfrak{A}$-isotropic unit normal vector field $N$
(see definition \ref{def:2.1}).
Moreover, a
Hopf hypersurface in $Q^m$ $(m\ge3)$ with $\mathfrak{A}$-principal unit normal vector field $N$ is an open part of a tube over a totally geodesic $Q^{m-1}\hookrightarrow Q^{m}$
(see Theorem \ref{thm:2.1}). Thus, we only need to consider the Hopf hypersurfaces with constant principal curvatures and $\mathfrak{A}$-isotropic unit normal vector field $N$.
We can prove that $M$ is an
isoparametric hypersurface of $Q^m$ ($m\geq3$) (see Theorem \ref{thm:4.2w}). Then, using the relationship between Stiefel manifold $V_2(\mathbb{R}^{m+2})$
and $Q^m$, the inverse image $\pi^{-1}(M)$ of $M$ under the Hopf map $\pi$ is an isoparametric hypersurface of Stiefel manifold $V_2(\mathbb{R}^{m+2})$.
However, the research on isoparametric hypersurfaces of Stiefel manifold is still poorly understood.
So we need to find another approach.
Note that, we prove that all parallel hypersurfaces of $M$ still have constant principal
curvatures (see Theorem \ref{thm:4.2w}). This allows us to apply a theorem obtained by Ge-Tang \cite{GT}, 
it follows that the focal submanifolds of $M$ are austere. At the same time, we also establish
two Cartan's formulas (see \eqref{eqn:ca2} and \eqref{eqn:call}) for any Hopf hypersurface of $Q^m$ ($m\geq3$) with constant principal curvatures and $\mathfrak{A}$-isotropic unit normal vector field $N$.
In contrast to the situation of $\mathbb{C}H^m$, the Cartan's formulas of $Q^m$ also involve
the almost product structure $A$, and it is difficult to make clear the interaction among the shape operator $S$, the complex structure $J$ and the almost product structure $A$.
Considering the important role played by the almost product structure $A$ in the study of real hypersurfaces of $Q^m$, we find a very useful lemma which states that the almost product structure $A$ can project an eigenspace $V_{\lambda}$ onto the orthogonal complement of $V_{\lambda}\oplus JV_{\lambda}$ (see Lemma \ref{lemma:5.1}).

In this paper, we study the classification problem of Hopf hypersurfaces of $Q^m$ with constant principal curvatures from two aspects.
When restricting the number of distinct principal curvatures to be at most five,
we obtain the following result.

\begin{theorem}\label{thm:1.1a}
Let $M$ be a Hopf hypersurface of $Q^m$ ($m\geq3$) with at most five distinct constant principal curvatures. Then,
\begin{enumerate}
\item[(1)]
$M$ is an open part of a tube over a totally geodesic $Q^{m-1}\hookrightarrow Q^{m}$ ($m\geq3$); or

\item[(2)]
$M$ is an open part of a tube over a totally geodesic $\mathbb{C}P^k\hookrightarrow Q^{2k}$ ($m=2k$, $k\geq2$).
\end{enumerate}
\end{theorem}

\begin{remark}\label{rem:1.1}
In Theorem \ref{thm:1.1a}, the condition of at most five distinct constant principal curvatures
is necessary. In fact, there exist Hopf hypersurfaces of $Q^{4k-2}$, $k\geq 2$, with six distinct constant principal curvatures, such as Example \ref{E3}.
In addition, the condition of constant principal curvature is also necessary. In fact, we can construct a complex submanifold $\tilde{N}^{2m-2}_k$ of $Q^m$ ($m\geq3$), the nearby tubes of $\tilde{N}^{2m-2}_k$ are Hopf hypersurfaces with
non-constant principal curvatures, see Example \ref{E6}.
\end{remark}

Then, we restrict the dimension of $Q^m$ to $m=3,4,5,6$.
As a direct consequence of Theorem \ref{thm:1.1a},
we classify the Hopf hypersurface of $Q^3$ with constant principal curvatures.

\begin{corollary}\label{cor:1.1}
Let $M$ be a Hopf hypersurface of $Q^3$ with constant principal curvatures. Then,
$M$ is an open part of a tube over a totally geodesic $Q^{2}\hookrightarrow Q^{3}$.
\end{corollary}

For the complex quadrics $Q^4$ and $Q^5$, we can get the following classification results.
\begin{theorem}\label{thm:1.3}
Let $M$ be a Hopf hypersurface of $Q^4$ with constant principal curvatures. Then,
\begin{enumerate}
\item[(1)]
$M$ is an open part of a tube over a totally geodesic $Q^{3}\hookrightarrow Q^{4}$; or

\item[(2)]
$M$ is an open part of a tube over a totally geodesic $\mathbb{C}P^2\hookrightarrow Q^{4}$.
\end{enumerate}
\end{theorem}

\begin{theorem}\label{thm:1.4}
Let $M$ be a Hopf hypersurface of $Q^5$ with constant principal curvatures. Then,
$M$ is an open part of a tube over a totally geodesic $Q^{4}\hookrightarrow Q^{5}$.
\end{theorem}

In the complex quadrics $Q^6$, we find that there are Hopf hypersurfaces with six distinct constant principal curvatures.
In fact, in Section \ref{sect:3}, we describe the homogeneous real hypersurface of type (C) in Theorem 6.3.1 of
\cite{BS-2022} in detail (see Example \ref{E3}), and by Proposition \ref{prop:E3P}, we know that the tubes of radius $0<r<\frac{\pi}{4}$ around the equivariant embedding of the $(8k-7)$-dimensional homogeneous space $Sp_kSp_1/SO_2Sp_{k-2}Sp_1$ in $Q^{4k-2}$, $k\geq 2$, are Hopf hypersurfaces with six distinct constant principal curvatures. Now, we state the result for $Q^6$.

\begin{theorem}\label{thm:1.5}
Let $M$ be a Hopf hypersurface of $Q^6$ with constant principal curvatures. Then,
\begin{enumerate}
\item[(1)]
$M$ is an open part of a tube over a totally geodesic $Q^{5}\hookrightarrow Q^{6}$; or

\item[(2)]
$M$ is an open part of a tube over a totally geodesic $\mathbb{C}P^3\hookrightarrow Q^{6}$; or

\item[(3)]
$M$ has six distinct constant principal curvatures. Their values and multiplicities
are given by
$$
\begin{tabular}{|c|c|c|c|c|c|c|}
  \hline
  {\rm value} & $2\tan(2r)$ & $0$ & $\tan(r)$ & $-\cot(r)$ & $\frac{\cos(r)+\sin(r)}{\cos(r)-\sin(r)}$ & $-\frac{\cos(r)-\sin(r)}{\cos(r)+\sin(r)}$\\
  \hline
  {\rm multiplicity} & $1$ & $2$ & $2$ & $2$ & $2$ & $2$\\
  \hline
\end{tabular}
$$
%It holds $\mathcal{Q}=V_{\lambda_{\tan(t)}}\oplus V_{\lambda_{-\cot(t)}}\oplus V_{\lambda_{\frac{\cos(t)+\sin(t)}{\cos(t)-\sin(t)}}}\oplus V_{\lambda_{-\frac{\cos(t)-\sin(t)}{\cos(t)+\sin(t)}}}$. The almost product structure $A\in\mathfrak{A}$ maps
%$V_{\lambda_{\tan(t)}}\oplus V_{\lambda_{-\cot(t)}}$ into $V_{\lambda_{\frac{\cos(t)+\sin(t)}{\cos(t)-\sin(t)}}}\oplus V_{\lambda_{-\frac{\cos(t)-\sin(t)}{\cos(t)+\sin(t)}}}$, and vice versa.
where $0<r<\frac{\pi}{4}$. The Reeb function is $\alpha=2\tan(2r)$. In particular, Example \ref{E3} for $k=2$ is contained in this case.
\end{enumerate}
\end{theorem}

\begin{remark}\label{rem:1.2}
The examples appeared in (1) and (2) of Theorem \ref{thm:1.1a} are open parts of 
homogeneous real hypersurfaces. 
The homogeneous real hypersurfaces of the complex quadric $Q^m$ ($m\geq3$) were classified
by Kollross in \cite{Ko} (see also Theorem 6.3.1 of \cite{BS-2022}). There are four types of homogeneous real hypersurfaces, they are Examples \ref{E1}, \ref{E2}, \ref{E3} and the tube of radius $0<r<\frac{\pi}{4}$ around the equivariant embedding of the $20$-dimensional
homogeneous space $Spin_9/SO_2Spin_6$ in $Q^{14}$.
%We will give a detailed calculation of Examples \ref{E1}, \ref{E2}, \ref{E3} in Section \ref{sect:3}.
Note that, Examples \ref{E1} and \ref{E2} has been studied extensively,
there have been many articles giving the characterizations of these two types of examples,
we refer to \cite{BS-2012, BS, BS-2015, BS-2022} and the references therein.
As far as we know, there is no paper that gives the characterization of the Example \ref{E3}.
\end{remark}

The paper is organized as follows. In Section \ref{sect:2}, we collect some basic
properties of $Q^m$ ($m\geq3$) and some preliminaries of the geometry of the real hypersurfaces of $Q^m$.
In Section \ref{sect:3}, we introduce three types of the homogeneous real hypersurfaces, and we also construct a new family of Hopf hypersurfaces of $Q^m$ ($m\geq3$) with non-constant principal curvatures.
In Section \ref{sect:4}, we study the parallel hypersurfaces and focal submanifolds of Hopf hypersurfaces of $Q^m$ ($m\geq3$) with constant principal curvatures and $\mathfrak{A}$-isotropic unit normal vector field $N$, we also establish two Cartan's formulas for Hopf hypersurfaces with constant principal curvatures and $\mathfrak{A}$-isotropic unit normal vector field $N$. Section \ref{sect:5} is dedicated to the proof of Theorem \ref{thm:1.1a}. Section \ref{sect:6} is dedicated to
the proofs of Corollary \ref{cor:1.1} and Theorems \ref{thm:1.3}--\ref{thm:1.5}.

\textbf{Acknowledgments:}
H. Li was supported by NSFC Grant No.12471047. H. Tamaru was supported by JSPS KAKENHI Grant Numbers JP22H01124 and JP24K21193, and was also partly supported by MEXT Promotion of Distinctive Joint
Research Center Program JPMXP00723833165. Z. Yao was supported by NSFC Grant No. 12401061.

%========================================
\section{Preliminaries}\label{sect:2}

\subsection{The complex quadric}\label{sect:2.1}~

Let $\mathbb{C}P^{m+1}$ be the $(m+1)$-dimensional complex projective space
equipped with the Fubini-Study
metric $g$ of constant holomorphic sectional curvature $4$. The Hopf map
$$
\pi:\mathbb{S}^{2m+3}(1)\rightarrow \mathbb{C}P^{m+1}: z\mapsto [z]
$$
is a Riemannian submersion from $\mathbb{S}^{2m+3}(1)$ to $\mathbb{C}P^{m+1}$.
For any $z\in \mathbb{S}^{2m+3}(1)$ we have $\pi^{-1}([z])=\{e^{\mathbf{i}t}z|t\in \mathbb{R}\}$
and ${\rm ker}(d\pi)_z={\rm Span}\{\mathbf{i}z\}$. The complex
structure $J$ on $\mathbb{C}P^{m+1}$ is induced from multiplication by
$\mathbf{i}$ on $T\mathbb{S}^{2m+3}(1)$. The complex quadric as the complex hypersurface of $\mathbb{C}P^{m+1}$ is defined by
$$
Q^m=\{[(z_1,z_2,\ldots,z_{m+2})]\in \mathbb{C}P^{m+1}|z_1^2+z_2^2+\cdots+z_{m+2}^2=0\}.
$$
$Q^m$ is equipped with the induced metric, which we still denote by $g$, and the induced almost
complex structure, which we still denote by $J$, then $(Q^m, g, J)$ is a K\"{a}hler manifold.
The inverse image of $Q^m$ under the Hopf map is the $2m+1$-dimension Stiefel manifold
$$
V_2(\mathbb{R}^{m+2})=\{z=u+\mathbf{i}v|u,v\in \mathbb{R}^{m+2}, \langle u,u\rangle=\langle v,v\rangle=\frac{1}{2},
\langle u,v\rangle=0\}\subset\mathbb{S}^{2m+3}(1),
$$
where $\langle\cdot,\cdot\rangle$ denotes the Euclidean inner product on $\mathbb{R}^{m+2}$.
The normal space to $V_2(\mathbb{R}^{m+2})$ in $\mathbb{S}^{2m+3}(1)$ at a point $z$ is spanned by
$\bar{z}$ and $\mathbf{i}\bar{z}$, which implies that
the normal space to $Q^m$ in $\mathbb{C}P^{m+1}$ at a point $[z]$ is spanned by $(d\pi)_z(\bar{z})$ and $J(d\pi)_z(\bar{z})=(d\pi)_z(\mathbf{i}\bar{z})$,
where $z$ is any representative of $[z]$. %These vectors depend on the chosen representative $z$.
We denote by $\mathfrak{A}$ the set of all shape operators of $Q^m$ in $\mathbb{C}P^{m+1}$ associated with unit normal vector fields, and $\mathfrak{A}$ is a collection of $(1,1)$-tensor fields on $Q^m$.
Then, $\mathfrak{A}$ is an $S^1$-subbundle of the endomorphism
bundle End($TQ^m$). In summary, we have
\begin{lemma}[cf. \cite{R,Sm}]\label{lemma:2.1}
For each $A\in\mathfrak{A}$, it holds that
\begin{equation}\label{eqn:2.1}
A^2=Id, \ \ g(AX,Y)=g(X,AY), \ \ AJ=-JA, \ \ \forall X,Y\in TQ^m.
\end{equation}
\end{lemma}

$\mathfrak{A}$ is commonly called consisting
of complex conjugations (cf. \cite{R}). Lemma \ref{lemma:2.1} implies in
particular that $\mathfrak{A}$ is a family of {\it almost product structures}
on $Q^m$ (cf. \cite{LMVVW}) and we will use the latter term in this paper.
Furthermore, since the second fundamental form of the embedding $Q^m\hookrightarrow\mathbb{C}P^{m+1}$
is parallel, $\mathfrak{A}$ is a parallel subbundle of End($TQ^m$). It follows
that, for each almost product structure $A\in\mathfrak{A}$, there exists a
one-form $q$ on $Q^m$ such that it holds the relation (see \cite{Sm})
\begin{equation}\label{eqn:2.2}
(\bar{\nabla}_XA)Y=q(X)JAY,\ \ \forall\, X,Y\in TQ^m,
\end{equation}
where $\bar{\nabla}$ denotes the Levi-Civita connection of $Q^m$.

The Gauss equation for the complex hypersurface $Q^m\hookrightarrow \mathbb{C}P^{m+1}$
implies that the Riemannian curvature tensor $\bar{R}$ of $Q^m$ can be expressed in terms
of the Riemannian metric $g$, the complex structure $J$ and a generic almost product
structure $A\in\mathfrak{A}$ as follows (cf. \cite{R}):
\begin{equation}\label{eqn:2.3}
\begin{aligned}
\bar{R}(X,Y)Z=&g(Y,Z)X-g(X,Z)Y\\
&+g(JY,Z)JX-g(JX,Z)JY-2g(JX,Y)JZ\\
&+g(AY,Z)AX-g(AX,Z)AY\\
&+g(JAY,Z)JAX-g(JAX,Z)JAY,
\end{aligned}
\end{equation}
where, it should be noted that $\bar{R}$ is independent of the special choice of $A\in\mathfrak{A}$.

According to Reckziegel \cite{R}, a real $2$-dimensional linear subspace $\sigma$
of $T_{[z]}Q^m$ is called a $2$-flat if the curvature tensor $\bar{R}$ of $Q^m$ vanishes
identically on $\sigma$. A nonzero tangent vector $W\in T_{[z]}Q^m$ is called
$\mathfrak{A}$-singular if it is contained in more than one $2$-flat. There are two
types of $\mathfrak{A}$-singular tangent vectors for $Q^m$, namely,
$\mathfrak{A}$-principal and $\mathfrak{A}$-isotropic tangent vectors, described as
follows (cf. \cite{BS,R}):
\begin{enumerate}[1.]
\item%[1.]
If there exists an almost product structure $A\in \mathfrak{A}_{[z]}$ such that $W\in V(A)$,
then $W$ is $\mathfrak{A}$-singular. Here, we denote $V(A)$ as
the eigenspace of $A$ corresponding to $1$. Such $W\in T_{[z]}Q^m$ is called an
$\mathfrak{A}$-principal tangent vector.

\item%[2.]
If there exists an almost product structure $A\in \mathfrak{A}_{[z]}$ and orthonormal
vectors $X,Y\in V(A)$ such that $\tfrac{W}{\|W\|}=\tfrac{X+JY}{\sqrt{2}}$, then $W$ is
$\mathfrak{A}$-singular. In this case we have $g(AW,W)=0$ and $W\in T_{[z]}Q^m$ is
called an $\mathfrak{A}$-isotropic tangent vector.
\end{enumerate}

For some further results about the complex quadric, we refer the readers to \cite{BS,BS-2022,K,R}.

\subsection{Real hypersurfaces of the complex quadric}\label{sect:2.2}~

Let $M$ be a real hypersurface of $Q^m$ and $N$ its unit normal vector field. For any
tangent vector field $X$ of $M$, we have the decomposition
\begin{equation}\label{eqn:2.4}
JX=\phi X+ \eta(X)N,
\end{equation}
where $\phi X$ and $\eta(X)N$ are, respectively, the tangent and normal parts of
$JX$. Then $\phi$ is a tensor field of type $(1,1)$, and $\eta$ is a $1$-form on $M$.
By definition, $\phi$ and $\eta$ satisfy the following relations:
\begin{equation}\label{eqn:2.5}
\left\{
\begin{aligned}
&\eta(X)=g(X,\xi),\ \ \eta(\phi X)=0,\ \ \phi^2X=-X+\eta(X)\xi,\\
&g(\phi X,Y)=-g(X,\phi Y),\ \ g(\phi X,\phi Y)=g(X,Y)-\eta(X)\eta(Y),
\end{aligned}\right.
\end{equation}
where $\xi:=-JN$ is called the {\it Reeb vector field} of $M$. The equations in
\eqref{eqn:2.5} show that $\{\phi,\xi,\eta\}$ determines an {\it almost contact
structure} on $M$.

Let $\nabla$ be the induced connection on $M$ with $R$ its Riemannian curvature
tensor. The formulas of Gauss and Weingarten state that
\begin{equation}\label{eqn:2.6}
\begin{split}
\bar \nabla_X Y=\nabla_X Y + g(SX,Y)N,\quad \bar \nabla_X N=-
S X , \ \ \forall\, X,Y \in TM,
\end{split}
\end{equation}
where $S$ is the shape operator of $M\hookrightarrow Q^m$. Let $H={\rm tr}S$
be the mean curvature of $M$.
By using \eqref{eqn:2.6},
we can derive that
\begin{equation}\label{eqn:2.7}
\nabla_X \xi=\phi SX.
\end{equation}

The Gauss and Codazzi equations of $M$ are given by
\begin{equation}\label{eqn:2.8}
\begin{split}
R(X,Y)Z=&g(Y,Z)X-g(X,Z)Y+g(\phi Y,Z)\phi X-g(\phi X,Z)\phi Y\\
 &-2g(\phi X,Y)\phi Z+g(AY,Z)(AX)^\top-g(AX,Z)(AY)^\top\\
 &+g(JAY,Z)(JAX)^\top-g(JAX,Z)(JAY)^\top\\
 &+g(SZ,Y)SX-g(SZ,X)SY,
 \end{split}
\end{equation}
and
\begin{equation}\label{eqn:2.9}
\begin{split}
(\nabla_X S)Y-(\nabla_Y S)X=&\eta(X)\phi Y-\eta(Y)\phi X-2g(\phi X,Y)\xi\\
&+g(X,AN)(AY)^\top-g(Y,AN)(AX)^\top\\
 &\qquad + g(X,A\xi)(JAY)^\top - g(Y,A\xi)(JAX)^\top,
\end{split}
\end{equation}
where $\cdot^\top$ means the tangential part.

%From \eqref{eqn:2.8}, contracting $Y$ and $Z$, we get the Ricci tensor of
%$M$: %(cf. (4.1) of \cite{Suh2}):
%
%\begin{equation}\label{eqn:2.10}
%\begin{split}
%\text{Ric}(X)=&(2m-1)X-3\eta(X)\xi-g(AN,N)(AX)^\top+g(AX,N)(AN)^\top\\
%&-g(A\xi,N)(JAX)^\top+g(A\xi,X)(A\xi)^\top+HSX-S^2X,
%\end{split}
%\end{equation}
%where $H={\rm tr}\,S$ denotes the mean curvature of the hypersurface $M$.

Notice that the tangent bundle $TM$ of $M$ splits orthogonally into
$TM=\mathcal{C}\oplus \mathcal{F}$, where $\mathcal{C}=\text{ker}(\eta)$
is the maximal complex subbundle of $TM$ and $\mathcal{F}=\mathbb{R}\xi$.
When restricted to $\mathcal{C}$, the structure tensor field $\phi$
coincides with the complex structure $J$. Moreover, at each point
$[z]\in M$, the set
$$
\mathcal{Q}_{[z]}=\{X\in T_{[z]}M\mid AX\in T_{[z]}M\ \text{for all}\ A\in \mathfrak{A}_{[z]}\}
$$
defines a maximal $\mathfrak{A}_{[z]}$-invariant subspace of $T_{[z]}M$.

The following definitions can be found in \cite{BS,BS-2022,Loo11}.

\begin{definition}[\cite{BS,BS-2022,Loo11}]\label{def:2.1}
Let $M$ be a real hypersurface of $Q^m$ ($m\geq3$) with unit normal vector field $N$.

{\rm (a)} If there exists an almost product structure $A\in \mathfrak{A}$
such that $AN=N$ everywhere, then we say that $N$ is $\mathfrak{A}$-principal on $M$,
and $M$ has $\mathfrak{A}$-principal unit normal vector field $N$.

{\rm (b)} If there exists an almost product structure $A\in\mathfrak{A}$
such that $AN,A\xi\in \mathcal{C}$ everywhere, then we say that
$N$ is $\mathfrak{A}$-isotropic on $M$, and $M$ has $\mathfrak{A}$-isotropic unit normal vector field $N$.
\end{definition}

\begin{remark}\label{rem:2.1}
Due that $\mathfrak{A}$ is an $S^1$-subbundle of the endomorphism
bundle End($TQ^m$), so if there exists an almost product structure $A_0\in\mathfrak{A}$
such that $A_0N,A_0\xi\in \mathcal{C}$ everywhere, then for any almost product structure $A\in\mathfrak{A}$, it also holds $AN,A\xi\in \mathcal{C}$ everywhere.
\end{remark}

%\begin{definition}\label{def:2.1}
%Let $M$ be a real hypersurface of the complex quadric $Q^m,\ m\ge3$. If the
%unit normal vector field $N$ of $M$ is $\mathfrak{A}$-principal everywhere,
%then $M$ is called $\mathfrak{A}$-principal. Analogously, if the unit normal
%vector field of $M$ is $\mathfrak{A}$-isotropic everywhere, then $M$ is
%called $\mathfrak{A}$-isotropic.
%\end{definition}

\vskip 1mm
In the following, we assume that $M$ is a Hopf hypersurface of $Q^m$ satisfying
$S\xi=\alpha \xi$, here the function $\alpha$ is called the Reeb
function. In later sections, we need the following results on Hopf hypersurfaces in $Q^m$.

\begin{lemma}[\cite{BS,BS-2022}]\label{lemma:2.3}
Let $M$ be a Hopf hypersurface in $Q^m$ with unit normal
vector field $N$ and $S\xi=\alpha\xi$.
Then the following equations hold for all vector fields $X,Y\in \mathcal{Q}$:
%
%\begin{equation}\label{eqn:2.12}
%X\alpha=(\xi\alpha)\eta(X)+2g(X,AN)g(A\xi,\xi),
%\end{equation}
%
\begin{equation}\label{eqn:2.13}
2g(S\phi SX,Y)-\alpha g((\phi S+S\phi)X,Y)-2g(\phi X,Y)=0.
\end{equation}
\end{lemma}

\begin{proof}
By using the fact $g(AN,X)=g(A\xi,X)=0$ for any $X\in \mathcal{Q}$, the proof of
\eqref{eqn:2.13} follows from those of both Lemma
5.1 in \cite{BS} and Lemma 4.2 in \cite{Suh1}.
\end{proof}

\begin{lemma}[\cite{Loo11,Suh1}]\label{lemma:2.5}
Let $M$ be a Hopf hypersurface in $Q^m$ ($m\geq3$). Then, the Reeb function $\alpha$ is constant if and only if $M$ has either $\mathfrak{A}$-principal unit normal
vector field $N$ or $\mathfrak{A}$-isotropic unit normal
vector field $N$.
\end{lemma}

\begin{theorem}[\cite{LS,Loo11}]\label{thm:2.1}
Let $M$ be a Hopf hypersurface in $Q^m$ ($m\geq3$). Then, $M$ has $\mathfrak{A}$-principal unit normal vector field $N$ if and only if $M$ is an open part of a tube over a totally geodesic $Q^{m-1}\hookrightarrow Q^{m}$.
\end{theorem}

\vskip1mm
To our knowledge, there are few known examples of Hopf hypersurfaces in $Q^m$ with $\mathfrak{A}$-isotropic
unit normal vector field $N$. In Example \ref{E6} of Section \ref{sect:3}, we construct a new family of Hopf hypersurfaces of $Q^m$ ($m\geq3$) with non-constant principal curvatures and $\mathfrak{A}$-isotropic unit normal vector field $N$.
The problem of classifying Hopf hypersurfaces in $Q^m$ with $\mathfrak{A}$-isotropic
unit normal vector field $N$ is difficult and still open.

\begin{lemma}[\cite{Loo11}]\label{lemma:2.6}
Let $M$ be a real hypersurface in $Q^m\ (m\ge3)$ with $\mathfrak{A}$-isotropic
unit normal vector field $N$. Then there exists an almost product structure $A\in\mathfrak{A}$
on $M$ such that $SAN=SA\xi=0$.
\end{lemma}

\begin{remark}\label{rem:2.2}
From the fact that $\mathfrak{A}$ is an $S^1$-subbundle of the endomorphism
bundle End($TQ^m$), then Lemma \ref{lemma:2.6} can also imply that, on a
real hypersurface $M$ of $Q^m\ (m\ge3)$ with $\mathfrak{A}$-isotropic unit normal vector field $N$,
it holds that $SAN=SA\xi=0$ for any almost product structure $A\in\mathfrak{A}$.
\end{remark}

\begin{theorem}[\cite{BS}]\label{thm:2.2}
Let $M$ be a real hypersurface of $Q^m$ ($m\geq3$) with isometric Reeb flow.
Then, $M$ is an open part of a tube over a totally geodesic $\mathbb{C}P^k\hookrightarrow Q^{2k}$ ($m=2k$, $k\geq2$).
\end{theorem}

%==================================================
\section{Canonical examples of Hopf hypersurfaces}\label{sect:3}

In this section, we describe three classes of the homogeneous real hypersurfaces in $Q^m$. We also construct a new family of Hopf hypersurfaces of $Q^m$ ($m\geq3$) with non-constant principal curvatures and $\mathfrak{A}$-isotropic unit normal vector field $N$.

First of all, we recall the homogeneous real hypersurfaces of types (A) and (B) in Theorem 6.3.1 of
\cite{BS-2022}. %These two families of examples were first introduced in \cite{BS-2012}. %There are many articles that give the characterization of these two examples.

\begin{example}\label{E1}
The map
$$
Q^{m-1}\rightarrow Q^m\hookrightarrow \mathbb{C}P^{m+1}, \
[(z_1, . . . , z_{m+1})]\mapsto[(z_1,...,z_{m+1},0)]
$$
provides an embedding of $Q^{m-1}$ into $Q^m$ as a totally geodesic complex hypersurface.

\begin{proposition}[\cite{BS-2012}]\label{prop:E1P}
Let $M$ be the tube of radius $0<t<\frac{\pi}{2\sqrt{2}}$ around the totally
geodesic $Q^{m-1}$ in $Q^m$. Then the following statements hold:
\begin{enumerate}
\item[(i)]
$M$ is a Hopf hypersurface, its Reeb function is $\alpha=\sqrt{2}\cot(\sqrt{2}t)$.

\item[(ii)]
$M$ has $\mathfrak{A}$-principal unit normal vector field $N$.

\item[(iii)]
$M$ has three distinct constant principal curvatures. Their values and multiplicities
are given as follows
$$
\begin{tabular}{|c|c|c|c|}
  \hline
  {\rm value} & $\sqrt{2}\cot(\sqrt{2}t)$ & $0$ & $-\sqrt{2}\tan(\sqrt{2}t)$ \\
  \hline
  {\rm multiplicity} & $1$ & $m-1$ & $m-1$ \\
  \hline
\end{tabular}
$$
It holds that $\mathcal{Q}=V_{0}\oplus V_{-\sqrt{2}\tan(\sqrt{2}t)}$ and $JV_{0}=V_{-\sqrt{2}\tan(\sqrt{2}t)}$.
Here $V_{0}$ and $V_{-\sqrt{2}\tan(\sqrt{2}t)}$ are the eigenspaces corresponding to the principal curvatures $0$ and $-\sqrt{2}\tan(\sqrt{2}t)$ restricted to $\mathcal{Q}$, respectively.
There is an almost product structure $A\in\mathfrak{A}$ such that
$AV_{0}=-V_{0}$ and $AV_{-\sqrt{2}\tan(\sqrt{2}t)}=V_{-\sqrt{2}\tan(\sqrt{2}t)}$.
\end{enumerate}
\end{proposition}

%We denote the tubes around $Q^{m-1}\hookrightarrow Q^{m}$ ($m\geq2$) by $\Phi_r(Q^{m-1})$, %$0<r<\frac{\pi}{2\sqrt{2}}$.
%Let
%$$
%\begin{aligned}
%{\text [z]}=[(z_1,z_2,z_3,...,z_{m+2})]=[(u_1,...,u_{m+2})+\mathbf{i}(v_1,...,v_{m+2})]\in Q^{m}.
%\end{aligned}
%$$
%Then hypersurfaces $\Phi_r(Q^{m-1})$ are given by
%\begin{equation}\label{eqn:3.2aaa}
%\Phi_r(Q^{m-1}):=\{[z]\in Q^{m}|\ |z_{m+2}|^2=u_{m+2}^2+v_{m+2}^2=\frac{1}{2}\sin^2 (\sqrt{2}r)\}.
%\end{equation}
%Moreover, $\Phi_r(Q^{m-1})$ has $\mathfrak{A}$-principal unit normal vector field.
%Its principal curvatures are:

%$$
%\begin{tabular}{|c|c|c|c|}
%  \hline
%  % after \\: \hline or \cline{col1-col2} \cline{col3-col4} ...
%  {\rm value} & $-\sqrt{2}\cot(\sqrt{2}r)$ & $0$ & $\sqrt{2}\tan(\sqrt{2}r)$ \\
%  \hline
%  {\rm multiplicity} & $1$ & $m-1$ & $m-1$ \\
%  \hline
%\end{tabular}
%$$

%For any $r\in(0,\frac{\pi}{2\sqrt{2}})$, hypersurface $\Phi_r(Q^{m-1})$ has
%three distinct constant principal curvatures. %They are contact hypersurfaces.
\end{example}

\begin{example}\label{E2}
We assume that $m$ is even, say $m=2k$. The map
$$
\mathbb{C}P^k\rightarrow Q^{2k}\hookrightarrow \mathbb{C}P^{2k+1}, \
[(z_1, . . . , z_{k+1})]\mapsto[(z_1,...,z_{k+1},\mathbf{i}z_1,...,\mathbf{i}z_{k+1})]
$$
provides an embedding of $\mathbb{C}P^k$ into $Q^{2k}$ as a totally geodesic complex submanifold.

\begin{proposition}[\cite{BS}]\label{prop:E2P}
Let $M$ be the tube of radius $0<t<\frac{\pi}{2}$ around the totally
geodesic $\mathbb{C}P^k$ in $Q^{2k}$, $k\geq2$. Then the following statements hold:
\begin{enumerate}
\item[(i)]
$M$ is a Hopf hypersurface, its Reeb function is $\alpha=2\cot(2t)$.

\item[(ii)]
$M$ has $\mathfrak{A}$-isotropic unit normal vector field $N$.

\item[(iii)]
$M$ has three or four distinct constant principal curvatures. Their values and multiplicities
are given as follows
$$
\begin{tabular}{|c|c|c|c|c|}
  \hline
  {\rm value} & $2\cot(2t)$ & $0$ & $-\tan(t)$ & $\cot(t)$ \\
  \hline
  {\rm multiplicity} & $1$ & $2$ & $2k-2$ & $2k-2$\\
  \hline
\end{tabular}
$$
It holds that $\mathcal{Q}=V_{-\tan(t)}\oplus V_{\cot(t)}$, $JV_{-\tan(t)}=V_{-\tan(t)}$ and
$JV_{\cot(t)}=V_{\cot(t)}$. Here $V_{-\tan(t)}$ and $V_{\cot(t)}$ are the eigenspaces corresponding to the principal curvatures $-\tan(t)$ and $\cot(t)$ restricted to $\mathcal{Q}$, respectively.
Any almost product structure $A\in\mathfrak{A}$ maps
$V_{-\tan(t)}$ into $V_{\cot(t)}$. When $t=\frac{\pi}{4}$, $M$ has
three distinct constant principal curvatures $0,1,-1$.
When $t\in(0,\frac{\pi}{4})\cup(\frac{\pi}{4},\frac{\pi}{2})$, $M$ has
four distinct constant principal curvatures.

\item[(iv)] $M$ has isometric Reeb flow, i.e., it satisfies $S\phi=\phi S$.
\end{enumerate}
\end{proposition}

%We denote the tubes around $\mathbb{C}P^k\hookrightarrow Q^{2k}$ ($k\geq2$) by $\Phi_r(\mathbb{C}P^k)$, $0<r<\frac{\pi}{2}$.
%Let
%$$
%\begin{aligned}
%{\text [z]}&=[(z_1,...,z_{k+1},z_{k+2},...,z_{2k+2})]\\
%&=[(u_1,...,u_{k+1},v_1,...,v_{k+1})+\mathbf{i}(x_1,...,x_{k+1},y_1,...,y_{k+1})]\in Q^{2k}.
%\end{aligned}
%$$
%Then hypersurfaces $\Phi_r(\mathbb{C}P^k)$ are given by
%\begin{equation}\label{eqn:3.1aaa}
%\Phi_r(\mathbb{C}P^k):=\{[z]\in Q^{2k}|\ -{\rm Im}\Big(\sum_{j=1}^{k+1}(z_j\bar{z}_{k+1+j})\Big)
%=\sum_{j=1}^{k+1}(u_jy_j-v_jx_j)=\frac{1}{2}\cos (2r)\}.
%\end{equation}
%Moreover, $\Phi_r(\mathbb{C}P^k)$ has $\mathfrak{A}$-isotropic unit normal vector field.
%Its principal curvatures are:
%$$
%\begin{tabular}{|c|c|c|c|c|}
%  \hline
%  {\rm value} & $-2\cot(2r)$ & 0 & $\tan(r)$ & $-\cot(r)$ \\
%  \hline
%  {\rm multiplicity} & $1$ & $2$ & $2k-2$ & $2k-2$ \\
%  \hline
%\end{tabular}
%$$

%When $r=\frac{\pi}{4}$, hypersurface $\Phi_{\frac{\pi}{4}}(\mathbb{C}P^k)$ has
%three distinct constant principal curvatures $0,1,-1$.
%When $r\in(0,\frac{\pi}{4})\cup(\frac{\pi}{4},\frac{\pi}{2})$, hypersurface %$\Phi_{\frac{\pi}{4}}(\mathbb{C}P^k)$ has
%four distinct constant principal curvatures.
\end{example}

%\begin{remark}\label{rem:3.1}
%Examples \ref{E1} and \ref{E2} were first introduced in \cite{BS-2012}, but to our knowledge, the %expressions \eqref{eqn:3.2aaa} and \eqref{eqn:3.1aaa} has not appeared
%previously in the literature.
%\end{remark}

%\begin{example}\label{E3}
%For $m=4k-2$, $k\geq 2$, a tube around the equivariant embedding of the $4k$-dimensional
%homogeneous space $Sp_kSp_1/U_1Sp_{k-1}SO_2$ in $Q^{4k-2}$ (or equivalently, of
%the $(8k-7)$-dimensional homogeneous space $Sp_kSp_1/SO_2Sp_{k-2}Sp_1$ in $Q^{4k-2}$).
%\end{example}
%The

Next, we describe the homogeneous real hypersurfaces of type (C) in Theorem 6.3.1 of
\cite{BS-2022} in detail.

\begin{example}\label{E3}
%For $m=4k-2$, $k\geq 2$, a tube around the equivariant embedding of the $4k$-dimensional
%homogeneous space $Sp_kSp_1/U_1Sp_{k-1}SO_2$ in $Q^{4k-2}$ (or equivalently, of
%the $(8k-7)$-dimensional homogeneous space $Sp_kSp_1/SO_2Sp_{k-2}Sp_1$ in $Q^{4k-2}$).
Consider the natural isomorphism $\mathbb{R}^{4k}\cong \mathbb{H}^k$ and the standard action of
$H:=Sp_kSp_1$ on $\mathbb{H}^k$. Specifically, for any $z=(z_1,z_2,...,z_k)\in \mathbb{H}^k$,
$g_1\in Sp_k$ and $g_2\in Sp_1$, the action of $Sp_kSp_1$ on $\mathbb{H}^k$ is given by
$(g_1,g_2)\cdot z=g_1z^\top g_2$. Let $e\in SO_{4k}$ be the identity
$(4k\times 4k)$-matrix. We put $G=SO_{4k}$, $K=SO_2SO_{4k-2}$ and denote by
$o\in Q^{4k-2}$ the ``base point'' $eK$ of the homogeneous space $G/K$.
Then $K$ is the isotropy group of $G$ at $o$.
Let $\mathfrak{g}=\mathfrak{so}_{4k}$ be the Lie algebra of $G$,
$\mathfrak{k}=\mathfrak{so}_2\oplus\mathfrak{so}_{4k-2}$ be the Lie algebra of the isotropy group
$K$. Then we have the
Cartan decomposition $\mathfrak{g}=\mathfrak{p}\oplus \mathfrak{k}$ of $\mathfrak{g}$.
We identify the tangent space $T_oQ^{4k-2}$ of $Q^{4k-2}$ at $o$
with $\mathfrak{p}$ in the usual way.
For more details about the construction of the complex quadric as a Riemannian symmetric space,
we refer to \cite{K}.
The tangent space $T_eSp_kSp_1$ in $\mathfrak{so}_{4k}$ is spanned by the following basis
$$
\left(
  \begin{array}{cccc}
    0_1 & -B & 0_1 & 0_1 \\
    B & 0_1 & 0_1 & 0_1 \\
    0_1 & 0_1 & 0_1 & -B \\
    0_1 & 0_1 & B & 0_1 \\
  \end{array}
\right),
\left(
  \begin{array}{cccc}
    0_1 & 0_1 & -C & 0_1 \\
    0_1 & 0_1 & 0_1 & C \\
    C & 0_1 & 0_1 & 0_1 \\
    0_1 & -C & 0_1 & 0_1 \\
  \end{array}
\right),
\left(
  \begin{array}{cccc}
    0_1 & 0_1 & 0_1 & -D \\
    0_1 & 0_1 & -D & 0_1 \\
    0_1 & D & 0_1 & 0_1 \\
    D & 0_1 & 0_1 & 0_1 \\
  \end{array}
\right),
\left(
  \begin{array}{cccc}
    F & 0_1 & 0_1 & 0_1 \\
    0_1 & F & 0_1 & 0_1 \\
    0_1 & 0_1 & F & 0_1 \\
    0_1 & 0_1 & 0_1 & F \\
  \end{array}
\right),
$$
$$
\left(
  \begin{array}{cccc}
    0_1 & -I_{k,k} & 0_1 & 0_1 \\
    I_{k,k} & 0_1 & 0_1 & 0_1 \\
    0_1 & 0_1 & 0_1 & I_{k,k} \\
    0_1 & 0_1 & -I_{k,k} & 0_1 \\
  \end{array}
\right),
\left(
  \begin{array}{cccc}
    0_1 & 0_1 & -I_{k,k} & 0_1 \\
    0_1 & 0_1 & 0_1 & -I_{k,k} \\
    I_{k,k} & 0 & 0_1 & 0_1 \\
    0_1 & I_{k,k} & 0_1 & 0_1 \\
  \end{array}
\right),
\left(
  \begin{array}{cccc}
    0_1 & 0_1 & 0_1 & -I_{k,k} \\
    0_1 & 0_1 & I_{k,k} & 0_1 \\
    0_1 & -I_{k,k} & 0_1 & 0_1 \\
    I_{k,k} & 0_1 & 0_1 & 0_1 \\
  \end{array}
\right),
$$
where $B,C,D$ are the symmetric $(k\times k)$-matrixes, $F$ is the antisymmetric $(k\times k)$-matrix,
$I_{k,k}$ is the identity $(k\times k)$-matrix, $0_1$ is the null $(k\times k)$-matrix. %$0$ is the null $(k\times k)$-matrix.
Then it can be checked that ${\rm dim}((T_eSp_kSp_1)|_{\mathfrak{p}})=8k-7$.
For any $X,Y\in \mathfrak{p}$, let $\langle X,Y\rangle_{\mathfrak{p}}=-\frac{1}{4}{\rm tr}(XY)$.
The orthogonal complement of $(T_eSp_kSp_1)|_{\mathfrak{p}}$ in $\mathfrak{p}$ with respect to $\langle\cdot,\cdot\rangle_{\mathfrak{p}}$ is spanned by
$$
\eta_1=\left(
  \begin{array}{cccc}
    0_1 & \tilde{J} & 0_1 & 0_1 \\
    \tilde{J} & 0_1 & 0_1 & 0_1 \\
    0_1 & 0_1 & 0_1 & 0_1 \\
    0_1 & 0_1 & 0_1 & 0_1 \\
  \end{array}
\right),
\eta_2=\left(
  \begin{array}{cccc}
    0_1 & 0_1 & \tilde{J} & 0_1 \\
    0_1 & 0_1 & 0_1 & 0_1 \\
    \tilde{J} & 0_1 & 0_1 & 0_1 \\
    0_1 & 0_1 & 0_1 & 0_1 \\
  \end{array}
\right),
\eta_3=\left(
  \begin{array}{cccc}
    0_1 & 0_1 & 0_1 & \tilde{J} \\
    0_1 & 0_1 & 0_1 & 0_1 \\
    0_1 & 0_1 & 0_1 & 0_1 \\
    \tilde{J} & 0_1 & 0_1 & 0_1 \\
  \end{array}
\right),
$$
where
$$
\tilde{J}=
\left(
\begin{array}{cc}
\bar{J} & 0_2 \\
0_3 & 0_4 \\
\end{array}
\right), \ \
\bar{J}=\left(
\begin{array}{cc}
0 & -1 \\
1 & 0 \\
\end{array}
\right),
$$
where $0_2$ is the null $(2\times (k-2))$-matrix, $0_3$ is the null $((k-2)\times 2)$-matrix,
$0_4$ is the null $((k-2)\times (k-2))$-matrix.

Choose a unit vector $N_o=-\eta_1\in \mathfrak{so}_{4k}$, then $g_t={\rm Exp}(tN_o)$ is a
one parameter subgroup of $SO_{4k}$, which is briefly given by
$$
g_t=\left(
  \begin{array}{cccc}
    \cos(t)I_{2,2} & 0_{2,k-2} & \sin(t)\bar{J} & 0_{2,3k-2} \\
    0_{k-2,2} & I_{k-2,k-2} & 0_{k-2,2} & 0_{k-2,3k-2} \\
    \sin(t)\bar{J} & 0_{2,k-2} & \cos(t)I_{2,2} & 0_{2,3k-2} \\
    0_{3k-2,2} & 0_{3k-2,k-2} & 0_{3k-2,2} & I_{3k-2,3k-2} \\
  \end{array}
\right),
$$
where $I_{2,2}$, $I_{k-2,k-2}$ and $I_{3k-2,3k-2}$ are the identity $(2\times 2)$-matrix,
$((k-2)\times (k-2))$-matrix and $((3k-2)\times (3k-2))$-matrix, respectively.
So $\gamma(t)=g_t\cdot o$ is
a geodesic normal to $H\cdot o$. %where $0\leq t\leq\frac{\pi}{4}$.
The orbit of $Hg_t^{-1}$ through $o$ is
$Hg_t^{-1}\cdot o$ and we have a congruence $(Hg_t^{-1})\cdot o\cong (g_tHg_t^{-1})\cdot o$.
%Therefore, we have only to study the orbit of $gHg^{-1}$ through the point $o$, $0< t<\frac{\pi}{4}$.
Let $M_t:=g_tHg_t^{-1}\cdot o$, then $M_t$ is a homogeneous real hypersurface.
In fact, the Lie algebra $\mathfrak{h}'$ of $H':=g_tHg_t^{-1}$ in $\mathfrak{so}_{4k}$ is given by $\mathfrak{h}'=g_t(T_eSp_kSp_1)g_t^{-1}$. By direct calculation, we have
${\rm dim}T_o(M_t)=8k-5$ for $0< t<\frac{\pi}{4}$, which means that the orbit $M_t$ is a homogeneous real hypersurface
of $Q^{4k-2}$ for $0< t<\frac{\pi}{4}$, the unit normal vector of $M_t$ at $o$ is still $N_o$.
Here, we always identify the tangent space $T_oQ^{4k-2}$
with $\mathfrak{p}$.
In the following, we just look at the geometry of $M_t$ at point $o$.

We recall the general theory,
for any $X\in \mathfrak{h}'$, we denote by $X_o^*=\frac{d}{ds}({\rm Exp}(sX)\cdot o)|_{s=0}=\pi(X)$, where
$\pi$ is the projection from $\mathfrak{g}$ to $\mathfrak{p}$.
Let $\mathfrak{p}'=\pi(\mathfrak{h}')$. Then $\mathfrak{p}'$
is the tangent space $T_oM_t$.
From $N_o$ is a normal vector
of $M_t$ at $o$, then the second fundamental form $h:\mathfrak{p}'\times \mathfrak{p}'\rightarrow \mathbb{R}N_o$ of $M_t$ at $o$ is given by
\begin{equation}\label{eqn:hhh}
2\langle h(X_o^*,Y_o^*),N_o\rangle_{\mathfrak{p}}=\langle [N_o,X]_o^*,Y_o^*\rangle_{\mathfrak{p}}+\langle X_o^*,[N_o,Y]_o^*\rangle_{\mathfrak{p}}.
\end{equation}
%where $\langle X,Y\rangle=-\frac{1}{4}{\rm tr}(XY)$ for any $X,Y\in \mathfrak{p}$.
%So we need to calculate $[N,X]$.

%Let $B_1,C_1,D_1$ be $2\times2$ symmetric matrixes, $F_1$ be $2\times2$ antisymmetric matrix, $B_2,C_2,D_2,F_2$ be $2\times(k-2)$ matrixes.
We identify
\begin{equation}\label{eqn:pid}
\mathfrak{p}\cong M_{2\times(k-2)}\oplus M_{2\times2}\oplus M_{2\times(k-2)}\oplus M_{2\times2}\oplus M_{2\times(k-2)}\oplus M_{2\times2}\oplus M_{2\times(k-2)},
\end{equation}
where $M_{2\times(k-2)}$ is the $(2\times(k-2))$-matrixs, $M_{2\times2}$ is the $(2\times2)$-matrixs.
Let
$$
\bar{I}=\left(
    \begin{array}{cc}
      1 & 0 \\
      0 & 1 \\
    \end{array}
  \right),\ \
\bar{H}=\left(
    \begin{array}{cc}
      0 & 1 \\
      1 & 0 \\
    \end{array}
  \right),\ \
\bar{K}=\left(
    \begin{array}{cc}
      1 & 0 \\
      0 & -1 \\
    \end{array}
  \right).
$$
Then, at point $o$, the normal vector $N_o$ is written as $N_o=(0_2,-\bar{J},0_2,0_5,0_2,0_5,0_2)$, and the Reeb vector field $\xi$ at $o$
is $\xi_o=(0_2,-\bar{I},0_2,0_5,0_2,0_5,0_2)$, where $0_5$ is the null $(2\times 2)$-matrix. %Here, for brevity, by the identification \eqref{eqn:pid}, we use $0$ to stand for the null $2\times(k-2)$ matrix or the null $2\times2$ matrix at the corresponding position.
Now, according to the formula \eqref{eqn:hhh},
by direct calculation, we can have
$$
S\xi_o=2\tan(2t)\xi_o,
$$
where $S$ is the shape operator of $M_t$.
It implies that $M_t$ is a Hopf hypersurface, and the Reeb function $\alpha$ is $2\tan(2t)$.
Let $A_{-d\pi(\bar{z})}$ be the almost product structure of $Q^{4k-2}\hookrightarrow \mathbb{C}P^{4k-1}$ with respect to the unit normal vector field $-d\pi(\bar{z})$, where $z\in \pi^{-1}(M_t)$.
It follows that $A_{-d\pi(\bar{z})}N_o=(0_2,\bar{H},0_2,0_5,0_2,0_5,0_2)$, $A_{-d\pi(\bar{z})}\xi_o=(0_2,-\bar{K},0_2,0_5,0_2,0_5,0_2)$.
Thus we have $g(A_{-d\pi(\bar{z})}N_o,N_o)=0$ and $g(A_{-d\pi(\bar{z})}N_o,JN_o)=0$. It means that $M_t$ has $\mathfrak{A}$-isotropic unit normal vector field.

At point $o$, through direct calculation, it follows from \eqref{eqn:hhh} that $SA_{-d\pi(\bar{z})}N_o=SA_{-d\pi(\bar{z})}\xi_o=0$.
Moreover, the principal curvatures of $M_t$ restricted to $\mathcal{Q}$ are
$$
\lambda_1=\tan(t), \ \lambda_2=-\cot(t),\ \lambda_3=\frac{\cos(t)+\sin(t)}{\cos(t)-\sin(t)},\ \lambda_4=-\frac{\cos(t)-\sin(t)}{\cos(t)+\sin(t)}.
$$
The corresponding eigenspaces are given by
$$
\begin{aligned}
V_{\lambda_1}={\rm Span}\{&(0_2,0_5,0_2,\bar{I},0_2,0_5,0_2),(0_2,0_5,0_2,0_5,0_2,\bar{I},0_2)\},\\
V_{\lambda_2}={\rm Span}\{&(0_2,0_5,0_2,\bar{J},0_2,0_5,0_2),(0_2,0_5,0_2,0_5,0_2,\bar{J},0_2)\},\\
V_{\lambda_3}={\rm Span}\{&(0_2,0_5,0_2,\bar{H},0_2,-\bar{K},0_2),(0_2,0_5,0_2,\bar{K},0_2,\bar{H},0_2),\\
&(\bar{J}B_1,0_5,B_1,0_5,0_2,0_5,0_2),(0_2,0_5,0_2,0_5,C_1,0_5,\bar{J}C_1)\},\\
V_{\lambda_4}={\rm Span}\{&(0_2,0_5,0_2,\bar{H},0_2,\bar{K},0_2),(0_2,0_5,0_2,\bar{K},0_2,-\bar{H},0_2),\\
&(\bar{J}B_1,0_5,-B_1,0_5,0_2,0_5,0_2),(0_2,0_5,0_2,0_5,-C_1,0_5,\bar{J}C_1)\},
%&V_{\lambda_5}={\rm Span}\{(0_2,\bar{H},0_2,0_5,0_2,0_5,0_2),(0_2,-\bar{K},0_2,0_5,0_2,0_5,0_2)\},\ \
%V_{\lambda_6}={\rm Span}\{(0_2,-\bar{I},0_2,0_5,0_2,0_5,0_2)\},
\end{aligned}
$$
where $B_1,C_1$ are $(2\times (k-2))$-matrixes.
Thus, the dimensions of eigenspaces $V_{\lambda_1},V_{\lambda_2},V_{\lambda_3}$ and $V_{\lambda_4}$ are
$$
{\rm dim V_{\lambda_1}}={\rm dim V_{\lambda_2}}=2,\
{\rm dim V_{\lambda_3}}={\rm dim V_{\lambda_4}}=4k-6.
$$
It can also be checked that
$$
JV_{\lambda_1}=V_{\lambda_2},\ \ JV_{\lambda_3}=V_{\lambda_3},\ \ JV_{\lambda_4}=V_{\lambda_4},\ \
A_{-d\pi(\bar{z})}(V_{\lambda_1}\oplus V_{\lambda_2})\subset(V_{\lambda_3}\oplus V_{\lambda_4}).
$$
Furthermore, the principal curvatures satisfy the following relationship:
$$
\lambda_1=\frac{-2+\sqrt{\alpha^2+4}}{\alpha},\ \lambda_2=\frac{-2-\sqrt{\alpha^2+4}}{\alpha},\
\lambda_3=\frac{\alpha+\sqrt{\alpha^2+4}}{2},\ \lambda_4=\frac{\alpha-\sqrt{\alpha^2+4}}{2}.
$$
\end{example}

We summarize the previous discussion in the following proposition.

\begin{proposition}\label{prop:E3P}
Let $M$ be the tube of radius $0<t<\frac{\pi}{4}$ around the equivariant embedding of the $(8k-7)$-dimensional homogeneous space $Sp_kSp_1/SO_2Sp_{k-2}Sp_1$ in $Q^{4k-2}$, $k\geq 2$.
Then the following statements hold:
\begin{enumerate}
\item[(i)]
$M$ is a Hopf hypersurface, its Reeb function is $\alpha=2\tan(2t)$.

\item[(ii)]
$M$ has $\mathfrak{A}$-isotropic unit normal vector field $N$.

\item[(iii)]
$M$ has six distinct constant principal curvatures. Their values and multiplicities
are given as follows
$$
\begin{tabular}{|c|c|c|c|c|c|c|}
  \hline
  {\rm value} & $2\tan(2t)$ & $0$ & $\tan(t)$ & $-\cot(t)$ & $\frac{\cos(t)+\sin(t)}{\cos(t)-\sin(t)}$ & $-\frac{\cos(t)-\sin(t)}{\cos(t)+\sin(t)}$\\
  \hline
  {\rm multiplicity} & $1$ & $2$ & $2$ & $2$ & $4k-6$ & $4k-6$\\
  \hline
\end{tabular}
$$
It holds $\mathcal{Q}=V_{\lambda_1}\oplus V_{\lambda_2}\oplus V_{\lambda_3}\oplus V_{\lambda_4}$. The almost product structure $A\in\mathfrak{A}$ maps
$V_{\lambda_1}\oplus V_{\lambda_2}$ into $V_{\lambda_3}\oplus V_{\lambda_4}$.
\end{enumerate}
\end{proposition}

\begin{remark}\label{rem:3.1as}
Examples \ref{E1} and \ref{E2} were first introduced in \cite{BS-2012}, and they have been studied extensively. But to our knowledge, the properties about Example \ref{E3} have not appeared previously in the literature.
\end{remark}

Finally, we construct a family of Hopf hypersurfaces of $Q^m$ with non-constant principal curvatures.

\begin{example}\label{E6}
For any $2\leq k\leq m$, we define the $(2m-2)$-dimensional complex submanifold $\tilde{N}^{2m-2}_k$ of $Q^m$ ($m\geq3$) as follows:
$$
\tilde{N}^{2m-2}_k=\{[z]=[(z_1,z_2,...,z_{m+2})]\in Q^{m}|\sum_{i=1}^k z_i^2=0\}.
$$
If we write
$$
\begin{aligned}
{\text [z]}&=[(z_1,z_2,...,z_{m+2})]=[(u_1,u_2,...,u_{m+2})+\mathbf{i}(v_1,v_2,...,v_{m+2})]\in Q^{m}.
\end{aligned}
$$
Then, for any $[z]\in \tilde{N}^{2m-2}_k$, it holds
$$
\sum_{i=1}^ku_i^2=\sum_{i=1}^kv_i^2,\ \ \sum_{i=1}^ku_iv_i=0,
\ \ \sum_{i=k+1}^{m+2}u_i^2=\sum_{i=k+1}^{m+2}v_i^2,\ \ \sum_{i=k+1}^{m+2}u_iv_i=0,
\ \ \sum_{i=1}^{m+2}u_i^2=\sum_{i=1}^{m+2}v_i^2=\frac{1}{2}.
$$
$\tilde{N}^{2m-2}_k$ is a complex submanifold of $Q^m$ with
two subsets of singularities:
$\{[z]=[(z_1,z_2,...,z_{m+2})]\in \tilde{N}^{2m-2}_k|\ z_{k+1}=\cdots=z_{m+2}=0\}$ and
$\{[z]=[(z_1,z_2,...,z_{m+2})]\in \tilde{N}^{2m-2}_k|\ z_{1}=\cdots=z_{k}=0\}$.
In the following, we always discuss the complex submanifold $\tilde{N}^{2m-2}_k$ outside of these singularities.
At any point $[z]=[(z_1,z_2,...,z_{m+2})]\in \tilde{N}^{2m-2}_k$, we define that $B([z]):=\sum_{i=1}^k|z_i|^2$,
which does not depend on the chosen representative $z$. For any $z=(z_1,z_2,...,z_{m+2})\in\pi^{-1}(\tilde{N}^{2m-2}_k)$, let
$$
V_1=(\overline{z_1},\overline{z_2},...,\overline{z_k},0,...,0),
\ V_2=(\mathbf{i}\overline{z_1},\mathbf{i}\overline{z_2},...,\mathbf{i}\overline{z_k},0,...,0)
$$
be two vector fields of $\mathbb{S}^{2m+3}$ along $\pi^{-1}(\tilde{N}^{2m-2}_k)$.
Then, $B$ is a well defined function on $\tilde{N}^{2m-2}_k$ outside of the singularities and $0<B<1$, and the normal space $T_{[z]}^\bot \tilde{N}^{2m-2}_k$ in $Q^m$ is given by
$$
T_{[z]}^\bot \tilde{N}^{2m-2}_k={\rm Span}\{{\eta_1}_{[z]}=(d\pi)_z(\frac{V_1-B\bar{z}}{\sqrt{B(1-B)}}),
\ {\eta_2}_{[z]}=(d\pi)_z(\frac{V_2-B\mathbf{i}\bar{z}}{\sqrt{B(1-B)}})\},
$$
where $g({\eta_1}_{[z]},{\eta_1}_{[z]})=g({\eta_2}_{[z]},{\eta_2}_{[z]})=1$ and ${\eta_2}_{[z]}=J{\eta_1}_{[z]}$. Let $A_{-d\pi(\bar{z})}$ be the almost product structure of $Q^{m}\hookrightarrow \mathbb{C}P^{m+1}$ with respect to the unit normal vector field $-d\pi(\bar{z})$ along $\tilde{N}^{2m-2}_k$, where $z\in \pi^{-1}(\tilde{N}^{2m-2}_k)$.
Then it can be checked that $A_{-d\pi(\bar{z})}\eta_1,A_{-d\pi(\bar{z})}\eta_2\in T\tilde{N}^{2m-2}_k$. Thus the normal bundle of $\tilde{N}^{2m-2}_k$ consists of $\mathfrak{A}$-isotropic singular tangent vectors of $Q^m$. By $B([z])=\sum_{i=1}^k|z_i|^2$, we have
$A_{-d\pi(\bar{z})}\eta_1B=2\sqrt{B(1-B)}$ and $A_{-d\pi(\bar{z})}\eta_2B=0$.
Now, by direct calculation, we can have
$$
S_{\eta_1}(A_{-d\pi(\bar{z})}\eta_1)=S_{\eta_1}(A_{-d\pi(\bar{z})}\eta_2)
=S_{\eta_2}(A_{-d\pi(\bar{z})}\eta_1)=S_{\eta_2}(A_{-d\pi(\bar{z})}\eta_2)=0,
$$
where $S_{\eta_i}$ is the shape operator of $\tilde{N}^{2m-2}_k$ with respect to $\eta_i$,
$1\leq i\leq2$. By the fact that $\mathfrak{A}$ is an $S^1$-subbundle of the endomorphism
bundle End($TQ^m$), we have $S_{\eta_1}(A\eta_1)=S_{\eta_1}(A\eta_2)
=S_{\eta_2}(A\eta_1)=S_{\eta_2}(A\eta_2)=0$
for any almost product structure $A\in \mathfrak{A}$.
It deduces that, for any unit normal vector field $\tilde{\eta}\in T^\bot \tilde{N}^{2m-2}_k$,
it holds $S_{\tilde{\eta}}A\eta_1=S_{\tilde{\eta}}A\eta_2=0$, where $S_{\tilde{\eta}}$ is the shape operator of $\tilde{N}^{2m-2}_k$ with respect to $\tilde{\eta}$.
For $z=(z_1,z_2,...,z_{m+2})\in\pi^{-1}(\tilde{N}^{2m-2}_k)$, let
$$
\begin{aligned}
&D_1=\{(d\pi)_z(x_1,...,x_k,0,...,0)\in T_{[z]}\tilde{N}^{2m-2}_k|(x_1,...,x_k)\in \mathbb{R}^k,\
\sum_{i=1}^{k}x_iu_i=\sum_{i=1}^{k}x_iv_i=0\},\\
&D_2=\{(d\pi)_z(0,...,0,x_{k+1},...,x_{m+2})\in T_{[z]}\tilde{N}^{2m-2}_k|(x_{k+1},...,x_{m+2})\in \mathbb{R}^{m+2-k},\  \sum_{i=k+1}^{m+2}x_iu_i=\sum_{i=k+1}^{m+2}x_iv_i=0\},\\
&D_3=JD_1,\   D_4=JD_2,\  D_5={\rm Span}\{A{\eta_1}_{[z]},A{\eta_2}_{[z]}\}.
\end{aligned}
$$
Then $T_{[z]}\tilde{N}^{2m-2}_k=\oplus_{i=1}^{5} D_i$.
Now, for unit normal vector field $\eta_1$, the principal curvatures of the shape operator $S_{\eta_1}$
on $\tilde{N}^{2m-2}_k$ are
$$
0,\ -\sqrt{\frac{1-B}{B}},\ \sqrt{\frac{1-B}{B}},\ \sqrt{\frac{B}{1-B}},\ -\sqrt{\frac{B}{1-B}},
$$
the corresponding eigenspaces are $D_5$, $D_1$, $D_3$,
$D_2$ and $D_4$, respectively.
For unit normal vector field $\eta_2$, it follows from $S_{\eta_2}=JS_{\eta_1}$ that the principal curvatures of the shape operator $S_{\eta_2}$ on $\tilde{N}^{2m-2}_k$ are also
$$
0,\ -\sqrt{\frac{1-B}{B}},\ \sqrt{\frac{1-B}{B}},\ \sqrt{\frac{B}{1-B}},\ -\sqrt{\frac{B}{1-B}}.
$$
%{\color{red}It shows that $\tilde{N}^{2m-2}_k$ is a austere submanifold.}
Thus, for any unit normal vector field $\tilde{\eta}$, by direct calculation, it can be checked that
the principal curvatures of the shape operator $S_{\tilde{\eta}}$ on $\tilde{N}^{2m-2}_k$ are still
$$
0,\ -\sqrt{\frac{1-B}{B}},\ \sqrt{\frac{1-B}{B}},\ \sqrt{\frac{B}{1-B}},\ -\sqrt{\frac{B}{1-B}},
$$
where the eigenspace corresponding to $0$ is $D_5$. Note that, for unit normal vector field $\tilde{\eta}$, the normal Jacobi operator is given by
$$
\bar{R}_{\tilde{\eta}}Z=\bar{R}(Z,\tilde{\eta})\tilde{\eta}=Z-g(Z,\tilde{\eta})\tilde{\eta}
+3g(Z,J\tilde{\eta})J\tilde{\eta}-g(Z,A\tilde{\eta})A\tilde{\eta}-g(Z,JA\tilde{\eta})JA\tilde{\eta}.
$$
This implies that $\bar{R}_{\tilde{\eta}}$ has three eigenvalues $0$, $1$, $4$ with corresponding
eigenspaces ${\rm Span}\{\tilde{\eta},A\tilde{\eta},JA\tilde{\eta}\}$,
$\oplus_{i=1}^{4} D_i$ and ${\rm Span}\{J\tilde{\eta}\}$.

For any given $\epsilon\in(0,\frac{\pi}{4})$, let
$\Omega_\epsilon=\{[z]\in\tilde{N}^{2m-2}_k\mid \sin^2(\epsilon)<B([z])<\cos^2(\epsilon)\}$ be a open subset of $\tilde{N}^{2m-2}_k$.
We consider the tube $\Phi_r(\Omega_\epsilon)$ over $\Omega_\epsilon$ at distance $r\in(0,\epsilon)$.
To calculate the principal curvatures of $\Phi_r(\Omega_\epsilon)$, we use the Jacobi field method as described in [\cite{B-C-O}, Sec. 8.2]. Let $\gamma(t)$ be the
geodesic in $Q^m$ with $\gamma(0)=[z]\in \Omega_\epsilon$ and $\dot{\gamma}(0)=\tilde{\eta}_{[z]}$ and denote by $\gamma^\perp$ the parallel
subbundle of $TQ^m$ along $\gamma(t)$ defined by $\gamma^\perp_{\gamma(t)}
=T_{\gamma(t)}Q^m\ominus \mathbb{R}\dot{\gamma}(t)$. Moreover, define
the $\gamma^\perp$-valued tensor field $\bar{R}^\perp_{\gamma}$ along $\gamma(t)$ by
$\bar{R}^\perp_{\gamma(t)}X=\bar{R}(X,\dot{\gamma}(t))\dot{\gamma}(t)$. Now consider
the End$(\gamma^\perp)$-valued differential equation
$$
Y''+\bar{R}^\perp_\gamma\circ Y=0.
$$
Let $D$ be the unique solution of this differential equation with initial values
$$
D(0)=\left(
       \begin{array}{cc}
         I & 0 \\
         0 & 0 \\
       \end{array}
     \right),\ \
     D'(0)=\left(
             \begin{array}{cc}
               -S_{\dot{\gamma}(0)} & 0 \\
               0 & I \\
             \end{array}
           \right),
$$
where the decomposition of the matrices is with respect to
$T_{[z]}\tilde{N}^{2m-2}_k\oplus(T^{\bot}_{[z]}\tilde{N}^{2m-2}_k\ominus \tilde{\eta}_{[z]})$.
Here $I$ denotes the identity transformation on the corresponding space. Then the
shape operator $S_r$ of the tube $\Phi_r(\Omega_\epsilon)$ with respect to
$-\dot{\gamma}(r)$ is given by $S_r=D'(r)\circ D^{-1}(r)$.
We decompose $\gamma^\perp_{[z]}$ further into
$$
\gamma^\perp_{[z]}=\oplus_{i=1}^{4} D_i\oplus{\rm Span}\{A\tilde{\eta}_{[z]},JA\tilde{\eta}_{[z]}\}
\oplus{\rm Span}\{J\tilde{\eta}_{[z]}\}.
$$
By explicit computation, one can know that $\Phi_r(\Omega_\epsilon)$ is a Hopf hypersurface of $Q^m$
with $\mathfrak{A}$-isotropic unit normal vector field. The principal curvatures of the real hypersurface $\Phi_r(\Omega_\epsilon)$ at $\gamma(r)$ with respect to $-\dot{\gamma}(r)$ are given by
$$
\begin{aligned}
&\lambda_1=\frac{\sqrt{\frac{1-B}{B}}\cos r-\sin r}{\cos r+\sqrt{\frac{1-B}{B}}\sin r},
\ \lambda_2=\frac{-\sqrt{\frac{1-B}{B}}\cos r-\sin r}{\cos r-\sqrt{\frac{1-B}{B}}\sin r},
\ \lambda_3=\frac{-\sin r+\sqrt{\frac{B}{1-B}}\cos r}{\cos r+\sqrt{\frac{B}{1-B}}\sin r},\\
&\ \lambda_4=\frac{-\sin r-\sqrt{\frac{B}{1-B}}\cos r}{\cos r-\sqrt{\frac{B}{1-B}}\sin r},
\ \lambda_5=0,\ \ \alpha=2\cot(2r),\ \ S\xi=2\cot(2r)\xi,
\end{aligned}
$$
where $\xi$ is the Reeb vector field of $\Phi_r(\Omega_\epsilon)$, and the
eigenspace of $\lambda_5$ is ${\rm Span}\{A\dot{\gamma}(r),JA\dot{\gamma}(r)\}$.
The multiplicities of the corresponding principal curvatures are
$$
{\rm dim}V_{\lambda_1}={\rm dim}V_{\lambda_2}=k-2,\
{\rm dim}V_{\lambda_3}={\rm dim}V_{\lambda_4}=m-k,\ {\rm dim}V_{\lambda_5}=2,\ {\rm dim}V_\alpha=1.
$$
When $k=2$ or $k=m$, the real hypersurface $\Phi_r(\Omega_\epsilon)$ has 
four distinct principal curvatures. When $2<k<m$, the real hypersurface $\Phi_r(\Omega_\epsilon)$ has
six distinct principal curvatures. 
\end{example}

%==================================================
\section{Parallel hypersurfaces, focal submanifolds and Cartan's formulas}\label{sect:4}

%==================================================
\subsection{Parallel hypersurfaces and focal submanifolds of the Hopf hypersurface with constant principal curvatures and $\mathfrak{A}$-isotropic unit normal vector field}\label{sect:4.1}~

In this subsection, we study the parallel hypersurfaces and the focal submanifolds of Hopf hypersurfaces of $Q^m$ ($m\geq3$) with constant principal curvatures and $\mathfrak{A}$-isotropic unit normal vector field $N$.

Let $M$ be a Hopf hypersurface of $Q^m$ ($m\geq3$) with constant principal curvatures and $\mathfrak{A}$-isotropic unit normal vector field $N$. Then it holds $S\xi=\alpha\xi$. Without loss of generality, up to a sign of the unit normal vector field $N$, we always assume that $\alpha\geq0$.
By Lemma \ref{lemma:2.6} and $\mathcal{Q}=TM\ominus{\rm Span}\{\xi,AN,A\xi\}$,
we know that $SAN=SA\xi=0$, and $\mathcal{Q}$ is $S$-invariant, $J$-invariant and $A$-invariant. We denote the set
of eigenvalues of $S$ restricted to $\mathcal{Q}$ by $\sigma(\mathcal{Q})$.
For any $\lambda\in\sigma(\mathcal{Q})$, let $V_\lambda$ be the corresponding eigenspace
restricted to $\mathcal{Q}$.

We define the map
$$
\Phi_r: M\rightarrow Q^m,\ \ \ \ \ p \mapsto \Phi_r(p)=\exp_{p}(rN_{p}), \quad \Phi_{0}(M)=M,
$$
where $p\in M$ and ${\rm exp}$ is the Riemannian exponential map of $Q^m$
and $N$ is the unit normal vector field of $M$.
Here, we allow that $r$ can be negative. If $r<0$, we means that $\Phi_r(M)$ is obtained by moving distance $|r|$ from $M$ along the unit normal vector field $-N$. Let $\{E_1,\cdots, E_{2m-4},E_{2m-3}=AN,E_{2m-2}=A\xi,E_{2m-1}=\xi\}$ be an orthonormal basis at $p\in M$
such that $SE_i=\lambda_iE_i$ and $E_i\in\mathcal{Q}$ for $1\leq i\leq 2m-4$.
Let $\{E_1^r,\cdots, E_{2m-4}^r,E_{2m-3}^r,E_{2m-2}^r,E_{2m-1}^r\}$
be the parallel translation of $\{E_i\}_{i=1}^{2m-1}$ along the geodesic to the nearby parallel hypersurface $\Phi_r(M)$.
Then by standard Jacobi field theory,
we have the following proposition about the parallel hypersurfaces of $M$.

\begin{proposition}\label{prop:4.1w}
Let $M$ be a Hopf hypersurface of $Q^m$ ($m\geq3$) with constant principal curvatures and $\mathfrak{A}$-isotropic unit normal vector field $N$.
Then, for any $p\in M$, the tangent map of $\Phi_r$ has the following expression:
\begin{equation}\label{eqn:rank}
\left(
  \begin{array}{c}
    d\Phi_r(E_1) \\
    \vdots \\
    d\Phi_r(E_{2m-4}) \\
    d\Phi_r(E_{2m-3}) \\
    d\Phi_r(E_{2m-2}) \\
    d\Phi_r(E_{2m-1}) \\
  \end{array}
\right)=(B_{ij})
\left(
  \begin{array}{c}
    E_1^r \\
    \vdots \\
    E_{2m-4}^r \\
    E_{2m-3}^r \\
    E_{2m-2}^r \\
    E_{2m-1}^r \\
  \end{array}
\right),
\end{equation}
where
\begin{equation}\label{Bij111}
(B_{ij})=
\left(
\begin{array}{cccccc}
\cos(r)-\lambda_1\sin(r) & & & & & \\
 & \ddots & & & & \\
 & & \cos(r)-\lambda_{2m-4}\sin(r) & & & \\
  & &  & 1 & & \\
   & &  & & 1  & \\
    & &  & & & \cos(2r)-\frac{\alpha}{2}\sin(2r) \\
\end{array}
\right).
\end{equation}
The determinant of $(B_{ij})$ is given by
\begin{equation}\label{eqn:det}
{\rm Det}(B_{ij})=\prod_{i=1}^{2m-4}\Big(\cos(r)-\lambda_i\sin(r)\Big)\Big(\cos(2r)-\frac{\alpha}{2}\sin(2r)\Big).
\end{equation}

Furthermore, let $S_r$ be the shape operator of the parallel hypersurface $\Phi_r(M)$ with
respect to the unit normal vector field $\frac{d\exp_{p}(rN_{p})}{dr}$, where $p\in M$. Then the expression of $S_r$ at $\Phi_r(p)$ is given by
\begin{equation}\label{eqn:Arr}
\left(
  \begin{array}{c}
    S_rE_1^r \\
    \vdots \\
    S_r E_{2m-4}^r \\
    S_r E_{2m-3}^r \\
    S_r E_{2m-2}^r \\
    S_r E_{2m-1}^r \\
  \end{array}
\right)=
(C_{ij})
\left(
  \begin{array}{c}
    E_1^r \\
    \vdots \\
    E_{2m-4}^r \\
    E_{2m-3}^r \\
    E_{2m-2}^r \\
    E_{2m-1}^r \\
  \end{array}
\right),
\end{equation}
where
\begin{equation}\label{Bij222}
(C_{ij})=
\left(
\begin{array}{cccccc}
\frac{\sin(r)+\lambda_1\cos(r)}{\cos(r)-\lambda_1\sin(r)} & & & & & \\
 & \ddots & & & & \\
 & & \frac{\sin(r)+\lambda_{2m-4}\cos(r)}{\cos(r)-\lambda_{2m-4}\sin(r)} & & & \\
  & &  & 0 & & \\
   & &  & & 0 & \\
    & &  & & & \frac{2\sin(2r)+\alpha\cos(2r)}{\cos(2r)-\frac{\alpha}{2}\sin(2r)} \\
\end{array}
\right).
\end{equation}
\end{proposition}
\begin{proof}
At $p\in M$, let $N_p$ be a unit normal vector of $M$. Since $M$ has $\mathfrak{A}$-isotropic unit normal vector field $N$, the four vectors
$N_p, \xi_p ,AN_p,A\xi_p$ are pairwise orthonormal and the normal Jacobi operator
is given by
$$
\bar{R}_NZ=\bar{R}(Z,N)N=Z-g(Z,N)N+3g(Z,JN)JN-g(Z,AN)AN-g(Z,JAN)JAN.
$$
This implies that $\bar{R}_N$ has three eigenvalues $0$, $1$, $4$ with corresponding
eigenspaces ${\rm Span}\{N_p,AN_p,A\xi_p\}$, $\mathcal{Q}_p$ and ${\rm Span}\{\xi_p\}$.

To calculate the principal curvatures of the parallel hypersurface $\Phi_r(M)$ around
$M$, we use the Jacobi field method as described in [\cite{B-C-O}, Sec. 8.2]. Let $\gamma$ be the
geodesic in $Q^n$ with $\gamma(0)=p\in M$ and $\dot{\gamma}(0)=N_p$ and denote by $\gamma^\perp$ the parallel
subbundle of $TQ^m$ along $\gamma$ defined by $\gamma^\perp_{\gamma(t)}
=T_{\gamma(t)}Q^m\ominus \mathbb{R}\dot{\gamma}(t)$. Moreover, define
the $\gamma^\perp$-valued tensor field $\bar{R}^\perp_{\gamma}$ along $\gamma$ by
$\bar{R}^\perp_{\gamma(t)}X=\bar{R}(X,\dot{\gamma}(t))\dot{\gamma}(t)$. Now consider
the End$(\gamma^\perp)$-valued differential equation
$$
Y''+\bar{R}^\perp_\gamma\circ Y=0.
$$
Let $D$ be the unique solution of this differential equation with initial values
$$
D(0)=I,\ \ D'(0)=-S,
$$
where $I$ denotes the identity transformation. We decompose $\gamma^\perp_p$ further into
$$
\gamma^\perp_p=\oplus_{i=1}^{2m-4} V_{\lambda_i}\oplus{\rm Span}\{AN_p,A\xi_p\}\oplus{\rm Span}\{\xi_p\}.
$$
By explicit computation, we obtain \eqref{eqn:rank}. Moreover, the
shape operator $S_r$ of the parallel hypersurface $\Phi_r(M)$ around $M$ with respect to
$\dot{\gamma}(r)$ is given by $S_r=-D'(r)\circ D^{-1}(r)$.
Then by further computation, we can have \eqref{eqn:Arr}.
\end{proof}

From Proposition \ref{prop:4.1w} that, the parallel hypersurface of a Hopf hypersurface $M$ of $Q^m$ with constant principal curvatures and $\mathfrak{A}$-isotropic unit normal vector field has constant mean curvature
\begin{equation}\label{eqn:Hr}
H(r)={\rm Tr}S_r=\sum_{i=1}^{2m-4}\Big(\frac{\sin(r)+\lambda_i\cos(r)}{\cos(r)-\lambda_i\sin(r)}\Big)
+\frac{2\sin(2r)+\alpha\cos(2r)}{\cos(2r)-\frac{\alpha}{2}\sin(2r)}.
\end{equation}

Recall that, let $M$ be an orientable hypersurface of Riemannian manifold $(\overline{M}, g)$.
We say that $M$ is an isoparametric hypersurface of $\overline{M}$, i.e.,
there exists an isoparametric function $F:\overline{M} \rightarrow \mathbb{R}$ such that $M=F^{-1}(l)$,
for some regular value $l$ of $F$. Here $F$ is called an isoparametric function
if the gradient and the Laplacian of $F$ satisfy
$$
\|\nabla F\|^{2}=f_1(F), \quad \Delta F=f_2(F),
$$
where $f_1, f_2: \mathbb{R} \rightarrow \mathbb{R}$ are smooth functions.
In addition to above definition,
there is another equivalent characterization for isoparametric hypersurfaces. A hypersurface of a Riemannian manifold
is isoparametric if and only if its locally defined parallel hypersurfaces
have constant mean curvature.
It follows from \eqref{eqn:Arr} and \eqref{eqn:Hr} that a Hopf hypersurface $M$ of $Q^m$ with constant principal curvatures and $\mathfrak{A}$-isotropic unit normal vector field is an isoparametric hypersurface,
and all parallel hypersurfaces of $M$ have constant principal curvatures.
On the other hand, by the fact that Hopf hypersurface of $Q^m$ with constant principal curvatures has either $\mathfrak{A}$-principal unit normal vector field or $\mathfrak{A}$-isotropic unit normal vector field,
and the Hopf hypersurface of $Q^m$ with $\mathfrak{A}$-principal unit normal vector field is
an open part of a tube over a totally geodesic $Q^{m-1}\hookrightarrow Q^{m}$.
Note that, the homogeneous real hypersurfaces
are isoparametric hypersurfaces. Thus, in fact, we get the following result:
\begin{theorem}\label{thm:4.2w}
Let $M$ be a Hopf hypersurface of $Q^m$ ($m\geq3$) with constant principal curvatures. Then $M$ is an isoparametric hypersurface. Moreover, all parallel hypersurfaces of $M$ have constant principal curvatures.
\end{theorem}

%{\color{red}\begin{remark}\label{rem:4.1asas}
%We point out that above discussion is still hold for $Q^2$, it means that Hopf hypersurfaces of $Q^2$ with constant principal curvatures are isoparametric hypersurfaces. Then due to the fact that
%the complex quadric $Q^2$ with Einstein constant $1$ is holomorphically isometric to the K\"ahler surface
%$\mathbb{S}^2\times\mathbb{S}^2$, and using the classification of isoparametric hypersurfaces
%of $\mathbb{S}^2\times\mathbb{S}^2$ obtained by Urbano \cite{Ur}, it follows that
%Hopf hypersurface of $\mathbb{S}^2\times\mathbb{S}^2$ with constant principal curvatures is either
%an open part of $\Gamma\times \mathbb{S}^2$, where $\Gamma$ is a curve of $\mathbb{S}^2$ with
%constant curvature, or an open part of $M_t$ for some $t\in (-1,1)$. Note that the classification of
%Hopf hypersurface of $\mathbb{S}^2\times\mathbb{S}^2$ with constant principal curvatures
%can also be obtained by Theorem 1.3 of \cite{ZGHY}.
%\end{remark}}

In the following, we study the focal submanifolds of a Hopf hypersurface of $Q^n$ with constant principal curvatures and $\mathfrak{A}$-isotropic unit normal vector field $N$. Note that, if $\Phi_r(M)$ is a focal
submanifold of $M$, then it holds ${\rm Det}(B_{ij})=0$ on $\Phi_r(M)$. So, in order to find the distance $|r|$, we observe the factors $\cos(r)-\lambda_i\sin(r)$ for $1\leq i\leq2m-4$ and $\cos(2r)-\frac{\alpha}{2}\sin(2r)$,
and find out two closest distances $r_1$ and $-r_2$ from $0$ such that ${\rm Det}(B_{ij})$ equals to $0$.
%Without loss of generality, we always assume that $\alpha\geq0$.
Recall that
$\lambda_i\in \sigma(\mathcal{Q})$, $1\leq i\leq2m-4$.
Now, we define
$\lambda_+={\rm max}\{\lambda_i,\frac{\alpha+\sqrt{\alpha^2+4}}{2}\}$ and $\lambda_-={\rm min}\{\lambda_i,\frac{\alpha-\sqrt{\alpha^2+4}}{2}\}$.
Then, the two closest distances $r_1$ and $-r_2$ satisfy $\cot(r_1)=\lambda_+$ ($0<r_1<\frac{\pi}{2}$) and $\cot(r_2)=\lambda_-$ ($-\frac{\pi}{2}<r_2<0$). %Here, the distance $r<0$ is with respect to $-N$.

For the distance $r_1$, according to the computation of Proposition \ref{prop:4.1w},
we know that $M_+:=\Phi_{r_1}(M)$ is a focal
submanifold of $M$. If $\lambda_+>\frac{\alpha+\sqrt{\alpha^2+4}}{2}$, then
${\rm dim}M_+=2m-{\rm dim}V_{\lambda_+}-1$. If $\lambda_+=\frac{\alpha+\sqrt{\alpha^2+4}}{2}$,
then ${\rm dim}M_+=2m-{\rm dim}V_{\lambda_+}-2$. Here, $V_{\lambda_+}$ is the corresponding eigenspace
restricted to $\mathcal{Q}$. Now, we have the following result about $M_+$.

\begin{proposition}\label{prop:4.3w}
Let $M$ be a Hopf hypersurface of $Q^m$ ($m\geq3$) with constant principal curvatures and $\mathfrak{A}$-isotropic unit normal vector field $N$. Then the focal submanifold $M_+$ has constant principal curvatures with respect to any unit normal vector field, and $M_+$ is austere.
Moreover, if $\lambda_+>\frac{\alpha+\sqrt{\alpha^2+4}}{2}$, then
the constant principal curvatures of $M_+$ are
$$
0,\ \ \ \alpha+\frac{(4+\alpha^2)\lambda_+}{\lambda_+^2-\alpha\lambda_+-1},
\ \ \ \ \frac{1+\lambda_+\lambda_i}{\lambda_+-\lambda_i},\ \  {\rm where}\ \  \lambda_i<\lambda_+.
$$
If $\lambda_+=\frac{\alpha+\sqrt{\alpha^2+4}}{2}$, then
the constant principal curvatures of $M_+$ are
$$
0,\ \ \ \frac{1+\lambda_+\lambda_i}{\lambda_+-\lambda_i},\ \  {\rm where}\ \  \lambda_i<\lambda_+.
$$
For each of the above situations, it holds that
$$
\frac{1+\lambda_+\lambda_i}{\lambda_+-\lambda_i}
<\frac{1+\lambda_+\lambda_j}{\lambda_+-\lambda_j},
\ \ {\rm for}\ \ \lambda_i<\lambda_j<\lambda_+.
$$
\end{proposition}
\begin{proof}
According to Proposition \ref{prop:4.1w} and Theorem \ref{thm:4.2w}, we know that
$M$ is an isoparametric hypersurface,
and all its parallel hypersurfaces of $M$ have constant principal curvatures.
Then, by using Theorem 1.2 obtained by Ge-Tang \cite{GT}, the focal submanifold $M_+$ is austere.
Furthermore, by the Jacobi field method as used in Proposition \ref{prop:4.1w},
by taking $r=r_1$, one can get that the constant principal curvatures of $M_+$ with respect to any unit normal vector field are given as above.

Note that, if we fix $\lambda_+$ , then the function
$\frac{1+{\lambda_+}\lambda_i}{{\lambda_+}-\lambda_i}$ is increasing with respect to $\lambda_i$
for $\lambda_i<\lambda_+$.
\end{proof}

We recall that a submanifold is said to be austere if its multiset of principal curvatures is invariant under change of sign. In particular, austere submanifolds are automatically minimal.
This concept was introduced by Harvey-Lawson \cite{HL} for constructing special
Lagrangian submanifolds in $\mathbb{C}^n$.

For the distance $-r_2$, according to the computation of Proposition \ref{prop:4.1w}, we know that $M_-:=\Phi_{r_2}(M)$ is a focal
submanifold of $M$. If $\lambda_-<\frac{\alpha-\sqrt{\alpha^2+4}}{2}$, then
${\rm dim}M_-=2m-{\rm dim}V_{\lambda_-}-1$. If $\lambda_-=\frac{\alpha-\sqrt{\alpha^2+4}}{2}$, then ${\rm dim}M_-=2m-{\rm dim}V_{\lambda_-}-2$. Here, $V_{\lambda_-}$ is the corresponding eigenspace
restricted to $\mathcal{Q}$. Then, similar as the proof of Proposition \ref{prop:4.3w}, we have the following result about $M_-$.

\begin{proposition}\label{prop:4.4w}
Let $M$ be a Hopf hypersurface of $Q^m$ ($m\geq3$) with constant principal curvatures and $\mathfrak{A}$-isotropic unit normal vector field $N$. Then focal submanifold $M_-$ has constant principal curvatures with respect to any unit normal vector field, and $M_-$ is austere.
Moreover, if $\lambda_-<\frac{\alpha-\sqrt{\alpha^2+4}}{2}$,
the constant principal curvatures of $M_-$ are
$$
0,\ \ \ \alpha+\frac{(4+\alpha^2){\lambda_-}}{\lambda_-^2-\alpha{\lambda_-}-1}, \ \ \ \ \frac{1+{\lambda_-}\lambda_i}{{\lambda_-}-\lambda_i},\ \  {\rm where}\ \  {\lambda_-}<\lambda_i.
$$
If $\lambda_-=\frac{\alpha-\sqrt{\alpha^2+4}}{2}$, the constant principal curvatures of $M_-$ are
$$
0,\ \ \ \frac{1+{\lambda_-}\lambda_i}{{\lambda_-}-\lambda_i},\ \  {\rm where}\ \  {\lambda_-}<\lambda_i.
$$
For each of the above situations, it holds that
$$
\frac{1+{\lambda_-}\lambda_i}{{\lambda_-}-\lambda_i}
<\frac{1+{\lambda_-}\lambda_j}{{\lambda_-}-\lambda_j},
\ \ {\rm for}\ \ {\lambda_-}<\lambda_i<\lambda_j.
$$
\end{proposition}

\begin{remark}\label{rem:4.1aa}
%Note that, if we fix $\lambda_+$ (or $\lambda_-$), then the function
%$\frac{1+{\lambda_+}\lambda_i}{{\lambda_+}-\lambda_i}$ (or  $\frac{1+{\lambda_-}\lambda_i}{{\lambda_-}-\lambda_i}$) is increasing with respect to $\lambda_i$.
If $\lambda_+>\frac{\alpha+\sqrt{\alpha^2+4}}{2}$ (or $\lambda_-<\frac{\alpha-\sqrt{\alpha^2+4}}{2}$), in order to arrange the principal curvatures of $M_+$ (or $M_-$) from small to large,
we only need to consider the relationship between $\alpha+\frac{(4+\alpha^2){\lambda_+}}{{\lambda_+}^2-\alpha{\lambda_+}-1}$ and $\frac{1+{\lambda_+}\lambda_i}{{\lambda_+}-\lambda_i}$
(or $\alpha+\frac{(4+\alpha^2){\lambda_-}}{{\lambda_-}^2-\alpha{\lambda_-}-1}$ and $\frac{1+{\lambda_-}\lambda_i}{{\lambda_-}-\lambda_i}$).
Then, by the fact that $M_+$ (or $M_-$) is austere, we can get some
relationship between the principal curvatures $\{\alpha,0,\lambda_1,...,\lambda_{2m-4}\}$ of $M$.
\end{remark}

%==================================================
\subsection{Cartan's formulas for Hopf hypersurfaces with constant principal curvatures and $\mathfrak{A}$-isotropic unit normal vector field}\label{sect:4.2}~

Let $M$ be a Hopf hypersurface with constant principal curvatures and $\mathfrak{A}$-isotropic unit normal vector field $N$.
Suppose tangent vector fields $X,Y\in TM$ satisfy $SX=\lambda X$ and $SY=\mu Y$, then it holds
$$
g((\nabla_Z S)X,Y)=(\lambda-\mu)g(\nabla_Z X,Y),
$$
for any tangent vector $Z$.

Observe that distribution $\mathcal{Q}$ is $S$-invariant, for any $\lambda\in\sigma(\mathcal{Q})$, let $V_\lambda$ be the corresponding eigenspace
restricted to $\mathcal{Q}$.

\begin{lemma}\label{lemma:5.1}
Let $M$ be a Hopf hypersurface of $Q^m$ ($m\geq3$) with constant principal curvatures and $\mathfrak{A}$-isotropic unit normal vector field $N$. For any $\lambda$ in $\sigma(\mathcal{Q})$ and any almost product structure $A\in\mathfrak{A}$,
we have
\begin{enumerate}
\item[(1)]
$\nabla_X (AN)=q(X)A\xi-\lambda AX$ for all $X\in V_\lambda$, where $q$ is the corresponding $1$-form of $A$ described in \eqref{eqn:2.2};

\item[(2)]
$\nabla_X (A\xi)=-q(X)AN-\lambda JAX$ for all $X\in V_\lambda$, where $q$ is the corresponding $1$-form of $A$
described in \eqref{eqn:2.2};

\item[(3)]
$AV_\lambda\bot {\rm Span}\{V_\lambda, JV_\lambda\}$, i.e.
$$
g(AX,Y)=g(AX,JY)=0,\ for\ all\  X,Y\in V_\lambda.
$$
\end{enumerate}
\end{lemma}
\begin{proof}
For (1), by \eqref{eqn:2.2}, we have
$$
\begin{aligned}
\nabla_X (AN)&=\bar{\nabla}_X (AN)-g(X,SAN)N=\bar{\nabla}_X (AN)=(\bar{\nabla}_X A)N+A\bar{\nabla}_X N\\
&=q(X)JAN-ASX=q(X)A\xi-\lambda AX, \ \ \forall\ X\in V_\lambda.
\end{aligned}
$$

For (2), by \eqref{eqn:2.2}, we have
$$
\begin{aligned}
\nabla_X (A\xi)&=\bar{\nabla}_X (A\xi)-g(X,SA\xi)N=\bar{\nabla}_X (A\xi)=(\bar{\nabla}_X A)\xi+A\bar{\nabla}_X \xi\\
&=q(X)JA\xi+A\nabla_X \xi=-q(X)AN+A\phi SX=-q(X)AN+\lambda A\phi X\\
&=-q(X)AN-\lambda JAX, \ \ \forall\ X\in V_\lambda.
\end{aligned}
$$

For (3), on the one hand, by Codazzi equation \eqref{eqn:2.9}, we have
$$
g((\nabla_X S)AN-(\nabla_{AN} S)X,Y)=-g(AX,Y), \ \ \forall\ X,Y\in V_\lambda.
$$
On the other hand, by above (1), we can have
$$
\begin{aligned}
g((\nabla_X S)AN&-(\nabla_{AN} S)X,Y)=g(-S\nabla_X (AN),Y)=-\lambda g(\nabla_X (AN),Y)\\
&=-\lambda g(q(X)JAN-\lambda AX,Y)=\lambda^2 g(AX,Y), \ \ \forall\ X,Y\in V_\lambda.
\end{aligned}
$$
Combining above two equations, we have $g(AX,Y)=0$, for all $X,Y\in V_\lambda$.

Similarly, by Codazzi equation \eqref{eqn:2.9}, we have
$$
g((\nabla_X S)A\xi-(\nabla_{A\xi} S)X,Y)=-g(JAX,Y)=g(AX,JY), \ \ \forall\ X,Y\in V_\lambda.
$$
On the other hand, by above (2), we can have
$$
\begin{aligned}
g((\nabla_X S)A\xi&-(\nabla_{A\xi} S)X,Y)=g(-S\nabla_X (A\xi),Y)=-\lambda g(\nabla_X (A\xi),Y)\\
&=-\lambda g(-q(X)AN-\lambda JAX,Y)=-\lambda^2 g(AX,JY), \ \ \forall\ X,Y\in V_\lambda.
\end{aligned}
$$
Combining above two equations, we have $g(AX,JY)=0$, for all $X,Y\in V_\lambda$.
\end{proof}

\begin{lemma}\label{lemma:5.2}
Let $M$ be a Hopf hypersurface of $Q^m$ ($m\geq3$) with constant principal curvatures and $\mathfrak{A}$-isotropic unit normal vector field $N$. For all $\lambda,\mu$ in $\sigma(\mathcal{Q})$, we have
\begin{enumerate}
\item[(1)]
$\nabla_X Y+\lambda g(\phi X,Y)\xi\in V_\lambda$ for all $X,Y\in V_\lambda$;

\item[(2)]
$\nabla_X Y\bot V_\lambda$ if $X\in V_\lambda$, $Y\in V_\mu$, $\lambda\neq\mu$.
\end{enumerate}
\end{lemma}
\begin{proof}
By Lemma \ref{lemma:5.1}, for any $X,Y\in V_\lambda$, we have
$$
g(\nabla_X Y,\xi)=-g(Y,\nabla_X \xi)=-g(Y,\phi SX)=-\lambda g(\phi X,Y),
$$
$$
g(\nabla_X Y,AN)=-g(Y,\nabla_X (AN))=-g(Y,q(X)A\xi-\lambda AX)=0,
$$
$$
g(\nabla_X Y,A\xi)=-g(Y,\nabla_X (A\xi))=-g(Y,-q(X)AN-\lambda JAX)=0.
$$
It is easy to see that $\nabla_X Y+\lambda g(\phi X,Y)\xi\in\mathcal{Q}$.
Now take any $\mu\in \sigma(\mathcal{Q})$ with $\mu\neq\lambda$ and choose any $Z\in V_\mu$.
By the Codazzi equation,
$$
\begin{aligned}
0=g((\nabla_X S)Z&-(\nabla_Z S)X,Y)=g(\mu\nabla_X Z-S\nabla_X Z,Y)=(\mu-\lambda)g(\nabla_X Z,Y)\\
&=(\lambda-\mu)g(\nabla_X Y,Z), \ \ \forall\ X,Y\in V_\lambda, \ Z\in V_\mu, \ \mu\neq\lambda.
\end{aligned}
$$
Thus, $\nabla_X Y+\lambda g(\phi X,Y)\xi\in V_\lambda$. Note that
the second assertion follows from the first since $g(\nabla_X Z,Y)=-g(\nabla_X Y,Z)$.
\end{proof}

%From the first assertion in the lemma above, we immediately get the following
%corollary.

%\begin{corollary}\label{cor:5.1}
%Under the hypothesis of Lemma \ref{lemma:5.2}, $V_\lambda$ is integrable if and only if
%$\lambda=0$ or $\phi V_\lambda\subset V_\lambda^\bot$.
%\end{corollary}

In the following, we give two Cartan's formulas (see \eqref{eqn:ca2} and \eqref{eqn:call}) for the Hopf hypersurface of $Q^m$ ($m\geq3$) with constant principal curvatures and $\mathfrak{A}$-isotropic unit normal vector field, which play an important role in the subsequent proof, especially \eqref{eqn:call}.

\begin{lemma}\label{lemma:5.3}
Let $M$ be a Hopf hypersurface of $Q^m$ ($m\geq3$) with constant principal curvatures and $\mathfrak{A}$-isotropic unit normal vector field $N$. Let $X\in \mathcal{Q}$ be a unit principal vector at a point $p$ with
associated principal curvature $\lambda$. For any principal orthonormal basis $\{e_i\}_{i=1}^{2m-4}$ of
$\mathcal{Q}$ satisfying $Se_i=\mu_ie_i$, we have
\begin{equation}\label{eqn:ca2}
\sum_{i=1, \mu_i\neq\lambda}^{2m-4}\frac{1+\lambda\mu_i}{\lambda-\mu_i}\Big(1+2g(\phi X,e_i)^2-2g(AX,e_i)^2-2g(AX,Je_i)^2\Big)=0.
\end{equation}
\end{lemma}
\begin{proof}
Let $\mu\neq\lambda$, $\mu\in\sigma(\mathcal{Q})$ and $Y\in \mathcal{Q}$ be another unit
principal vector at $p$ with corresponding principal
curvature $\mu$. Extend $X$ and $Y$ to be principal vector fields near $p$.

(1). First note that $g(\nabla_X Y,\xi)=-g(Y,\nabla_X \xi)=-g(Y,\phi SX)=-\lambda g(\phi X,Y)$. Similarly,
$$
g(\nabla_Y X,\xi)=-\mu g(\phi Y,X)=\mu g(\phi X,Y).
$$
Thus, $g([X,Y],\xi)=-(\lambda+\mu)g(\phi X,Y)$.

By $g(\nabla_X Y,AN)=-g(Y,\nabla_X (AN))=-g(Y,q(X)A\xi-\lambda AX)=\lambda g(A X,Y)$, and
$g(\nabla_Y X,AN)=\mu g(AX,Y)$, we have $g([X,Y],AN)=(\lambda-\mu)g(AX,Y)$.

By $g(\nabla_X Y,A\xi)=-g(Y,\nabla_X (A\xi))=-g(Y,-q(X)AN-\lambda JAX)=-\lambda g(AX,JY)$, and
$g(\nabla_Y X,A\xi)=-\mu g(AX,JY)$, we have $g([X,Y],A\xi)=-(\lambda-\mu)g(AX,JY)$.

Next, using Codazzi equation \eqref{eqn:2.9}, we compute
\begin{equation}\label{eqn:5.9}
\begin{aligned}
g(&(\nabla_{[X,Y]}S)X,Y)=g((\nabla_X S)[X,Y],Y)+g([X,Y],\xi)g(\phi X,Y)\\
&\qquad +g([X,Y],AN)g(AX,Y)+g([X,Y],A\xi)g(JAX,Y)\\
&=g((\nabla_X S)[X,Y],Y)-(\lambda+\mu)g(\phi X,Y)^2+(\lambda-\mu)g(AX,Y)^2+(\lambda-\mu)g(AX,JY)^2\\
&=g((\nabla_Y S)X,[X,Y])-2g(\phi X,Y)g([X,Y],\xi)-(\lambda+\mu)g(\phi X,Y)^2+(\lambda-\mu)\{g(AX,Y)^2+g(AX,JY)^2\}\\
&=g((\nabla_Y S)X,[X,Y])+(\lambda+\mu)g(\phi X,Y)^2+(\lambda-\mu)g(AX,Y)^2+(\lambda-\mu)g(AX,JY)^2.
\end{aligned}
\end{equation}
But now,
$$
g(\nabla_X Y,(\nabla_Y S)X)=g((\lambda I-S)\nabla_Y X,\nabla_X Y),
$$
while
$$
\begin{aligned}
g(\nabla_Y X,(\nabla_Y S)X)&=g(\nabla_Y X,(\nabla_X S)Y)-2g(\phi Y,X)g(\nabla_Y X,\xi)\\
&=g((\nabla_X S)Y,\nabla_Y X)+2g(\phi X,Y)g(\nabla_Y X,\xi)\\
&=g((\mu I-S)\nabla_X Y,\nabla_Y X)+2\mu g(\phi X,Y)^2.
\end{aligned}
$$
Thus
\begin{equation}\label{eqn:5.10}
g([X,Y],(\nabla_Y S)X)=(\lambda-\mu)g(\nabla_X Y,\nabla_Y X)-2\mu g(\phi X,Y)^2.
\end{equation}
On substituting in equation \eqref{eqn:5.9} we obtain
\begin{equation}\label{eqn:5.1}
g((\nabla_{[X,Y]} S)X,Y)=(\lambda-\mu)\Big\{g(\nabla_X Y,\nabla_Y X)
+g(\phi X,Y)^2+g(AX,Y)^2+g(AX,JY)^2\Big\}.
\end{equation}

(2). By using the Gauss equation \eqref{eqn:2.8} and the fact of (3) in Lemma \ref{lemma:5.1}, we have
\begin{equation}\label{eqn:5.2}
g(R(X,Y)Y,X)=1+\lambda\mu+3g(\phi X,Y)^2-g(AX,Y)^2-g(AX,JY)^2.
\end{equation}

(3). Note that $g(\nabla_Y Y,X)=0$ by Lemma \ref{lemma:5.2}, and so
$$
g(\nabla_X {\nabla_Y Y},X)=-g(\nabla_Y Y,\nabla_X X),
$$
which vanishes, again by Lemma \ref{lemma:5.2}. Similarly, $g(\nabla_X Y,X)=0$, so that
$$
g(\nabla_Y {\nabla_X Y},X)=-g(\nabla_X Y,\nabla_Y X).
$$
So, $g((\nabla_{[X,Y]} S)X,Y)=(\lambda-\mu)g(\nabla_{[X,Y]} X,Y)$. We then compute
$$
g(\nabla_X{\nabla_Y Y}-\nabla_Y {\nabla_X Y}-\nabla_{[X,Y]} Y,X)=g(\nabla_X Y,\nabla_Y X)
+\frac{1}{\lambda-\mu}g((\nabla_{[X,Y]} S)X,Y),
$$
which gives
\begin{equation}\label{eqn:5.3}
g(R(X,Y)Y,X)=g(\nabla_X Y,\nabla_Y X)+\frac{1}{\lambda-\mu}g((\nabla_{[X,Y]} S)X,Y).
\end{equation}

(4). For a unit principal vector $Z\in \mathcal{Q}$
corresponding to a principal curvature $\nu$ not equal to $\lambda$ or $\mu$. Extend $Z$ to be principal vector fields near $p$.
By using the Codazzi equation \eqref{eqn:2.9}, we compute
$$
g((\nabla_Z S)X,Y)=g((\nabla_X S)Z,Y)=g(Z,(\nabla_X S)Y)=(\mu-\nu)g(Z,\nabla_X Y).
$$
The same calculation with $X$ and $Y$ interchanged gives
$$
g((\nabla_Z S)X,Y)=(\lambda-\nu)g(Z,\nabla_Y X).
$$
Multiplying these two equations together gives
\begin{equation}\label{eqn:5.4}
(\lambda-\nu)(\mu-\nu)g(\nabla_X Y,Z)g(\nabla_Y X,Z)=g((\nabla_Z S)X,Y)^2.
\end{equation}

(5). Note that
$$
g(\nabla_X Y,\xi)g(\nabla_Y X,\xi)=g(\phi SX,Y)g(\phi SY,X)=\lambda\mu g(\phi X,Y)g(\phi Y,X)=-\lambda\mu g(\phi X,Y)^2,
$$
$$
\begin{aligned}
g(\nabla_X Y,AN)&g(\nabla_Y X,AN)=g(Y,\nabla_X (AN))g(X,\nabla_Y (AN))\\
&=g(Y,q(X)A\xi-\lambda AX)g(X,q(Y)A\xi-\mu AY)
=\lambda\mu g(AX,Y)^2,
\end{aligned}
$$
$$
\begin{aligned}
g(\nabla_X Y,A\xi)&g(\nabla_Y X,A\xi)=g(Y,\nabla_X (A\xi))g(X,\nabla_Y (A\xi))\\
&=g(Y,-q(X)AN-\lambda JAX)g(X,-q(Y)AN-\mu JAY)=\lambda\mu g(AX,JY)^2.
\end{aligned}
$$
Thus, in order to express $g(\nabla_X Y,\nabla_Y X)$ in terms of the orthonormal principal basis,
we need only observe that the terms omitted from the full summation of
the $g(\nabla_X Y,e_i)g(\nabla_Y X,e_i)$ (i.e., those $i$ for which $\mu_i=\lambda$ or $\mu_i=\mu$
actually vanish and make no contribution to the sum). Then by Lemma \ref{lemma:5.2},
we can have
\begin{equation}\label{eqn:5.5}
g(\nabla_X Y,\nabla_Y X)=\sum_{\mu_i\neq\lambda,\mu}g(\nabla_X Y,e_i)g(\nabla_Y X,e_i)
-\lambda\mu\Big\{g(\phi X,Y)^2-g(AX,Y)^2-g(AX,JY)^2\Big\}.
\end{equation}

(6). Combining equations \eqref{eqn:5.1}, \eqref{eqn:5.2} and \eqref{eqn:5.3}, we get
$$
2g(\nabla_X Y,\nabla_Y X)=1+\lambda\mu+2g(\phi X,Y)^2-2g(AX,Y)^2-2g(AX,JY)^2.
$$
Then, we use this and the result of equation \eqref{eqn:5.4} in equation \eqref{eqn:5.5}, we get
\begin{equation}\label{eqn:5.6}
2\sum_{\mu_i\neq\lambda,\mu}\frac{g((\nabla_{e_i} S)X,Y)^2}{(\lambda-\mu_i)(\mu-\mu_i)}
=(1+\lambda\mu)\Big\{1+2g(\phi X,Y)^2-2g(AX,Y)^2-2g(AX,JY)^2\Big\}.
\end{equation}

Now for any $j$ with $\mu_j\neq\lambda$, we have (setting $Y=e_j$ in equation \eqref{eqn:5.6}),
\begin{equation}\label{eqn:5.7}
2\sum_{\mu_i\neq\lambda,\mu_j}\frac{g((\nabla_{e_i} S)e_j,X)^2}{(\lambda-\mu_i)(\lambda-\mu_j)(\mu_j-\mu_i)}
=\frac{1+\lambda\mu_j}{\lambda-\mu_j}\Big\{1+2g(\phi X,e_j)^2-2g(AX,e_j)^2-2g(AX,Je_j)^2\Big\}.
\end{equation}

Summing this over all $j$ for which $\mu_j\neq\lambda$, we have
\begin{equation}\label{eqn:5.8}
\begin{aligned}
2\sum_{i,j;\mu_i\neq\mu_j;\mu_i,\mu_j\neq\lambda}&\frac{g((\nabla_{e_i} S)e_j,X)^2}{(\lambda-\mu_i)(\lambda-\mu_j)(\mu_j-\mu_i)}\\
&=\sum_{\mu_j\neq\lambda}\frac{1+\lambda\mu_j}{\lambda-\mu_j}\Big\{1+2g(\phi X,e_j)^2-2g(AX,e_j)^2-2g(AX,Je_j)^2\Big\}.
\end{aligned}
\end{equation}

Since the summand on the left side of equation \eqref{eqn:5.8} is skew-symmetric in $\{i,j\}$,
the value of the sum is $0$, and so the sum on the right is $0$.
\end{proof}

\begin{remark}\label{rem:4.1a}
The proof of Lemma \ref{lemma:5.3} is inspired by the Cartan's formulas for
Hopf hypersurfaces of $\mathbb{C}H^m$ with constant principal curvatures, which
is obtained by Berndt \cite{B}. For more details of Cartan's formulas, one can also see the book of
Cecil and Ryan \cite{CR}.
\end{remark}

Now, from equation \eqref{eqn:5.6}, we have the following key lemma:

\begin{lemma}\label{lemma:5.4}
Let $M$ be a Hopf hypersurface of $Q^m$ ($m\geq3$) with constant principal curvatures and $\mathfrak{A}$-isotropic unit normal vector field $N$. Let $X,Y\in \mathcal{Q}$ be unit principal vector fields with
associated principal curvatures $\lambda,\mu$ ($\mu\neq\lambda$). For any principal orthonormal basis $\{e_i\}_{i=1}^{2m-4}$ of
$\mathcal{Q}$ satisfying $Se_i=\mu_ie_i$, we have
\begin{equation}\label{eqn:call}
2\sum_{\mu_i\neq\lambda,\mu}\frac{g((\nabla_{e_i} S)X,Y)^2}{(\lambda-\mu_i)(\mu-\mu_i)}
=(1+\lambda\mu)\Big\{1+2g(\phi X,Y)^2-2g(AX,Y)^2-2g(AX,JY)^2\Big\}.
\end{equation}
\end{lemma}

At the end of this subsection, as a direct application of Lemma \ref{lemma:2.3},
one can have the following useful lemma.
%In fact, it is also true at each point.
%Recalling Lemma \ref{lemma:2.3} and that $M$ is a Hopf hypersurface with $\mathfrak{A}$-isotropic normal, we
%can state the following lemma. In fact, it is true at each point.
\begin{lemma}\label{lemma:5.5}
Let $M$ be a Hopf hypersurface of $Q^m$ ($m\geq3$) with constant principal curvatures
and $\mathfrak{A}$-isotropic unit normal vector field $N$, and it holds $S\xi=\alpha\xi$.
Then for every $\lambda\in\sigma(\mathcal{Q})$, there is a unique
$\mu\in\sigma(\mathcal{Q})$ such that the following conditions are satisfied:
\begin{enumerate}
\item[(1)]
$2\lambda-\alpha\not=0$,

\item[(2)]
$\mu=\tfrac{\alpha\lambda+2}{2\lambda-\alpha}$.

\item[(3)]
$\phi V_{\lambda}=V_{\mu}$.
\end{enumerate}
\end{lemma}
\begin{proof}
By $M$ has $\mathfrak{A}$-isotropic unit normal vector field $N$, then $\mathcal{Q}$ is $S$-invariant and $J$-invariant.
Now, for any $\lambda\in\sigma(\mathcal{Q})$, taking a unit principal vector field $X\in\mathcal{Q}$ such that $SX=\lambda X$, then by \eqref{eqn:2.13}, we obtain
\begin{equation}\label{eqn:SS1}
(2\lambda-\alpha)S\phi X=(\alpha\lambda+2)\phi X.
\end{equation}
If $2\lambda-\alpha=0$, then
$\lambda=\frac{\alpha}{2}$, which contradicts with $\alpha\lambda+2=0$. So we have
$S\phi X=\frac{\alpha\lambda+2}{2\lambda-\alpha}\phi X$. It means that $\mu=\tfrac{\alpha\lambda+2}{2\lambda-\alpha}$ is also a principal curvature in $\mathcal{Q}$,
and it holds $\phi V_{\lambda}\subset V_{\mu}$.
Conversely, for any $X\in V_{\mu}$, then by \eqref{eqn:2.13}, we have $S\phi X=\lambda \phi X$.
So $\phi V_{\lambda}=V_{\mu}$.
Thus, we have proved (1), (2) and (3) of Lemma \ref{lemma:5.5}.
\end{proof}

%==================================================
\section{Proof of Theorem \ref{thm:1.1a}}\label{sect:5}

According to Lemma \ref{lemma:2.5}, a Hopf hypersurface in $Q^m$ $(m\ge3)$ with constant
principal curvatures has either $\mathfrak{A}$-principal unit normal vector field or $\mathfrak{A}$-isotropic unit normal vector field. Moreover,
a Hopf hypersurface in $Q^m$ $(m\ge3)$ with $\mathfrak{A}$-principal unit normal vector field is an open part of a tube over a totally geodesic $Q^{m-1}\hookrightarrow Q^{m}$
(see Theorem \ref{thm:2.1}). Thus, when we consider Hopf hypersurfaces of $Q^m$ with constant principal curvatures, we only need to consider the Hopf hypersurfaces with constant principal curvatures and $\mathfrak{A}$-isotropic unit normal vector field. Without loss of generality, up to a sign of the unit normal vector field $N$, we always assume that $\alpha\geq0$.

%==================================================
\subsection{Hopf hypersurfaces of $Q^m$ with at most three distinct constant principal curvatures}\label{sect:5.1}~

For totally umbilical real hypersurface of $Q^m$ ($m\geq3$), we have the following result,
which is a direct consequence of Theorem \ref{thm:2.2}.

\begin{theorem}\label{thm:5.1}
There does not exist totally umbilical real hypersurface of $Q^m$ ($m\geq3$).
\end{theorem}
\begin{proof}
If $M$ is a totally umbilical real hypersurface of $Q^m$ ($m\geq3$), then $S=\lambda {\rm Id}$, so it holds $S\phi=\phi S$ on $M$, which means that $M$ has isometric Reeb flow. According the Theorem \ref{thm:2.2} and the
principal curvatures of Example \ref{E2}, we know that there does not exist
totally umbilical real hypersurface of $Q^m$ ($m\geq3$).
\end{proof}

\begin{theorem}\label{thm:5.2}
There does not exist Hopf hypersurface of $Q^m$ ($m\geq3$) with two distinct constant principal curvatures.
\end{theorem}
\begin{proof}
If $M$ has $\mathfrak{A}$-principal unit normal vector field $N$, then $M$ is an open part of a tube over a totally geodesic $Q^{m-1}\hookrightarrow Q^{m}$, which has three distinct principal curvatures. It is a contradiction.

If $M$ has $\mathfrak{A}$-isotropic unit normal vector field $N$, then by Lemma \ref{lemma:2.6}, we have $S\xi=\alpha\xi$, and $SAN=SA\xi=0$. In the following, we will divide the discussions into two cases depending on the value of the constant $\alpha$.

{\bf Case-i}: $\alpha>0$.

{\bf Case-ii}: $\alpha=0$.

In {\bf Case-i}, we know that $\alpha$ and $0$ are exactly those two distinct principal curvatures on $M$.
Then, at least one of them belongs in $\sigma(\mathcal{Q})$.

If $\alpha\in\sigma(\mathcal{Q})$, then there exists a tangent vector field $X\in\mathcal{Q}$, such that
$SX=\alpha X$. Now, by Lemma \ref{lemma:5.5}, it follows that $S\phi X=\mu\phi X$ and $\mu=\tfrac{\alpha^2+2}{\alpha}$, in which $\tfrac{\alpha^2+2}{\alpha}$ is obviously different from $\alpha$ and $0$. Thus, we get a contradiction.

If $0\in\sigma(\mathcal{Q})$, then there exists a tangent vector field $X\in\mathcal{Q}$, such that
$SX=0$. Now, by Lemma \ref{lemma:5.5}, it follows that $S\phi X=\mu\phi X$ and $\mu=\tfrac{-2}{\alpha}$, in which $\tfrac{-2}{\alpha}$ is obviously different from $\alpha$ and $0$. Thus, we get a contradiction.

In {\bf Case-ii}, we know that there is another one principal curvature $\lambda\neq0$ in $\sigma(\mathcal{Q})$.

If $0\in\sigma(\mathcal{Q})$, then there exists a tangent vector field $X\in\mathcal{Q}$, such that
$SX=0$. Now, by Lemma \ref{lemma:5.5}, it follows that $S\phi X=\mu\phi X$ and $\mu=\tfrac{-2}{\alpha}$,
which contradicts with $\alpha=0$. Thus, it follows that $0\notin\sigma(\mathcal{Q})$, which implies that
$SX=\lambda X$ for all $X\in\mathcal{Q}$. We know that $M$ satisfies $S\phi=\phi S$,
and $M$ is an open part of a tube over a totally geodesic $\mathbb{C}P^k\hookrightarrow Q^{2k}$ ($m=2k$, $k\geq2$), which has three or four distinct constant principal curvatures. We get a contradiction.
\end{proof}

\begin{theorem}\label{thm:5.3}
Let $M$ be a Hopf hypersurface of $Q^m$ ($m\geq3$) with three distinct constant principal curvatures.
Then,
\begin{enumerate}
\item[(1)]
$M$ is an open part of a tube over a totally geodesic $Q^{m-1}\hookrightarrow Q^{m}$; or

\item[(2)]
$M$ is an open part of a tube
of radius $r=\frac{\pi}{4}$ over a totally geodesic $\mathbb{C}P^k\hookrightarrow Q^{2k}$ ($m=2k$, $k\geq2$).
\end{enumerate}
\end{theorem}
\begin{proof}
If $M$ has $\mathfrak{A}$-principal unit normal vector field $N$, then $M$ is an open part of a tube over a totally geodesic $Q^{m-1}\hookrightarrow Q^{m}$ ($m\geq3$), which has three distinct principal curvatures.

If $M$ has $\mathfrak{A}$-isotropic unit normal vector field $N$, then by Lemma \ref{lemma:2.6}, we have $S\xi=\alpha\xi$, and $SAN=SA\xi=0$. In the following, we divide the discussions into two cases depending on the value of the constant $\alpha$.

{\bf Case-i}: $\alpha>0$.

{\bf Case-ii}: $\alpha=0$.

In {\bf Case-i}, $\alpha$ and $0$ are two distinct principal curvatures on $M$.
In the following, we further divide the discussions into four subcases depending on whether $\alpha$ or $0$
belongs in $\sigma(\mathcal{Q})$.

{\bf Case-i-1}: $\alpha\in\sigma(\mathcal{Q})$ and $0\notin\sigma(\mathcal{Q})$.

{\bf Case-i-2}: $\alpha\notin\sigma(\mathcal{Q})$ and $0\in\sigma(\mathcal{Q})$.

{\bf Case-i-3}: $\alpha,0\in\sigma(\mathcal{Q})$.

{\bf Case-i-4}: $\alpha,0\notin\sigma(\mathcal{Q})$.

In {\bf Case-i-1}, then there exists a tangent vector field $X\in\mathcal{Q}$, such that
$SX=\alpha X$. Now, by Lemma \ref{lemma:5.5}, it follows that $S\phi X=\mu\phi X$ and $\mu=\tfrac{\alpha^2+2}{\alpha}$, in which $\tfrac{\alpha^2+2}{\alpha}$ is different from $\alpha$ and $0$. Due to $0\notin\sigma(\mathcal{Q})$, it follows that $\mathcal{Q}=V_\alpha\oplus V_\mu$ and
$JV_\alpha=V_\mu$. Now, by (3) of Lemma \ref{lemma:5.1}, we have
$AV_\alpha\bot{\rm Span}\{V_\alpha,JV_\alpha\}$, i.e. $AV_\alpha\bot\mathcal{Q}$, which contradicts with $AV_\alpha\subset\mathcal{Q}$.

In {\bf Case-i-2}, then there exists a tangent vector field $X\in\mathcal{Q}$, such that
$SX=0$. Now, by Lemma \ref{lemma:5.5}, it follows that $S\phi X=\mu\phi X$ and $\mu=\tfrac{-2}{\alpha}$, in which $\tfrac{-2}{\alpha}$ is different from $\alpha$ and $0$. Due to $\alpha\notin\sigma(\mathcal{Q})$, it follows that $\mathcal{Q}=V_0\oplus V_\mu$ and
$JV_0=V_\mu$. Now, by (3) of Lemma \ref{lemma:5.1}, we have
$AV_0\bot{\rm Span}\{V_0,JV_0\}$, i.e. $AV_0\bot\mathcal{Q}$, which contradicts with $AV_0\subset\mathcal{Q}$.

In {\bf Case-i-3}, then there exist tangent vector fields $X,Y\in\mathcal{Q}$, such that
$SX=\alpha X$ and $SY=0$. Now, by Lemma \ref{lemma:5.5}, it follows that $S\phi X=\mu\phi X$ and
$S\phi Y=\nu\phi Y$, where $\mu=\tfrac{\alpha^2+2}{\alpha}$ and $\nu=\frac{-2}{\alpha}$. Due that $\{\tfrac{\alpha^2+2}{\alpha}, \frac{-2}{\alpha}, \alpha, 0\}$ are four distinct constants,
then it contradicts with the assumption that $M$ has three distinct principal curvatures.

In {\bf Case-i-4}, we know that there is another one principal curvature $\lambda\notin\{\alpha,0\}$ in $\sigma(\mathcal{Q})$. By $\alpha,0\notin\sigma(\mathcal{Q})$, it follows that $SX=\lambda X$ for all $X\in\mathcal{Q}$. So, $\mathcal{Q}=V_\lambda$.
Now, by (3) of Lemma \ref{lemma:5.1}, we have
$AV_\lambda\bot{\rm Span}\{V_\lambda,JV_\lambda\}$, which shows that
$AV_\lambda\bot\mathcal{Q}$. It contradicts with $AV_\lambda\subset\mathcal{Q}$.

In {\bf Case-ii}, it holds $\alpha=0$. If $0\in\sigma(\mathcal{Q})$, then there exists a tangent vector field $X\in\mathcal{Q}$ such that $SX=0$. Now, by Lemma \ref{lemma:5.5}, it follows that $S\phi X=\frac{-2}{\alpha}\phi X$, which contradicts with $\alpha=0$. So, $0\notin\sigma(\mathcal{Q})$.
Then, there are another two distinct principal curvatures $\lambda,\mu$ such that  $\{\lambda,\mu\}=\sigma(\mathcal{Q})$ and $\lambda,\mu\neq0$.

For the principal curvatures $\lambda$ and $\mu$, without
loss of generality, if $\lambda\notin\{1,-1\}$, then by Lemma \ref{lemma:5.5},
$\tfrac{1}{\lambda}$ is different from $0$ and $\lambda$, and it holds $\tfrac{1}{\lambda}\in\sigma(\mathcal{Q})$.
It deduces that $\mu=\tfrac{1}{\lambda}$. Again by Lemma \ref{lemma:5.5}, we have
$\phi V_\lambda=V_\mu$ and $\mathcal{Q}=V_\lambda\oplus V_\mu$.
On the other hand, by (3) of Lemma \ref{lemma:5.1}, we have
$AV_\lambda\bot{\rm Span}\{V_\lambda,JV_\lambda\}$, which contradicts with $AV_\lambda\subset\mathcal{Q}$. Thus, it must hold $\{1,-1\}=\sigma(\mathcal{Q})$.
By Lemma \ref{lemma:5.5}, it follows that $\phi V_1=V_1$, $\phi V_{-1}=V_{-1}$
and $\mathcal{Q}=V_1\oplus V_{-1}$, which deduces that
$M$ satisfies $S\phi=\phi S$. Then by Theorem \ref{thm:2.2} and the principal curvatures of Example \ref{E2},
$M$ is an open part of a tube of radius $r=\frac{\pi}{4}$ over a totally geodesic $\mathbb{C}P^k\hookrightarrow Q^{2k}$, $k\geq2$.
\end{proof}

%==================================================
\subsection{Hopf hypersurfaces of $Q^m$ with four distinct constant principal curvatures}\label{sect:5.2}

\begin{theorem}\label{thm:5.4}
Let $M$ be a Hopf hypersurface of $Q^m$ ($m\geq3$) with four distinct constant principal curvatures.
Then, $M$ is an open part of a tube of radius $r\in(0,\frac{\pi}{4})\cup(\frac{\pi}{4},\frac{\pi}{2})$ over a totally geodesic $\mathbb{C}P^k\hookrightarrow Q^{2k}$ ($m=2k$, $k\geq2$).
\end{theorem}
\begin{proof}
If $M$ has $\mathfrak{A}$-principal unit normal vector field $N$, then $M$ is an open part of a tube over a totally geodesic $Q^{m-1}\hookrightarrow Q^{m}$ ($m\geq3$), which has three distinct principal curvatures. We get a contradiction.

If $M$ has $\mathfrak{A}$-isotropic unit normal vector field $N$, then by Lemma \ref{lemma:2.6}, we have $S\xi=\alpha\xi$, and $SAN=SA\xi=0$. In the following, we divide the discussions into two cases depending on the value of the constant $\alpha$.

{\bf Case-i}: $\alpha>0$.

{\bf Case-ii}: $\alpha=0$.

In {\bf Case-i}, $\alpha$ and $0$ are two distinct principal curvatures on $M$.
In the following, we further divide the discussions into four subcases depending on whether $\alpha$ or $0$
belongs in $\sigma(\mathcal{Q})$.

{\bf Case-i-1}: $\alpha\in\sigma(\mathcal{Q})$ and $0\notin\sigma(\mathcal{Q})$.

{\bf Case-i-2}: $\alpha\notin\sigma(\mathcal{Q})$ and $0\in\sigma(\mathcal{Q})$.

{\bf Case-i-3}: $\alpha,0\in\sigma(\mathcal{Q})$.

{\bf Case-i-4}: $\alpha,0\notin\sigma(\mathcal{Q})$.

In {\bf Case-i-1}, there exists a tangent vector field $X\in\mathcal{Q}$, such that
$SX=\alpha X$. Now, by Lemma \ref{lemma:5.5}, it follows that $S\phi X=\mu\phi X$ and $\mu=\tfrac{\alpha^2+2}{\alpha}$, in which $\tfrac{\alpha^2+2}{\alpha}$ is different from $\alpha$ and $0$. By the assumption that $M$ has four distinct principal curvatures, and combining Lemma \ref{lemma:5.5}, we must have that the fourth principal curvature $\nu$ satisfies $\nu=\tfrac{\alpha\nu+2}{2\nu-\alpha}$. Otherwise, by Lemma \ref{lemma:5.5}, there will be
at least five distinct principal curvatures on $M$. Note that, the roots of the equation $x=\tfrac{\alpha x+2}{2x-\alpha}$ are
$x_1=\frac{\alpha-\sqrt{\alpha^2+4}}{2}$ and $x_2=\frac{\alpha+\sqrt{\alpha^2+4}}{2}$.

Now, we have
$$
\begin{aligned}
\mathcal{Q}=V_\alpha\oplus V_\mu\oplus V_\nu, \ \ JV_\alpha=V_\mu,\ \ JV_\nu=V_\nu.
\end{aligned}
$$
Then, by (3) of Lemma \ref{lemma:5.1}, we have
$$
A(V_\alpha\oplus V_\mu)\bot(V_\alpha\oplus V_\mu),\ \
AV_\nu\bot V_\nu.
$$
It deduces that
\begin{equation*}
{\rm dim}V_\nu=2{\rm dim}V_\alpha=2{\rm dim}V_\mu=m-2, \ \
A(V_\alpha\oplus V_\mu)=V_\nu.
\end{equation*}
We take a unit vector field $X_1\in V_{\alpha}$, so $AX_1\in V_{\nu}$.
%Now, we can chose a local orthonormal frame field $\{e_i\}_{i=1}^{2m-4}$ of $\mathcal{Q}$, such that
%$$
%\begin{aligned}
%&{\rm Span}\{e_1,e_2,...,e_{\frac{m-2}{2}}\}=V_\alpha,\
%{\rm Span}\{e_{\frac{m-2}{2}+1}=Je_1,e_{\frac{m-2}{2}+2}=Je_2,...,e_{m-2}=Je_{\frac{m-2}{2}}\}=V_\mu,\\
%&{\rm Span}\{e_{m-1}=Ae_1,e_m=Ae_2,...,e_{\frac{3(m-2)}{2}}=Ae_{\frac{m-2}{2}},
%e_{\frac{3(m-2)}{2}+1}=Ae_{\frac{m-2}{2}+1},...,e_{2m-4}=Ae_{m-2}\}=V_\nu.
%\end{aligned}
%$$
%Without loss of generality, we can assume that $\alpha>0$.
By $\alpha>0$, it follows that
$$
\frac{\alpha-\sqrt{\alpha^2+4}}{2}<0<\alpha<\frac{\alpha+\sqrt{\alpha^2+4}}{2}<\frac{\alpha^2+2}{\alpha}.
$$

If $\nu=\frac{\alpha-\sqrt{\alpha^2+4}}{2}$, then $\nu<0<\alpha<\mu$.
Now, we take $(X,Y)=(X_1,AX_1)$ into \eqref{eqn:call}, it follows
%
%\begin{equation}\label{eqn:ca1}
%2\sum_{\mu_i\neq\lambda,\mu}\frac{g((\nabla_{e_i} S)X,Y)^2}{(\lambda-\mu_i)(\mu-\mu_i)}
%=(1+\lambda\mu)\Big\{1+2g(\phi X,Y)^2-2g(AX,Y)^2-2g(AX,JY)^2\Big\}.
%\end{equation}
%$$
\begin{equation}\label{eqn:5.13a1}
\begin{aligned}
2\sum_{\mu_i=\mu}\frac{g((\nabla_{e_i} S)X_1,AX_1)^2}{(\alpha-\mu)(\nu-\mu)}
%2\sum_{i=\frac{m-2}{2}+1}^{m-2}\frac{g((\nabla_{e_i} S)e_1,e_{m-1})^2}{(\alpha-\mu)(\nu-\mu)}
&=(1+\alpha\nu)\Big\{1+2g(\phi X_1,AX_1)^2-2g(AX_1,AX_1)^2-2g(AX_1,JAX_1)^2\Big\}\\
&=(1+\alpha\nu)(1-2)=-(1+\alpha\nu).
\end{aligned}
\end{equation}
Note that, by $\nu<0<\alpha<\mu$, the left hand side of \eqref{eqn:5.13a1} is non-negative.
However, by $\nu=\frac{\alpha-\sqrt{\alpha^2+4}}{2}$, we can have $1+\alpha\nu>0$, which deduces that
the right hand side of \eqref{eqn:5.13a1} is negative. We get a contradiction.

%with the use of $A(V_\alpha\oplus V_\mu)=V_\nu$ and
%${\rm dim}V_\nu=2{\rm dim}V_\alpha=2{\rm dim}V_\mu=m-2$, we have that
%the left hand side of $$

If $\nu=\frac{\alpha+\sqrt{\alpha^2+4}}{2}$, then $0<\alpha<\nu<\mu$.
Now, we take $(X,Y)=(X_1,AX_1)$ into \eqref{eqn:call}, it follows
%
%\begin{equation}\label{eqn:ca1}
%2\sum_{\mu_i\neq\lambda,\mu}\frac{g((\nabla_{e_i} S)X,Y)^2}{(\lambda-\mu_i)(\mu-\mu_i)}
%=(1+\lambda\mu)\Big\{1+2g(\phi X,Y)^2-2g(AX,Y)^2-2g(AX,JY)^2\Big\}.
%\end{equation}
%$$
\begin{equation}\label{eqn:5.14}
\begin{aligned}
2\sum_{\mu_i=\mu}\frac{g((\nabla_{e_i} S)X_1,AX_1)^2}{(\alpha-\mu)(\nu-\mu)}
&=(1+\alpha\nu)\Big\{1+2g(\phi X_1,AX_1)^2-2g(AX_1,AX_1)^2-2g(AX_1,JAX_1)^2\Big\}\\
&=(1+\alpha\nu)(1-2)=-(1+\alpha\nu).
\end{aligned}
\end{equation}
Note that, by $0<\alpha<\nu<\mu$, the left hand side of \eqref{eqn:5.14} is non-negative.
However, the right hand side of \eqref{eqn:5.14} is negative, which is also a contradiction.

In {\bf Case-i-2}, there exists a tangent vector field $X\in\mathcal{Q}$, such that
$SX=0$. Now, by Lemma \ref{lemma:5.5}, it follows that $S\phi X=\mu\phi X$ and $\mu=\tfrac{-2}{\alpha}$, in which $\tfrac{-2}{\alpha}$ is different from $\alpha$ and $0$. By the assumption that $M$ has four distinct principal curvatures, and combining Lemma \ref{lemma:5.5}, we must have that the fourth principal curvature $\nu$ satisfies $\nu=\tfrac{\alpha\nu+2}{2\nu-\alpha}$. Otherwise, by Lemma \ref{lemma:5.5}, there will be at least five distinct principal curvatures on $M$. %Note that, the roots of equation $x=\tfrac{\alpha x+2}{2x-\alpha}$ are $x_1=\frac{\alpha-\sqrt{\alpha^2+4}}{2}$ and $x_2=\frac{\alpha+\sqrt{\alpha^2+4}}{2}$.
Now, we have
$\mathcal{Q}=V_0\oplus V_\mu\oplus V_\nu$, $\phi V_0=V_\mu$, $\phi V_\nu=V_\nu$.
Then, by (3) of Lemma \ref{lemma:5.1}, we have
$A(V_0\oplus V_\mu)\bot(V_0\oplus V_\mu)$, $AV_\nu\bot V_\nu$.
It deduces that
\begin{equation*}
{\rm dim}V_\nu=2{\rm dim}V_0=2{\rm dim}V_\mu=m-2, \ \
A(V_0\oplus V_\mu)=V_\nu.
\end{equation*}
We take a unit vector field $X_1\in V_0$, so $AX_1\in V_{\nu}$.
%Now, we can chose a local orthonormal frame field $\{e_i\}_{i=1}^{2m-4}$ of $\mathcal{Q}$, such that
%$$
%\begin{aligned}
%&{\rm Span}\{e_1,e_2,...,e_{\frac{m-2}{2}}\}=V_0,\
%{\rm Span}\{e_{\frac{m-2}{2}+1}=Je_1,e_{\frac{m-2}{2}+2}=Je_2,...,e_{m-2}=Je_{\frac{m-2}{2}}\}=V_\mu,\\
%&{\rm Span}\{e_{m-1}=Ae_1,e_m=Ae_2,...,e_{\frac{3(m-2)}{2}}=Ae_{\frac{m-2}{2}},
%e_{\frac{3(m-2)}{2}+1}=Ae_{\frac{m-2}{2}+1},...,e_{2m-4}=Ae_{m-2}\}=V_\nu.
%\end{aligned}
%$$
By $\alpha>0$, it follows that
$$
\frac{-2}{\alpha}<\frac{\alpha-\sqrt{\alpha^2+4}}{2}<0<\alpha<\frac{\alpha+\sqrt{\alpha^2+4}}{2}.
$$

If $\nu=\frac{\alpha-\sqrt{\alpha^2+4}}{2}$, then $\mu<\nu<0<\alpha$.
Now, we take $(X,Y)=(X_1,AX_1)$ into \eqref{eqn:call}, it follows
%
%\begin{equation}\label{eqn:ca1}
%2\sum_{\mu_i\neq\lambda,\mu}\frac{g((\nabla_{e_i} S)X,Y)^2}{(\lambda-\mu_i)(\mu-\mu_i)}
%=(1+\lambda\mu)\Big\{1+2g(\phi X,Y)^2-2g(AX,Y)^2-2g(AX,JY)^2\Big\}.
%\end{equation}
%$$
\begin{equation}\label{eqn:5.15}
\begin{aligned}
2\sum_{\mu_i=\mu}\frac{g((\nabla_{e_i} S)X_1,AX_1)^2}{-\mu(\nu-\mu)}
&=\Big\{1+2g(\phi X_1,AX_1)^2-2g(AX_1,AX_1)^2-2g(AX_1,JAX_1)^2\Big\}\\
&=-1.
\end{aligned}
\end{equation}
Note that, by $\mu<\nu<0<\alpha$, the left hand side of \eqref{eqn:5.15} is non-negative.
However, the right hand side of \eqref{eqn:5.15} is negative. We get a contradiction.

%with the use of $A(V_\alpha\oplus V_\mu)=V_\nu$ and
%${\rm dim}V_\nu=2{\rm dim}V_\alpha=2{\rm dim}V_\mu=m-2$, we have that
%the left hand side of $$

If $\nu=\frac{\alpha+\sqrt{\alpha^2+4}}{2}$, then $\mu<0<\alpha<\nu$.
Now, we take $(X,Y)=(X_1,AX_1)$ into \eqref{eqn:call}, it follows
%
%\begin{equation}\label{eqn:ca1}
%2\sum_{\mu_i\neq\lambda,\mu}\frac{g((\nabla_{e_i} S)X,Y)^2}{(\lambda-\mu_i)(\mu-\mu_i)}
%=(1+\lambda\mu)\Big\{1+2g(\phi X,Y)^2-2g(AX,Y)^2-2g(AX,JY)^2\Big\}.
%\end{equation}
%$$
\begin{equation}\label{eqn:5.16}
\begin{aligned}
2\sum_{\mu_i=\mu}\frac{g((\nabla_{e_i} S)X_1,AX_1)^2}{-\mu(\nu-\mu)}
&=\Big\{1+2g(\phi X_1,AX_1)^2-2g(AX_1,AX_1)^2-2g(AX_1,JAX_1)^2\Big\}\\
&=-1.
\end{aligned}
\end{equation}
By $\mu<0<\alpha<\nu$, the left hand side of \eqref{eqn:5.16} is non-negative.
However, the right hand side of \eqref{eqn:5.16} is negative, which is also a contradiction.

In {\bf Case-i-3}, there exist tangent vector fields $X,Y\in\mathcal{Q}$, such that
$SX=\alpha X$ and $SY=0$. Now, by Lemma \ref{lemma:5.5}, it follows that $S\phi X=\mu\phi X$ and
$S\phi Y=\nu\phi Y$, which $\mu=\tfrac{\alpha^2+2}{\alpha}$ and $\nu=\frac{-2}{\alpha}$. It is
obviously that $\{\tfrac{\alpha^2+2}{\alpha}, \frac{-2}{\alpha}, \alpha, 0\}$ are exactly four
distinct principal curvatures.
Now, we have $JV_\alpha=V_\mu$, $JV_0=V_\nu$ and $\mathcal{Q}=V_\alpha\oplus V_\mu\oplus V_0\oplus V_\nu$. Then, by (3) of Lemma \ref{lemma:5.1}, we have
$$
A(V_\alpha\oplus V_\mu)\bot(V_\alpha\oplus V_\mu),\ \
A(V_0\oplus V_\nu)\bot(V_0\oplus V_\nu).
$$
It deduces that
\begin{equation}\label{eqn:5.17}
{\rm dim}V_\alpha={\rm dim}V_\mu={\rm dim}V_0={\rm dim}V_\nu=\frac{m-2}{2}, \ \
A(V_\alpha\oplus V_\mu)=V_0\oplus V_\nu.
\end{equation}

Now, we choose two unit vector fields $X_1\in V_\alpha$ and $X_2\in V_0$.
So $JX_1\in V_\mu$ and $JX_2\in V_\nu$.
%
%Now, we can chose a local orthonormal frame field $\{e_i\}_{i=1}^{2m-4}$ of $\mathcal{Q}$, such that
%$$
%\begin{aligned}
%&{\rm Span}\{e_1,e_2,...,e_{\frac{m-2}{2}}\}=V_\alpha,\
%{\rm Span}\{e_{\frac{m-2}{2}+1}=Je_1,e_{\frac{m-2}{2}+2}=Je_2,...,e_{m-2}=Je_{\frac{m-2}{2}}\}=V_\mu,\\
%&{\rm Span}\{e_{m-1},e_m,...,e_{\frac{3(m-2)}{2}}\}=V_0,\
%\{e_{\frac{3(m-2)}{2}+1}=Je_{m-1},...,e_{2m-4}=Je_{\frac{3(m-2)}{2}}\}=V_\nu.
%\end{aligned}
%$$
%
In the following, we take $X=X_1,JX_1,X_2$ and $JX_2$
into \eqref{eqn:ca2} respectively, with the use of \eqref{eqn:5.17}, we can obtain
$$
\begin{aligned}
&\frac{1+\alpha\nu}{\alpha-\nu}(\frac{m-2}{2}-2)+\frac{1}{\alpha}(\frac{m-2}{2}-2)
+\frac{1+\alpha\mu}{\alpha-\mu}(\frac{m-2}{2}+2)=0,\\
&\frac{1+\mu\nu}{\mu-\nu}(\frac{m-2}{2}-2)+\frac{1}{\mu}(\frac{m-2}{2}-2)
+\frac{1+\mu\alpha}{\mu-\alpha}(\frac{m-2}{2}+2)=0,\\
&\frac{1}{-\nu}(\frac{m-2}{2}+2)+\frac{1}{-\alpha}(\frac{m-2}{2}-2)
+\frac{1}{-\mu}(\frac{m-2}{2}-2)=0,\\
&\frac{1}{\nu}(\frac{m-2}{2}+2)+\frac{1+\nu\alpha}{\nu-\alpha}(\frac{m-2}{2}-2)
+\frac{1+\nu\mu}{\nu-\mu}(\frac{m-2}{2}-2)=0.
\end{aligned}
$$

Substituting the equations $\mu=\tfrac{\alpha^2+2}{\alpha}$ and $\nu=\frac{-2}{\alpha}$ into above four  equations, we have
\begin{equation}\label{eqn:5.18}
4(m-6)+2(m-14)\alpha^2-(m+2)\alpha^4=0,
\end{equation}
\begin{equation}\label{eqn:5.19}
-4(m-6)+6(m+2)\alpha^2+5(m+2)\alpha^4+(m+2)\alpha^6=0.
\end{equation}
According to \eqref{eqn:5.18}, we have $\alpha^4-2\alpha^2-4\neq0$, then it follows
$m=-2-\frac{32(\alpha^2+1)}{\alpha^4-2\alpha^2-4}$. Now, we substitute this value of $m$ into
\eqref{eqn:5.19}, we can get
$$
32\alpha^2(\alpha^2+2)(\alpha^4+4\alpha^2+2)=0,
$$
which contradicts with $\alpha>0$.

In {\bf Case-i-4}, by $\alpha,0\notin\sigma(\mathcal{Q})$ and Lemma \ref{lemma:5.5}, we have
$\frac{-2}{\alpha},\frac{\alpha^2+2}{\alpha}\notin\sigma(\mathcal{Q})$.
Recall that the roots of the equation $x=\tfrac{\alpha x+2}{2x-\alpha}$
are $\frac{\alpha\pm\sqrt{\alpha^2+4}}{2}$. If there is a constant $\lambda\in\sigma(\mathcal{Q})$, and
$\lambda\notin\{\alpha,0,\frac{-2}{\alpha},\frac{\alpha^2+2}{\alpha},\frac{\alpha\pm\sqrt{\alpha^2+4}}{2}\}$, then by Lemma \ref{lemma:5.5}, there is a constant $\mu=\tfrac{\alpha\lambda+2}{2\lambda-\alpha}\in\sigma(\mathcal{Q})$ and $\mu\notin\{\alpha,0,\frac{-2}{\alpha},\frac{\alpha^2+2}{\alpha},\frac{\alpha\pm\sqrt{\alpha^2+4}}{2},\lambda\}$. From above observation,
we know that either $\frac{\alpha\pm\sqrt{\alpha^2+4}}{2}\in\sigma(\mathcal{Q})$ or
$\frac{\alpha\pm\sqrt{\alpha^2+4}}{2}\notin\sigma(\mathcal{Q})$. In the following, we
divide the discussions into two subcases.

{\bf Case-i-4-1}: $\frac{\alpha\pm\sqrt{\alpha^2+4}}{2}\in\sigma(\mathcal{Q})$.

{\bf Case-i-4-2}: $\frac{\alpha\pm\sqrt{\alpha^2+4}}{2}\notin\sigma(\mathcal{Q})$.

In {\bf Case-i-4-1}, we have
$$
JV_{\frac{\alpha+\sqrt{\alpha^2+4}}{2}}=V_{\frac{\alpha+\sqrt{\alpha^2+4}}{2}},\ \
JV_{\frac{\alpha-\sqrt{\alpha^2+4}}{2}}=V_{\frac{\alpha-\sqrt{\alpha^2+4}}{2}},\ \
\mathcal{Q}=V_{\frac{\alpha+\sqrt{\alpha^2+4}}{2}}\oplus V_{\frac{\alpha-\sqrt{\alpha^2+4}}{2}}.
$$
It shows that $M$ satisfies $S\phi=\phi S$, then by Theorem \ref{thm:2.2} and the
principal curvatures of Example \ref{E2}, $M$ is an open part of a tube
of radius $r\in(0,\frac{\pi}{4})\cup(\frac{\pi}{4},\frac{\pi}{2})$ over a totally geodesic $\mathbb{C}P^k\hookrightarrow Q^{2k}$, $k\geq2$.

In {\bf Case-i-4-2}, by Lemma \ref{lemma:5.5}, there are two distinct principal curvatures $\lambda,\mu\in\sigma(\mathcal{Q})$, and they satisfy
$$
\mu=\tfrac{\alpha\lambda+2}{2\lambda-\alpha},\ \ \lambda,\mu\notin\{\alpha,0,\frac{-2}{\alpha},\frac{\alpha^2+2}{\alpha},
\frac{\alpha\pm\sqrt{\alpha^2+4}}{2}\}, \ \ JV_\lambda=V_\mu, \ \
\mathcal{Q}=V_\lambda\oplus V_\mu.
$$
However, by (3) of Lemma \ref{lemma:5.1}, we have
$AV_\lambda\bot{\rm Span}\{V_\lambda,JV_\lambda\}$, which contradicts with $AV_\lambda\subset\mathcal{Q}$. %Thus, {\bf Case-i-4-2} does not occur.

In {\bf Case-ii}, it holds $\alpha=0$. If $0\in\sigma(\mathcal{Q})$, there exists a tangent vector field $X\in\mathcal{Q}$ such that $SX=0$. Now, by Lemma \ref{lemma:5.5}, it follows that $S\phi X=\frac{-2}{\alpha}\phi X$, which contradicts with $\alpha=0$. So $0\notin\sigma(\mathcal{Q})$.
Then, there are three distinct principal curvatures $\lambda,\mu,\nu$ in $\sigma(\mathcal{Q})$, which holds $\lambda,\mu,\nu\neq0$.

Note that, there is only one of $\{1,-1\}$ in $\sigma(\mathcal{Q})$. In fact, if it holds both
$1,-1\in\sigma(\mathcal{Q})$, then we can assume that $\nu\notin\{0,1,-1\}$, and by Lemma \ref{lemma:5.5},
it also holds that $\frac{1}{\nu}\in\sigma(\mathcal{Q})$, which is different from $\{0,1,-1,\nu\}$.
It contradicts with the assumption that $M$ has four distinct principal curvatures.
If it holds both $1,-1\notin\sigma(\mathcal{Q})$, then by Lemma \ref{lemma:5.5}, we can assume that
$\mu=\frac{1}{\lambda}$ and $\nu\notin\{0,1,-1,\lambda,\frac{1}{\lambda}\}$. Again by Lemma \ref{lemma:5.5},
we have $\frac{1}{\nu}\in\sigma(\mathcal{Q})$, which is different from $\{0,1,-1,\lambda,\frac{1}{\lambda},\nu\}$. It also contradicts with the assumption that $M$ has four distinct principal curvatures.

In the following, we divide the discussions into two subcases depending on
whether $1$ or $-1$ belongs in $\sigma(\mathcal{Q})$.

{\bf Case-ii-1}: $1\in\sigma(\mathcal{Q})$ and $-1\notin\sigma(\mathcal{Q})$.

{\bf Case-ii-2}: $1\notin\sigma(\mathcal{Q})$ and $-1\in\sigma(\mathcal{Q})$.

In {\bf Case-ii-1}, there exist constant principal curvatures $\lambda,\mu\in\sigma(\mathcal{Q})$, such that
$$
\mathcal{Q}=V_\lambda\oplus V_\mu\oplus V_1,\ JV_\lambda= V_\mu,\ JV_1= V_1,\ \mu=\frac{1}{\lambda},
\ \lambda<\mu,\ \lambda,\mu\notin\{0,1,-1\}.
$$
Then, by (3) of Lemma \ref{lemma:5.1}, we have
$$
A(V_\lambda\oplus V_\mu)\bot(V_\lambda\oplus V_\mu),\ \
AV_1\bot V_1.
$$
It deduces that ${\rm dim}V_1=2{\rm dim}V_\lambda=2{\rm dim}V_\mu=m-2$,
$A(V_\lambda\oplus V_\mu)=V_1$.
Now, we choose a unit vector field $X_1\in V_\lambda$.
So $JX_1\in V_\mu$ and $AX_1,AJX_1\in V_1$.

%Now, we can chose a local orthonormal frame field $\{e_i\}_{i=1}^{2m-4}$ of $\mathcal{Q}$, such that
%$$
%\begin{aligned}
%&{\rm Span}\{e_1,e_2,...,e_{\frac{m-2}{2}}\}=V_\lambda,\
%{\rm Span}\{e_{\frac{m-2}{2}+1}=Je_1,e_{\frac{m-2}{2}+2}=Je_2,...,e_{m-2}=Je_{\frac{m-2}{2}}\}=V_\mu,\\
%&{\rm Span}\{e_{m-1}=Ae_1,e_m=Ae_2,...,e_{\frac{3(m-2)}{2}}=Ae_{\frac{m-2}{2}},
%e_{\frac{3(m-2)}{2}+1}=Ae_{\frac{m-2}{2}+1},...,e_{2m-4}=Ae_{m-2}\}=V_1.
%\end{aligned}
%$$

If $0<\lambda<1<\mu$, we take $(X,Y)=(X_1,AX_1)$ into \eqref{eqn:call}, it follows
\begin{equation}\label{eqn:5.21}
\begin{aligned}
2\sum_{\mu_i=\mu}\frac{g((\nabla_{e_i} S)X_1,AX_1)^2}{(\lambda-\mu_i)(1-\mu_i)}
&=(1+\lambda)\Big\{1+2g(\phi X_1,AX_1)^2-2g(AX_1,AX_1)^2-2g(AX_1,JAX_1)^2\Big\}\\
&=(1+\lambda)(1-2)=-(1+\lambda).
\end{aligned}
\end{equation}
By $0<\lambda<1<\mu$, the left hand side of \eqref{eqn:5.21} is non-negative.
However, the right hand side of \eqref{eqn:5.21} is negative. We get a contradiction.

If $\lambda<-1<\mu<0<1$, we take $(X,Y)=(JX_1,AJX_1)$ into \eqref{eqn:call},
it follows
\begin{equation}\label{eqn:5.22}
\begin{aligned}
2\sum_{\mu_i=\lambda}\frac{g((\nabla_{e_i} S)JX_1,AJX_1)^2}{(\mu-\mu_i)(1-\mu_i)}
&=(1+\mu)\Big\{1+2g(\phi JX_1,AJX_1)^2-2g(AJX_1,AJX_1)^2-2g(AJX_1,JAJX_1)^2\Big\}\\
&=(1+\mu)(1-2)=-(1+\mu).
\end{aligned}
\end{equation}
By $\lambda<-1<\mu<0<1$, the left hand side of \eqref{eqn:5.22} is non-negative.
However, the right hand side of \eqref{eqn:5.22} is negative. We get a contradiction.

In {\bf Case-ii-2}, there exist constant principal curvatures $\lambda,\mu\in\sigma(\mathcal{Q})$, such that
$$
\mathcal{Q}=V_\lambda\oplus V_\mu\oplus V_{-1},\ JV_\lambda= V_\mu,\ JV_{-1}= V_{-1},
\ \lambda<\mu,\ \mu=\frac{1}{\lambda},\ \lambda,\mu\notin\{0,1,-1\}.
$$
Then, by Lemma \ref{lemma:5.1}, we have
$A(V_\lambda\oplus V_\mu)\bot(V_\lambda\oplus V_\mu)$, $AV_{-1}\bot V_{-1}$.
It deduces that
$$
{\rm dim}V_{-1}=2{\rm dim}V_\lambda=2{\rm dim}V_\mu=m-2, \ \
A(V_\lambda\oplus V_\mu)=V_{-1}.
$$
Now, we choose a unit vector field $X_1\in V_\lambda$.
So $JX_1\in V_\mu$ and $AX_1,AJX_1\in V_{-1}$.

%Now, we can chose a local orthonormal frame field $\{e_i\}_{i=1}^{2m-4}$ of $\mathcal{Q}$, such that
%$$
%\begin{aligned}
%&{\rm Span}\{e_1,e_2,...,e_{\frac{m-2}{2}}\}=V_\lambda,\
%{\rm Span}\{e_{\frac{m-2}{2}+1}=Je_1,e_{\frac{m-2}{2}+2}=Je_2,...,e_{m-2}=Je_{\frac{m-2}{2}}\}=V_\mu,\\
%&{\rm Span}\{e_{m-1}=Ae_1,e_m=Ae_2,...,e_{\frac{3(m-2)}{2}}=Ae_{\frac{m-2}{2}},
%e_{\frac{3(m-2)}{2}+1}=Ae_{\frac{m-2}{2}+1},...,e_{2m-4}=Ae_{m-2}\}=V_{-1}.
%\end{aligned}
%$$

If $-1<0<\lambda<1<\mu$, we take $(X,Y)=(X_1,AX_1)$ into \eqref{eqn:call}, it follows
\begin{equation}\label{eqn:5.24}
\begin{aligned}
2\sum_{\mu_i=\mu}\frac{g((\nabla_{e_i} S)X_1,AX_1)^2}{(\lambda-\mu_i)(-1-\mu_i)}
&=(1-\lambda)\Big\{1+2g(\phi X_1,AX_1)^2-2g(AX_1,AX_1)^2-2g(AX_1,JAX_1)^2\Big\}\\
&=(1-\lambda)(1-2)=-(1-\lambda).
\end{aligned}
\end{equation}
By $-1<0<\lambda<1<\mu$, the left hand side of \eqref{eqn:5.24} is non-negative.
However, the right hand side of \eqref{eqn:5.24} is negative. We get a contradiction.

If $\lambda<-1<\mu<0<1$, we take $(X,Y)=(JX_1,AJX_1)$ into \eqref{eqn:call},
it follows
\begin{equation}\label{eqn:5.25}
\begin{aligned}
2\sum_{\mu_i=\lambda}\frac{g((\nabla_{e_i} S)JX_1,AJX_1)^2}{(\mu-\mu_i)(-1-\mu_i)}
&=(1-\mu)\Big\{1+2g(\phi JX_1,AJX_1)^2-2g(AJX_1,AJX_1)^2-2g(AJX_1,JAJX_1)^2\Big\}\\
&=(1-\mu)(1-2)=-(1-\mu).
\end{aligned}
\end{equation}
By $\lambda<-1<\mu<0<1$, the left hand side of \eqref{eqn:5.25} is non-negative.
However, the right hand side of \eqref{eqn:5.25} is negative. We get a contradiction.
\end{proof}

%As a direct corollary of Theorems \ref{thm:5.1}--\ref{thm:5.4}, we have the following

%\begin{corollary}\label{cor:5.2}
%Let $M$ be a Hopf hypersurface of $Q^m$ ($m\geq3$) with at most four distinct constant principal curvatures. Then,
%
%\begin{enumerate}
%\item[(1)]
%$M$ is an open part of a tube over a totally geodesic almost complex submanifold $\mathbb{C}P^k\hookrightarrow Q^{2k}$ ($m=2k$, $k\geq2$); or

%\item[(2)]
%$M$ is an open part of a tube over a totally geodesic Lagrangian submanifold $S^{m}(1)\hookrightarrow Q^{m}$ ($m\geq3$).
%\end{enumerate}
%\end{corollary}

%==================================================
\subsection{Hopf hypersurfaces of $Q^m$ with five distinct constant principal curvatures}\label{sect:5.3}

\begin{theorem}\label{thm:5.5}
There does not exist Hopf hypersurface of $Q^m$ ($m\geq3$) with five distinct constant principal curvatures.
\end{theorem}

\begin{proof}
If $M$ has $\mathfrak{A}$-principal unit normal vector field $N$, then $M$ is an open part of a tube over a totally geodesic $Q^{m-1}\hookrightarrow Q^{m}$ ($m\geq3$), which has three distinct principal curvatures. We get a contradiction.

If $M$ has $\mathfrak{A}$-isotropic unit normal vector field $N$, then by Lemma \ref{lemma:2.6}, we have $S\xi=\alpha\xi$, and $SAN=SA\xi=0$. In the following, we divide the discussions into two cases depending on the value of the constant $\alpha$.

{\bf Case-i}: $\alpha>0$.

{\bf Case-ii}: $\alpha=0$.

In {\bf Case-i}, $\alpha$ and $0$ are two distinct principal curvatures on $M$.
We further divide the discussions into four subcases depending on whether $\alpha$ or $0$
belongs in $\sigma(\mathcal{Q})$.

{\bf Case-i-1}: $\alpha\in\sigma(\mathcal{Q})$ and $0\notin\sigma(\mathcal{Q})$.

{\bf Case-i-2}: $\alpha\notin\sigma(\mathcal{Q})$ and $0\in\sigma(\mathcal{Q})$.

{\bf Case-i-3}: $\alpha,0\in\sigma(\mathcal{Q})$.

{\bf Case-i-4}: $\alpha,0\notin\sigma(\mathcal{Q})$.

In {\bf Case-i-1}, there exists a tangent vector field $X\in\mathcal{Q}$, such that
$SX=\alpha X$. Now, by Lemma \ref{lemma:5.5}, it follows that $S\phi X=\mu\phi X$ and $\mu=\tfrac{\alpha^2+2}{\alpha}$, in which $\tfrac{\alpha^2+2}{\alpha}$ is different from $\alpha$ and $0$. By the assumption that $M$ has five distinct principal curvatures, and combining Lemma \ref{lemma:5.5}, we must have the following two subcases:

{\bf Case-i-1-1}: $\frac{\alpha\pm\sqrt{\alpha^2+4}}{2}\in\sigma(\mathcal{Q})$.

{\bf Case-i-1-2}: $\frac{\alpha\pm\sqrt{\alpha^2+4}}{2}\notin\sigma(\mathcal{Q})$.

In {\bf Case-i-1-1}, we have
$$
\begin{aligned}
&\mathcal{Q}=V_\alpha\oplus V_{\tfrac{\alpha^2+2}{\alpha}}\oplus V_{\frac{\alpha+\sqrt{\alpha^2+4}}{2}}\oplus V_{\frac{\alpha-\sqrt{\alpha^2+4}}{2}},\\
&JV_\alpha=V_{\tfrac{\alpha^2+2}{\alpha}},\
JV_{\frac{\alpha+\sqrt{\alpha^2+4}}{2}}=V_{\frac{\alpha+\sqrt{\alpha^2+4}}{2}},\
JV_{\frac{\alpha-\sqrt{\alpha^2+4}}{2}}=V_{\frac{\alpha-\sqrt{\alpha^2+4}}{2}}.
\end{aligned}
$$
For convenience, let
\begin{equation}\label{eqn:5.26}
\lambda_1=\alpha,\ \lambda_2=\tfrac{\alpha^2+2}{\alpha},\ \lambda_3=\frac{\alpha-\sqrt{\alpha^2+4}}{2},\ \lambda_4=\frac{\alpha+\sqrt{\alpha^2+4}}{2}.
\end{equation}
Then, it holds $\lambda_3<0<\lambda_1<\lambda_4<\lambda_2$.
Now, it holds $\lambda_+=\tfrac{\alpha^2+2}{\alpha}$ on $M$. Then, we can use the fact that
$M_+$ is austere to prove that {\bf Case-i-1-1} does not occur.
In fact, by Proposition \ref{prop:4.3w} and \eqref{eqn:5.26}, the principal curvatures on $M_+$ are
%$$
%\tilde{\lambda}_4=\frac{-1}{\alpha},\ \tilde{\lambda}_3=\frac{\alpha}{2+\alpha^2},\
%\tilde{\lambda}_1=\frac{\alpha(3+\alpha^2)}{2},\
%\tilde{\lambda}_5=\frac{\sqrt{4+\alpha^2}+\alpha(3+\alpha(\alpha+\sqrt{4+\alpha^2}))}{2}, %\tilde{\alpha}=\alpha(3+\alpha^2),\ \ 0.
%$$
$$
\frac{-\sqrt{\alpha^2+4}+\alpha(3+\alpha^2-\alpha\sqrt{\alpha^2+4})}{2},\ \frac{\alpha(3+\alpha^2)}{2},\
\frac{\sqrt{\alpha^2+4}+\alpha(3+\alpha^2+\alpha\sqrt{\alpha^2+4})}{2},\
\alpha(3+\alpha^2),\ \ 0.
$$
There are three positive principal curvatures and one negative principal curvature,
which contradicts with the fact that $M_+$ is austere.

In {\bf Case-i-1-2}, we have
$$
\begin{aligned}
&\mathcal{Q}=V_\alpha\oplus V_{\tfrac{\alpha^2+2}{\alpha}}\oplus V_{\lambda}\oplus V_{\mu},
\ \lambda\notin\{0,\alpha,-\frac{2}{\alpha},\tfrac{\alpha^2+2}{\alpha},
\frac{\alpha\pm\sqrt{\alpha^2+4}}{2}\},\\
&JV_\alpha=V_{\tfrac{\alpha^2+2}{\alpha}},\
JV_{\lambda}=V_{\mu}, \ \mu=\frac{\alpha\lambda+2}{2\lambda-\alpha}.
\end{aligned}
$$
For convenience, let
\begin{equation}\label{eqn:5.33}
\lambda_1=\alpha,\ \lambda_2=\tfrac{\alpha^2+2}{\alpha},\ \lambda_3=\lambda,\ \lambda_4=\mu.
\end{equation}
In the following, without loss of generality, we divide the discussion into four subcases:

{\bf Case-i-1-2-1}: $\lambda_3<\lambda_4<0<\lambda_1<\lambda_2$;

{\bf Case-i-1-2-2}: $\lambda_3<0<\lambda_4<\lambda_1<\lambda_2$;

{\bf Case-i-1-2-3}: $0<\lambda_1<\lambda_3<\lambda_4<\lambda_2$;

{\bf Case-i-1-2-4}: $0<\lambda_3<\lambda_1<\lambda_2<\lambda_4$.

In {\bf Case-i-1-2-1}, we have $\lambda_+=\tfrac{\alpha^2+2}{\alpha}$. Then by
Proposition \ref{prop:4.3w} and \eqref{eqn:5.33},
on $M_+$, the principal curvatures are
\begin{equation}\label{eqn:Case-1}
\tilde{\alpha}=\alpha+\tfrac{(4+\alpha^2){\lambda_+}}{\lambda_+^2-\alpha{\lambda_+}-1}, \ \
\tilde{\lambda}_3<\tilde{\lambda}_4<\tilde{\lambda}_1,\ \  {\rm where}
\tilde{\lambda}_i=\tfrac{1+{\lambda_1}\lambda_i}{{\lambda_1}-\lambda_i}.
\end{equation}
From the fact that $M_+$ is austere, and by $\tilde{\alpha}>0$, it holds that
either $\tilde{\lambda}_3<\tilde{\lambda}_4<\tilde{\alpha}<\tilde{\lambda}_1$
or $\tilde{\lambda}_3<\tilde{\lambda}_4<\tilde{\lambda}_1<\tilde{\alpha}$.

If $\tilde{\lambda}_3<\tilde{\lambda}_4<\tilde{\alpha}<\tilde{\lambda}_1$, then
it holds that $\tilde{\lambda}_3+\tilde{\lambda}_1=\tilde{\lambda}_4+\tilde{\alpha}=0$.
By using \eqref{eqn:Case-1}, \eqref{eqn:5.33} and $\lambda<0$, we have
\begin{equation}\label{eqn:Case-2}
\alpha(8+5\alpha^2+\alpha^4)-(\alpha^4+\alpha^2-4)\lambda_3=0.
\end{equation}
\begin{equation}\label{eqn:Case-3}
\alpha(\alpha^2+4)+\frac{\alpha^2+1}{\lambda_3-\alpha}=0.
\end{equation}
By \eqref{eqn:Case-3}, we get $\lambda_3=\frac{\alpha^4+3\alpha^2-1}{\alpha(\alpha^2+4)}$.
Substituting it into \eqref{eqn:Case-2}, it follows that $5\alpha^6+30\alpha^4+45\alpha^2-4=0$,
which deduces that $\alpha=\sqrt{\sqrt[3]{\frac{7+2\sqrt{6}}{5}}+\sqrt[3]{\frac{7-2\sqrt{6}}{5}}-2}$.
Then, in this case, we get $\tilde{\lambda}_3=\frac{-1}{\sqrt{5}}$ and
$\tilde{\lambda}_4=\frac{-2}{\sqrt{5}}$, which contradicts with $\tilde{\lambda}_3<\tilde{\lambda}_4$.

If $\tilde{\lambda}_3<\tilde{\lambda}_4<\tilde{\lambda}_1<\tilde{\alpha}$, then
it holds that $\tilde{\lambda}_3+\tilde{\alpha}=\tilde{\lambda}_4+\tilde{\lambda}_1=0$.
By using \eqref{eqn:Case-1}, \eqref{eqn:5.33} and $\lambda<0$, we have
\begin{equation}\label{eqn:Case-4}
2\lambda_3+\alpha(7+5\alpha^2+\alpha^4-\alpha(\alpha^2+2)\lambda_3)=0.
\end{equation}
\begin{equation}\label{eqn:Case-5}
\frac{\alpha}{2}(\alpha^2+5)+\frac{\alpha^2+1}{\lambda_3-\alpha}=0.
\end{equation}
By \eqref{eqn:Case-5}, we get $\lambda_3=\frac{\alpha^4+3\alpha^2-2}{\alpha(\alpha^2+5)}$.
Substituting it into \eqref{eqn:Case-4}, it still follows that $5\alpha^6+30\alpha^4+45\alpha^2-4=0$,
which deduces that $\alpha=\sqrt{\sqrt[3]{\frac{7+2\sqrt{6}}{5}}+\sqrt[3]{\frac{7-2\sqrt{6}}{5}}-2}$.
Now, we need to consider $M_-$. By Proposition \ref{prop:4.4w}, on $M_-$, the curvatures are
$$
\frac{1+{\lambda_4}\lambda_3}{{\lambda_3}-\lambda_4}=\frac{-7}{\sqrt{5}},\
\frac{1+{\lambda_1}\lambda_3}{{\lambda_3}-\lambda_1}=\frac{-1}{\sqrt{5}},\
\frac{1+{\lambda_2}\lambda_3}{{\lambda_3}-\lambda_2}=\frac{2}{\sqrt{5}},\
\alpha+\frac{(4+\alpha^2){\lambda_3}}{\lambda_3^2-\alpha{\lambda_3}-1}=\frac{-14}{\sqrt{5}}.
$$
So it contradicts with the fact that $M_-$ is austere.

In {\bf Case-i-1-2-2}, we still have $\lambda_+=\tfrac{\alpha^2+2}{\alpha}$. Then by
Proposition \ref{prop:4.3w} and \eqref{eqn:5.33},
on $M_+$, the principal curvatures are still given by \eqref{eqn:Case-1}.
But in this case, it holds that $\tilde{\alpha},\tilde{\lambda}_1,\tilde{\lambda}_4>0$, i.e.
there are at least two distinct positive principal curvatures and at most one negative
principal curvatures on $M_+$. It contradicts with the fact that $M_+$ is austere.

In {\bf Case-i-1-2-3}, we have $\lambda_+=\lambda_2$. Then by
Proposition \ref{prop:4.3w} and \eqref{eqn:5.33},
there are at least three distinct positive principal curvatures, and there are no negative
principal curvatures on $M_+$. It contradicts with the fact that $M_+$ is austere.

In {\bf Case-i-1-2-4}, we have $\lambda_+=\lambda_4$. Then by
Proposition \ref{prop:4.3w} and \eqref{eqn:5.33},
there are at least three distinct positive principal curvatures, and there are no negative
principal curvatures on $M_+$. It contradicts with the fact that $M_+$ is austere.

In {\bf Case-i-2}, there exists a tangent vector field $X\in\mathcal{Q}$, such that
$SX=0$. Now, by Lemma \ref{lemma:5.5}, it follows that $S\phi X=\tfrac{-2}{\alpha}\phi X$, in which $\tfrac{-2}{\alpha}$ is different from $\alpha$ and $0$. By the assumption that $M$ has five distinct principal curvatures, and combining Lemma \ref{lemma:5.5}, we must have the following two subcases:

{\bf Case-i-2-1}: $\frac{\alpha\pm\sqrt{\alpha^2+4}}{2}\in\sigma(\mathcal{Q})$.

{\bf Case-i-2-2}: $\frac{\alpha\pm\sqrt{\alpha^2+4}}{2}\notin\sigma(\mathcal{Q})$.

In {\bf Case-i-2-1}, we have
$$
\begin{aligned}
&\mathcal{Q}=V_0\oplus V_{\tfrac{-2}{\alpha}}\oplus V_{\frac{\alpha+\sqrt{\alpha^2+4}}{2}}\oplus V_{\frac{\alpha-\sqrt{\alpha^2+4}}{2}},\\
&JV_0=V_{\tfrac{-2}{\alpha}},\
JV_{\frac{\alpha+\sqrt{\alpha^2+4}}{2}}=V_{\frac{\alpha+\sqrt{\alpha^2+4}}{2}},\
JV_{\frac{\alpha-\sqrt{\alpha^2+4}}{2}}=V_{\frac{\alpha-\sqrt{\alpha^2+4}}{2}}.
\end{aligned}
$$
For convenience, let
\begin{equation}\label{eqn:5.36}
\lambda_1=0,\ \lambda_2=\tfrac{-2}{\alpha},\ \lambda_3=\frac{\alpha-\sqrt{\alpha^2+4}}{2},\ \lambda_4=\frac{\alpha+\sqrt{\alpha^2+4}}{2}.
\end{equation}
Then, it holds $\lambda_2<\lambda_3<\lambda_1<\lambda_4$.
So $\lambda_-=\tfrac{-2}{\alpha}$ on $M$. Then, we can use the fact that
$M_-$ is austere to prove that {\bf Case-i-2-1} does not occur. In fact, by Proposition \ref{prop:4.4w} and \eqref{eqn:5.36}, the principal curvatures on $M_-$ are
%$$
%\tilde{\lambda}_4=\frac{-1}{\alpha},\ \tilde{\lambda}_3=\frac{\alpha}{2+\alpha^2},\
%\tilde{\lambda}_1=\frac{\alpha(3+\alpha^2)}{2},\
%\tilde{\lambda}_5=\frac{\sqrt{4+\alpha^2}+\alpha(3+\alpha(\alpha+\sqrt{4+\alpha^2}))}{2}, %\tilde{\alpha}=\alpha(3+\alpha^2),\ \ 0.
%$$
$$
\frac{-\alpha-\sqrt{4+\alpha^2}}{2},\ \frac{-\alpha+\sqrt{4+\alpha^2}}{2},\
\frac{-\alpha}{2},\ -\alpha,\ \ 0.
$$
There are one positive principal curvature and three negative principal curvatures on $M_-$,
which contradicts with the fact that $M_-$ is austere.

In {\bf Case-i-2-2}, we have
$$
\begin{aligned}
&\mathcal{Q}=V_0\oplus V_{\tfrac{-2}{\alpha}}\oplus V_{\lambda}\oplus V_{\mu},\ \ \mu=\tfrac{\alpha\lambda+2}{2\lambda-\alpha},
\ \lambda\notin\{0,\alpha,\tfrac{\alpha^2+2}{\alpha},
\tfrac{-2}{\alpha},\tfrac{\alpha\pm\sqrt{\alpha^2+4}}{2}\},\\
&JV_0=V_{\tfrac{-2}{\alpha}},\
JV_{\lambda}=V_{\mu}.
\end{aligned}
$$
For convenience, let
\begin{equation}\label{eqn:5.43}
\lambda_1=0,\ \lambda_2=\tfrac{-2}{\alpha},\ \lambda_3=\lambda,\ \lambda_4=\mu.
\end{equation}
By (3) of Lemma \ref{lemma:5.1}, we have $A(V_{\lambda_1}\oplus V_{\lambda_2})=V_{\lambda_3}\oplus V_{\lambda_4}$.
It holds ${\rm dim}V_{\lambda_1}={\rm dim}V_{\lambda_2}={\rm dim}V_{\lambda_3}={\rm dim}V_{\lambda_4}=\frac{m-2}{2}$.

In the following, by taking any unit vector field $X\in V_{\lambda_i}$ ($1\leq i\leq4$)
into \eqref{eqn:ca2} respectively, we obtain
\begin{gather}
\frac{1+\lambda_1\lambda_2}{\lambda_1-\lambda_2}(\frac{m-2}{2}+2)
+\frac{1+\lambda_1\lambda_3}{\lambda_1-\lambda_3}(\frac{m-2}{2}-2)
+\frac{1+\lambda_1\lambda_4}{\lambda_1-\lambda_4}(\frac{m-2}{2}-2)=0,\label{eqn:xzs1}\\
\frac{1+\lambda_2\lambda_1}{\lambda_2-\lambda_1}(\frac{m-2}{2}+2)
+\frac{1+\lambda_2\lambda_3}{\lambda_2-\lambda_3}(\frac{m-2}{2}-2)
+\frac{1+\lambda_2\lambda_4}{\lambda_2-\lambda_4}(\frac{m-2}{2}-2)=0,\label{eqn:xzs2}\\
\frac{1+\lambda_3\lambda_1}{\lambda_3-\lambda_1}(\frac{m-2}{2}-2)
+\frac{1+\lambda_3\lambda_2}{\lambda_3-\lambda_2}(\frac{m-2}{2}-2)
+\frac{1+\lambda_3\lambda_4}{\lambda_3-\lambda_4}(\frac{m-2}{2}+2)=0,\label{eqn:xzs3}\\
\frac{1+\lambda_4\lambda_1}{\lambda_4-\lambda_1}(\frac{m-2}{2}-2)
+\frac{1+\lambda_4\lambda_2}{\lambda_4-\lambda_2}(\frac{m-2}{2}-2)
+\frac{1+\lambda_4\lambda_3}{\lambda_4-\lambda_3}(\frac{m-2}{2}+2)=0,\label{eqn:xzs4}
\end{gather}
where $\lambda_1,\lambda_2,\lambda_3$ and $\lambda_4$ are given by \eqref{eqn:5.43}.
%$$
%\begin{aligned}
%&\frac{1+\lambda_1\lambda_2}{\lambda_1-\lambda_2}(\frac{m-2}{2}+2)
%+\frac{1+\lambda_1\lambda_3}{\lambda_1-\lambda_3}(\frac{m-2}{2}-2)
%+\frac{1+\lambda_1\lambda_4}{\lambda_1-\lambda_4}(\frac{m-2}{2}-2)=0,\\
%&\frac{1+\lambda_2\lambda_1}{\lambda_2-\lambda_1}(\frac{m-2}{2}+2)
%+\frac{1+\lambda_2\lambda_3}{\lambda_2-\lambda_3}(\frac{m-2}{2}-2)
%+\frac{1+\lambda_2\lambda_4}{\lambda_2-\lambda_4}(\frac{m-2}{2}-2)=0,\\
%&\frac{1+\lambda_3\lambda_1}{\lambda_3-\lambda_1}(\frac{m-2}{2}-2)
%+\frac{1+\lambda_3\lambda_2}{\lambda_3-\lambda_2}(\frac{m-2}{2}-2)
%+\frac{1+\lambda_3\lambda_4}{\lambda_3-\lambda_4}(\frac{m-2}{2}+2)=0,\\
%&\frac{1+\lambda_4\lambda_1}{\lambda_4-\lambda_1}(\frac{m-2}{2}-2)
%+\frac{1+\lambda_4\lambda_2}{\lambda_4-\lambda_2}(\frac{m-2}{2}-2)
%+\frac{1+\lambda_4\lambda_3}{\lambda_4-\lambda_3}(\frac{m-2}{2}+2)=0.
%\end{aligned}
%$$
By calculating \eqref{eqn:xzs1}--\eqref{eqn:xzs3}, we get
$$
(-4-4\alpha\lambda_3+(8+\alpha^2)\lambda_3^2)m
=-24-2\lambda_3(12\alpha+(-8+\alpha^2)\lambda_3).
$$
If $-4-4\alpha\lambda_3+(8+\alpha^2)\lambda_3^2=0$,
then $\lambda_3=\frac{-2}{\alpha \pm\sqrt{2}\sqrt{4+\alpha^2}}$, which contradicts with $-24-2\lambda_3(12\alpha+(-8+\alpha^2)\lambda_3)=0$. Then, it holds $m=\frac{-24-2\lambda_3(12\alpha+(-8+\alpha^2)\lambda_3)}{-4-4\alpha\lambda_3+(8+\alpha^2)\lambda_3^2}$.
Substituting such $m$ into \eqref{eqn:xzs1}+\eqref{eqn:xzs3}, we can have $\lambda_3=\frac{\alpha}{2}$. It contradicts with (1) of Lemma \ref{lemma:5.5}.

In {\bf Case-i-3}, there exist tangent vector fields $X,Y\in\mathcal{Q}$, such that
$SX=\alpha X$ and $SY=0$. Now, by Lemma \ref{lemma:5.5}, it follows that $S\phi X=\tfrac{\alpha^2+2}{\alpha}\phi X$ and $S\phi Y=\tfrac{-2}{\alpha}\phi Y$. Here,   $\{\alpha,0,\tfrac{\alpha^2+2}{\alpha},\tfrac{-2}{\alpha}\}$ are different to each other. By the assumption that $M$ has five distinct principal curvatures, and combining Lemma \ref{lemma:5.5}, we must have the following two subcases:

{\bf Case-i-3-1}: $\frac{\alpha+\sqrt{\alpha^2+4}}{2}\in\sigma(\mathcal{Q})$,
$\frac{\alpha-\sqrt{\alpha^2+4}}{2}\notin\sigma(\mathcal{Q})$.

{\bf Case-i-3-2}: $\frac{\alpha+\sqrt{\alpha^2+4}}{2}\notin\sigma(\mathcal{Q})$,
$\frac{\alpha-\sqrt{\alpha^2+4}}{2}\in\sigma(\mathcal{Q})$.

In {\bf Case-i-3-1}, we have
$$
\begin{aligned}
&\mathcal{Q}=V_\alpha\oplus V_{\tfrac{\alpha^2+2}{\alpha}}\oplus V_{0}\oplus V_{\frac{-2}{\alpha}}\oplus V_{\frac{\alpha+\sqrt{\alpha^2+4}}{2}},\\
&JV_\alpha=V_{\tfrac{\alpha^2+2}{\alpha}},\
JV_{0}=V_{\frac{-2}{\alpha}},\
JV_{\frac{\alpha+\sqrt{\alpha^2+4}}{2}}=V_{\frac{\alpha+\sqrt{\alpha^2+4}}{2}}.
\end{aligned}
$$
For convenience, let
\begin{equation}\label{eqn:5.44}
\lambda_1=\alpha,\ \lambda_2=\tfrac{\alpha^2+2}{\alpha},\ \lambda_3=0,\ \lambda_4=\frac{-2}{\alpha},
\ \lambda_5=\frac{\alpha+\sqrt{\alpha^2+4}}{2}.
\end{equation}
Now, it holds $\lambda_+=\tfrac{\alpha^2+2}{\alpha}$ on $M$. Then we can use the fact that
$M_+$ is austere to prove that {\bf Case-i-3-1} does not occur. In fact,
by Proposition \ref{prop:4.3w} and \eqref{eqn:5.44}, the principal curvatures on $M_+$ are
%$$
%\tilde{\lambda}_4=\frac{-1}{\alpha},\ \tilde{\lambda}_3=\frac{\alpha}{2+\alpha^2},\
%\tilde{\lambda}_1=\frac{\alpha(3+\alpha^2)}{2},\
%\tilde{\lambda}_5=\frac{\sqrt{4+\alpha^2}+\alpha(3+\alpha(\alpha+\sqrt{4+\alpha^2}))}{2}, %\tilde{\alpha}=\alpha(3+\alpha^2),\ \ 0.
%$$
$$
\frac{-1}{\alpha},\ \frac{\alpha}{2+\alpha^2},\
\frac{\alpha(3+\alpha^2)}{2},\
\frac{\sqrt{4+\alpha^2}+\alpha(3+\alpha(\alpha+\sqrt{4+\alpha^2}))}{2},\ \alpha(3+\alpha^2),\ \ 0.
$$
There are four positive principal curvatures and one negative principal curvature,
which contradicts with the fact that $M_+$ is austere.

In {\bf Case-i-3-2}, we have
$$
\begin{aligned}
&\mathcal{Q}=V_\alpha\oplus V_{\tfrac{\alpha^2+2}{\alpha}}\oplus V_{0}\oplus V_{\frac{-2}{\alpha}}\oplus V_{\frac{\alpha-\sqrt{\alpha^2+4}}{2}},\\
&J V_\alpha=V_{\tfrac{\alpha^2+2}{\alpha}},\
J V_{0}=V_{\frac{-2}{\alpha}},\
J V_{\frac{\alpha-\sqrt{\alpha^2+4}}{2}}=V_{\frac{\alpha-\sqrt{\alpha^2+4}}{2}}.
\end{aligned}
$$
For convenience, let
\begin{equation}\label{eqn:5.45}
\lambda_1=\alpha,\ \lambda_2=\tfrac{\alpha^2+2}{\alpha},\ \lambda_3=0,\ \lambda_4=\frac{-2}{\alpha},
\ \lambda_5=\frac{\alpha-\sqrt{\alpha^2+4}}{2}.
\end{equation}
Now, it holds $\lambda_-=\frac{-2}{\alpha}$ on $M$. Then we can use the fact that
$M_-$ is austere to prove that {\bf Case-i-3-2} does not occur. In fact,
by Proposition \ref{prop:4.4w} and \eqref{eqn:5.45}, the principal curvatures on $M_-$ are
%$$
%\tilde{\lambda}_4=\frac{-1}{\alpha},\ \tilde{\lambda}_3=\frac{\alpha}{2+\alpha^2},\
%\tilde{\lambda}_1=\frac{\alpha(3+\alpha^2)}{2},\
%\tilde{\lambda}_5=\frac{\sqrt{4+\alpha^2}+\alpha(3+\alpha(\alpha+\sqrt{4+\alpha^2}))}{2}, %\tilde{\alpha}=\alpha(3+\alpha^2),\ \ 0.
%$$
$$
\frac{1}{\alpha},\ \frac{\alpha}{2+\alpha^2},\
\frac{-\alpha}{2},\
\frac{-\alpha-\sqrt{4+\alpha^2}}{2}, \ -\alpha,\ \ 0.
$$
There are two positive principal curvatures and three negative principal curvature,
which contradicts with the fact that $M_-$ is austere.

In {\bf Case-i-4}, by the assumption that $M$ has five distinct principal curvatures, and combining Lemma \ref{lemma:5.5}, we must have the following two subcases:

{\bf Case-i-4-1}: $\frac{\alpha+\sqrt{\alpha^2+4}}{2}\in\sigma(\mathcal{Q})$,
$\frac{\alpha-\sqrt{\alpha^2+4}}{2}\notin\sigma(\mathcal{Q})$.

{\bf Case-i-4-2}: $\frac{\alpha+\sqrt{\alpha^2+4}}{2}\notin\sigma(\mathcal{Q})$,
$\frac{\alpha-\sqrt{\alpha^2+4}}{2}\in\sigma(\mathcal{Q})$.

In {\bf Case-i-4-1}, we have
$$
\begin{aligned}
&\mathcal{Q}=V_\lambda\oplus V_{\tfrac{\alpha\lambda+2}{2\lambda-\alpha}}\oplus  V_{\frac{\alpha+\sqrt{\alpha^2+4}}{2}},\ \lambda\notin\{\alpha,0,\frac{\alpha^2+2}{\alpha},\frac{-2}{\alpha},\frac{\alpha\pm\sqrt{\alpha^2+4}}{2}\},\\
&J V_\lambda=V_{\tfrac{\alpha\lambda+2}{2\lambda-\alpha}},\
J V_{\frac{\alpha+\sqrt{\alpha^2+4}}{2}}=V_{\frac{\alpha+\sqrt{\alpha^2+4}}{2}}.
\end{aligned}
$$
For convenience, let
$$
\lambda_1=\lambda,\ \lambda_2=\tfrac{\alpha\lambda+2}{2\lambda-\alpha},\ \lambda_3=\frac{\alpha+\sqrt{\alpha^2+4}}{2}.
$$
By (3) of Lemma \ref{lemma:5.1}, we have $A(V_{\lambda_1}\oplus V_{\lambda_2})= V_{\lambda_3}$.
So ${\rm dim}V_{\lambda_1}={\rm dim}V_{\lambda_2}=\frac{m-2}{2}$, ${\rm dim}V_{\lambda_3}=m-2$.
Without loss of generality, we assume $\lambda_1>\lambda_2$.

Now, from $\lambda\notin\{\alpha,0,\frac{\alpha^2+2}{\alpha},\frac{-2}{\alpha},\frac{\alpha\pm\sqrt{\alpha^2+4}}{2}\}$,
we have the following four situations:
$$
\begin{aligned}
&\lambda_1\in(\frac{\alpha+\sqrt{\alpha^2+4}}{2},\frac{\alpha^2+2}{\alpha}),\ \lambda_2\in(\alpha,\frac{\alpha+\sqrt{\alpha^2+4}}{2});\  {\rm or} \
\lambda_1\in(\frac{\alpha^2+2}{\alpha},+\infty),\ \lambda_2\in(\frac{\alpha}{2},\alpha); \\
&{\rm or}\  \lambda_1\in(0,\frac{\alpha}{2}),\ \lambda_2\in(-\infty,-\frac{2}{\alpha});\  {\rm or} \
\lambda_1\in(\frac{\alpha-\sqrt{\alpha^2+4}}{2},0),\ \lambda_2\in(\frac{-2}{\alpha},\frac{\alpha-\sqrt{\alpha^2+4}}{2}).
\end{aligned}
$$
For each one of above four situations, we all have $\lambda_1,\lambda_3>\lambda_2$ and $1+\lambda_1\lambda_3>0$.
Now, for any unit vector field $X_1\in V_{\lambda_1}$, it holds $AX_1\in V_{\lambda_3}$.
We take $(X,Y)=(X_1,AX_1)$ into \eqref{eqn:call}, it follows
\begin{equation}\label{eqn:5.47}
\begin{aligned}
2\sum_{\lambda_2}\frac{g((\nabla_{e_i} S)X_1,AX_1)^2}{(\lambda_1-\lambda_2)(\lambda_3-\lambda_2)}
&=(1+\lambda_1\lambda_3)\Big\{1+2g(\phi X_1,AX_1)^2-2g(AX_1,AX_1)^2-2g(AX_1,JAX_1)^2\Big\}\\
&=(1+\lambda_1\lambda_3)(1-2)=-(1+\lambda_1\lambda_3)<0.
\end{aligned}
\end{equation}
Note that, by $\lambda_1,\lambda_3>\lambda_2$, the left hand side of \eqref{eqn:5.47} is non-negative.
However, the right hand side of \eqref{eqn:5.47} is negative, which is a contradiction.

In {\bf Case-i-4-2}, we have
$$
\begin{aligned}
&\mathcal{Q}=V_\lambda\oplus V_{\tfrac{\alpha\lambda+2}{2\lambda-\alpha}}\oplus  V_{\frac{\alpha-\sqrt{\alpha^2+4}}{2}},\ \lambda\notin\{\alpha,0,\frac{\alpha^2+2}{\alpha},\frac{-2}{\alpha},\frac{\alpha\pm\sqrt{\alpha^2+4}}{2}\},\\
&J V_\lambda=V_{\tfrac{\alpha\lambda+2}{2\lambda-\alpha}},\
J V_{\frac{\alpha-\sqrt{\alpha^2+4}}{2}}=V_{\frac{\alpha-\sqrt{\alpha^2+4}}{2}}.
\end{aligned}
$$
For convenience, let
$$
\lambda_1=\lambda,\ \lambda_2=\tfrac{\alpha\lambda+2}{2\lambda-\alpha},\ \lambda_3=\frac{\alpha-\sqrt{\alpha^2+4}}{2}.
$$
By (3) of Lemma \ref{lemma:5.1}, we have $A(V_{\lambda_1}\oplus V_{\lambda_2})= V_{\lambda_3}$.
So ${\rm dim}V_{\lambda_1}={\rm dim}V_{\lambda_2}=\frac{m-2}{2}$, ${\rm dim}V_{\lambda_3}=m-2$.
Without loss of generality, we assume $\lambda_1>\lambda_2$.

Now, from $\lambda\notin\{\alpha,0,\frac{\alpha^2+2}{\alpha},\frac{-2}{\alpha},\frac{\alpha\pm\sqrt{\alpha^2+4}}{2}\}$,
we still have the following four situations:
$$
\begin{aligned}
&\lambda_1\in(\frac{\alpha+\sqrt{\alpha^2+4}}{2},\frac{\alpha^2+2}{\alpha}),\ \lambda_2\in(\alpha,\frac{\alpha+\sqrt{\alpha^2+4}}{2});\ {\rm or} \
\lambda_1\in(\frac{\alpha^2+2}{\alpha},+\infty),\ \lambda_2\in(\frac{\alpha}{2},\alpha); \\
&{\rm or}\  \lambda_1\in(0,\frac{\alpha}{2}),\ \lambda_2\in(-\infty,-\frac{2}{\alpha});\ {\rm or} \
\lambda_1\in(\frac{\alpha-\sqrt{\alpha^2+4}}{2},0),\ \lambda_2\in(\frac{-2}{\alpha},\frac{\alpha-\sqrt{\alpha^2+4}}{2}).
\end{aligned}
$$
For each one of above four situations, we all have $\lambda_2,\lambda_3<\lambda_1$ and $1+\lambda_2\lambda_3>0$.
Now, for any unit vector field $X_2\in V_{\lambda_2}$, it holds $AX_2\in V_{\lambda_3}$.
We take $(X,Y)=(X_2,AX_2)$ into \eqref{eqn:call}, it follows
\begin{equation}\label{eqn:5.49}
\begin{aligned}
2\sum_{\lambda_1}\frac{g((\nabla_{e_i} S)X_2,AX_2)^2}{(\lambda_2-\lambda_1)(\lambda_3-\lambda_1)}
&=(1+\lambda_2\lambda_3)\Big\{1+2g(\phi X_2,AX_2)^2-2g(AX_2,AX_2)^2-2g(AX_2,JAX_2)^2\Big\}\\
&=(1+\lambda_2\lambda_3)(1-2)=-(1+\lambda_2\lambda_3)<0.
\end{aligned}
\end{equation}
By $\lambda_2,\lambda_3<\lambda_1$, the left hand side of \eqref{eqn:5.49} is non-negative.
However, the right hand side of \eqref{eqn:5.49} is negative, which is a contradiction.

In {\bf Case-ii}, it holds $\alpha=0$. If $0\in\sigma(\mathcal{Q})$, then there exists a tangent vector field $X\in\mathcal{Q}$ such that $SX=0$. Now, by Lemma \ref{lemma:5.5}, it follows that $S\phi X=\frac{-2}{\alpha}\phi X$, which contradicts with $\alpha=0$. So, $0\notin\sigma(\mathcal{Q})$.
Then, there are four nonzero distinct principal curvatures $\{\lambda_1,\lambda_2,\lambda_3,\lambda_4\}=\sigma(\mathcal{Q})$ in $\mathcal{Q}$.

If there is only one of $\{1,-1\}$ in $\sigma(\mathcal{Q})$, without loss of generality, let $\lambda_4=1$ (or $\lambda_4=-1$). Then we can assume that $\lambda_1\notin\{0,1,-1\}$, and by Lemma \ref{lemma:5.5}, it also holds that $\lambda_2=\frac{1}{\lambda_1}\in\sigma(\mathcal{Q})$, which is different from $\{0,1,-1,\lambda_1\}$. Then, we have $\lambda_3\notin\{0,1,-1,\lambda_1,\lambda_2\}$.
By Lemma \ref{lemma:5.5}, we have $\frac{1}{\lambda_3}\in \sigma(\mathcal{Q})$, and
$\frac{1}{\lambda_3}\notin\{0,1,-1,\lambda_1,\lambda_2,\lambda_3\}$, which contradicts with
the assumption that $M$ has five distinct principal curvatures.
Thus, we only need to discuss the following two subcases.

{\bf Case-ii-1}: $1,-1\in\sigma(\mathcal{Q})$.

{\bf Case-ii-2}: $1,-1\notin\sigma(\mathcal{Q})$.

In {\bf Case-ii-1}, up to a sign of unit normal vector field,
there exist constant principal curvatures $\lambda,\mu\in\sigma(\mathcal{Q})$ and $\lambda>1>\mu>0$,
such that
$$
\mathcal{Q}=V_\lambda\oplus V_\mu\oplus V_1\oplus V_{-1},\ JV_\lambda= V_\mu,\ JV_1= V_1,
\ JV_{-1}= V_{-1},\ \mu=\frac{1}{\lambda},\ \lambda,\mu\notin\{0,1,-1\}.
$$

For convenience, let
\begin{equation}\label{eqn:5.50}
\lambda_1=\lambda,\ \lambda_2=\mu,\ \lambda_3=-1,\ \lambda_4=1.
\end{equation}
Then, it holds $\lambda_3<0<\lambda_2<\lambda_4<\lambda_1$.
Now, we have $\lambda_+=\lambda_1$ on $M$. Then, by Proposition \ref{prop:4.3w} and \eqref{eqn:5.50}, the principal curvatures on $M_+$ are
%$$
%\tilde{\lambda}_4=\frac{-1}{\alpha},\ \tilde{\lambda}_3=\frac{\alpha}{2+\alpha^2},\
%\tilde{\lambda}_1=\frac{\alpha(3+\alpha^2)}{2},\
%\tilde{\lambda}_5=\frac{\sqrt{4+\alpha^2}+\alpha(3+\alpha(\alpha+\sqrt{4+\alpha^2}))}{2}, %\tilde{\alpha}=\alpha(3+\alpha^2),\ \ 0.
%$$
$$
\frac{1-\lambda}{1+\lambda},\ \frac{2\lambda}{\lambda^2-1},\
\frac{\lambda+1}{\lambda-1},\
\frac{4\lambda}{\lambda^2-1},\ \ 0.
$$
By $\lambda>1$, there are three positive principal curvatures and one negative principal curvature
on $M_+$, which contradicts with the fact that $M_+$ is austere.
So {\bf Case-ii-1} does not occur.

In {\bf Case-ii-2}, without loss of generality, there exist constant principal curvatures $\lambda_1,\lambda_3\in\sigma(\mathcal{Q})$ and $\lambda_1>1$, such that
$$
\mathcal{Q}=V_{\lambda_1}\oplus V_{\lambda_2}\oplus V_{\lambda_3}\oplus V_{\lambda_4},
\ \lambda_2=\frac{1}{\lambda_1},\ \lambda_4=\frac{1}{\lambda_3},
\ JV_{\lambda_1}= V_{\lambda_2},\ JV_{\lambda_3}= V_{\lambda_4},
\ \lambda_i\notin\{0,1,-1\}.
$$
By Lemma \ref{lemma:5.1}, we have $A(V_{\lambda_1}\oplus V_{\lambda_2})=V_{\lambda_3}\oplus V_{\lambda_4}$.
So ${\rm dim}V_{\lambda_1}={\rm dim}V_{\lambda_2}={\rm dim}V_{\lambda_3}={\rm dim}V_{\lambda_4}=\frac{m-2}{2}$.

In the following, we take a unit vector $X\in V_{\lambda_i}$ ($1\leq i\leq4$)
into \eqref{eqn:ca2} respectively, we can obtain
\begin{gather}
\frac{1+\lambda_1\lambda_2}{\lambda_1-\lambda_2}(\frac{m-2}{2}+2)
+\frac{1+\lambda_1\lambda_3}{\lambda_1-\lambda_3}(\frac{m-2}{2}-2)
+\frac{1+\lambda_1\lambda_4}{\lambda_1-\lambda_4}(\frac{m-2}{2}-2)=0,\label{eqn:5.57}\\
\frac{1+\lambda_2\lambda_1}{\lambda_2-\lambda_1}(\frac{m-2}{2}+2)
+\frac{1+\lambda_2\lambda_3}{\lambda_2-\lambda_3}(\frac{m-2}{2}-2)
+\frac{1+\lambda_2\lambda_4}{\lambda_2-\lambda_4}(\frac{m-2}{2}-2)=0,\label{eqn:5.58}\\
\frac{1+\lambda_3\lambda_1}{\lambda_3-\lambda_1}(\frac{m-2}{2}-2)
+\frac{1+\lambda_3\lambda_2}{\lambda_3-\lambda_2}(\frac{m-2}{2}-2)
+\frac{1+\lambda_3\lambda_4}{\lambda_3-\lambda_4}(\frac{m-2}{2}+2)=0,\label{eqn:5.59}\\
\frac{1+\lambda_4\lambda_1}{\lambda_4-\lambda_1}(\frac{m-2}{2}-2)
+\frac{1+\lambda_4\lambda_2}{\lambda_4-\lambda_2}(\frac{m-2}{2}-2)
+\frac{1+\lambda_4\lambda_3}{\lambda_4-\lambda_3}(\frac{m-2}{2}+2)=0.\label{eqn:5.60}
\end{gather}
%If $\lambda_1=\frac{\sqrt{2}}{2}$ and $\lambda_3=\frac{3}{2\sqrt{2}}$, then by \eqref{eqn:5.57}+\eqref{eqn:5.59}, we have a contradiction. If $\lambda_1=\frac{\sqrt{2}}{2}$ and $\lambda_3\neq\frac{3}{2\sqrt{2}}$, then \eqref{eqn:5.57}+\eqref{eqn:5.59}=0, we can have

By calculating \eqref{eqn:5.57}+\eqref{eqn:5.59}, we have
$$
(\lambda_1+\lambda_3)\Big(-4+2m+6\lambda_1^2-m\lambda_1^2-2(2+m)\lambda_1\lambda_3
+\big(6-m+2(m-2)\lambda_1^2\big)\lambda_3^2\Big)=0.
$$
Without loss of generality, up to a change of $\lambda_3,\lambda_4$, we can assume that $\lambda_1+\lambda_3\neq0$, then we have
$$
(2-2\lambda_1\lambda_3-\lambda_3^2+\lambda_1^2(2\lambda_3^2-1))m
=4+4\lambda_1\lambda_3-6\lambda_3^2+\lambda_1^2(4\lambda_3^2-6).
$$
If $2-2\lambda_1\lambda_3-\lambda_3^2+\lambda_1^2(2\lambda_3^2-1)=0$,
then $\lambda_3=\frac{\lambda_1\pm\sqrt{2}(\lambda_1^2-1)}{2\lambda_1^2-1}$, which contradicts with
$4+4\lambda_1\lambda_3-6\lambda_3^2+\lambda_1^2(4\lambda_3^2-6)=0$. Thus, we get
$m=\frac{4+4\lambda_1\lambda_3-6\lambda_3^2+\lambda_1^2(4\lambda_3^2-6)}
{2-2\lambda_1\lambda_3-\lambda_3^2+\lambda_1^2(2\lambda_3^2-1)}$.
Substituting such $m$ into \eqref{eqn:5.57}, we have $(1+\lambda_1\lambda_3)
(1-4\lambda_1\lambda_3+\lambda_3^2+\lambda_1^2(1+\lambda_3^2))=0$. If $1+\lambda_1\lambda_3\neq0$,
then $1-4\lambda_1\lambda_3+\lambda_3^2+\lambda_1^2(1+\lambda_3^2)=0$. But, such equation does not have real root of $\lambda_3$, which deduces a contradiction. So, we have $1+\lambda_1\lambda_3=0$
and then $\lambda_3=\frac{-1}{\lambda_1}$. %Note that, up to a change of $\lambda_3,\lambda_4$,
%it is equivalent to $\lambda_3=-\lambda_1$.
Thus, without loss of generality, we can assume that
\begin{equation}\label{eqn:caseii-2-1}
\lambda_3=\frac{-1}{\lambda_1},\ \ \lambda_4=-\lambda_1,\ \
\lambda_4<-1<\lambda_3<0<\lambda_2<1<\lambda_1.
\end{equation}

Now, we have $\lambda_+=\lambda_1$ on $M$. Then, by Proposition \ref{prop:4.3w} and \eqref{eqn:caseii-2-1}, the principal curvatures on $M_+$ are
%$$
%\tilde{\lambda}_4=\frac{-1}{\alpha},\ \tilde{\lambda}_3=\frac{\alpha}{2+\alpha^2},\
%\tilde{\lambda}_1=\frac{\alpha(3+\alpha^2)}{2},\
%\tilde{\lambda}_5=\frac{\sqrt{4+\alpha^2}+\alpha(3+\alpha(\alpha+\sqrt{4+\alpha^2}))}{2}, %\tilde{\alpha}=\alpha(3+\alpha^2),\ \ 0.
%$$
$$
-\frac{\lambda_1^2-1}{2\lambda_1},\ \frac{2\lambda_1}{\lambda_1^2-1},\
\frac{4\lambda_1}{\lambda_1^2-1},\ 0.
$$
By $\lambda_1>1$, there are two positive principal curvatures and one negative principal curvature
on $M_+$, which contradicts with the fact that $M_+$ is austere. Thus, {\bf Case-ii-2} does not occur.

We have completed the proof of the Theorem \ref{thm:5.5}.
\end{proof}

\vskip 0.2cm
\noindent{\bf Completion of the proof of Theorem \ref{thm:1.1a}}.

Based on Theorem \ref{thm:5.1}, Theorem \ref{thm:5.2}, Theorem \ref{thm:5.3},
Theorem \ref{thm:5.4} and Theorem \ref{thm:5.5}, we have completed the proof of the Theorem \ref{thm:1.1a}.

%==================================================

%==================================================
%\section{Isoparametric hypersurface of $Q^3$}\label{sect:6}

%The isoparametric hypersurface of $Q^2$ (i.e. $\mathbb{S}^2\times \mathbb{S}^2$)
%is classified by Urbano \cite{Ur}, the same method by using Jacobi fields theorem may be applied to $Q^3$.

%\begin{theorem}\label{thm:6.1}
%Let $M$ be an isoparametric hypersurface of $Q^3$ (or $Q^4$) with $\mathfrak{A}$-principal normal.
%Then, $M$ has constant principal curvatures, and $M$ is an open part of a tube over a totally geodesic Lagrangian submanifold $S^{m}(1)\hookrightarrow Q^{m}$.
%\end{theorem}

%I guess that above theorem is right for any $Q^m$ ($m\geq3$).

%==================================================
\section{Proof of Corollary \ref{cor:1.1} and Theorems \ref{thm:1.3}--\ref{thm:1.5}}\label{sect:6}

In this section, for low dimension complex quadrics $Q^3,Q^4$ and $Q^5$, we classify the Hopf hypersurfaces with constant principal curvatures. For the Hopf
hypersurfaces of $Q^6$ with constant principal curvatures, we give the values
and multiplicities of the principal curvatures.

First of all, as a direct consequence of Theorem \ref{thm:1.1a}, we have the following classification result.

\begin{corollary}[=Corollary \ref{cor:1.1}]\label{cor:6.1}
Let $M$ be a Hopf hypersurface of $Q^3$ with constant principal curvatures. Then,
$M$ is an open part of a tube over a totally geodesic $Q^{2}\hookrightarrow Q^{3}$.
\end{corollary}
\begin{proof}
If $M$ is a Hopf hypersurface of $Q^3$ with constant principal curvatures,
from the fact that the dimension of $M$ is $5$, then $M$ has
at most five distinct constant principal curvatures. Thus by Theorem \ref{thm:1.1a},
$M$ is an open part of a tube over a totally geodesic $Q^{2}\hookrightarrow Q^{3}$.
\end{proof}

Next, we give the following classification results for $Q^4$ and $Q^5$.

\begin{theorem}[=Theorem \ref{thm:1.3}]\label{thm:6.2}
Let $M$ be a Hopf hypersurface of $Q^4$ with constant principal curvatures. Then,
\begin{enumerate}
\item[(1)]
$M$ is an open part of a tube over a totally geodesic $Q^{3}\hookrightarrow Q^{4}$; or

\item[(2)]
$M$ is an open part of a tube over a totally geodesic $\mathbb{C}P^2\hookrightarrow Q^{4}$.
\end{enumerate}
\end{theorem}
\begin{proof}
If $M$ has $\mathfrak{A}$-principal unit normal vector field $N$, then $M$ is an open part of a tube over a totally geodesic $Q^{3}\hookrightarrow Q^{4}$.

If $M$ has $\mathfrak{A}$-isotropic unit normal vector field $N$, then by Lemma \ref{lemma:2.6}, we have $S\xi=\alpha\xi$, and $SAN=SA\xi=0$. It follows that $\alpha$ and $0$ are principal curvatures. Based on Theorem \ref{thm:1.1a}
and the fact that the multiplicity of $0$ is at least $2$, we only need to consider the case that $M$
has six distinct constant principal curvatures, where the multiplicity of $0$ is $2$, and $\alpha>0$,
$\alpha,0\notin\sigma(\mathcal{Q})$. In this case, there are four distinct principal curvatures in
$\sigma(\mathcal{Q})$.

Now, by Lemma \ref{lemma:5.5}, we have
$$
\begin{aligned}
&\mathcal{Q}=V_{\lambda_1}\oplus V_{\lambda_2}\oplus V_{\lambda_3}\oplus V_{\lambda_4},\ \lambda_i\notin\{\alpha,0,\tfrac{\alpha^2+2}{\alpha},\tfrac{-2}{\alpha},
\tfrac{\alpha\pm\sqrt{\alpha^2+4}}{2}\},\ \ i\in\{1,2,3,4\},\\
&\lambda_2=\tfrac{\alpha\lambda_1+2}{2\lambda_1-\alpha},\ \lambda_4=\tfrac{\alpha\lambda_3+2}{2\lambda_3-\alpha},\ JV_{\lambda_1}=V_{\lambda_2},\ \ JV_{\lambda_3}=V_{\lambda_4}.
\end{aligned}
$$
Then, by (3) of Lemma \ref{lemma:5.1}, we have
$A(V_{\lambda_1}\oplus V_{\lambda_2})=(V_{\lambda_3}\oplus V_{\lambda_4})$.
Note that ${\rm dim}V_{\lambda_1}={\rm dim}V_{\lambda_2}={\rm dim}V_{\lambda_3}={\rm dim}V_{\lambda_4}=1$, then for any unit vector field $X_i\in V_{\lambda_i}$, $1\leq i\leq4$,
it holds that $g(AX_i,X_j)^2+g(AX_i,JX_j)^2=1$, for $1\leq i\leq2$ and $3\leq j\leq4$.

In the following, without loss of generality, we divide the discussions into three cases:

{\bf Case-i}: $0<\lambda_1<\lambda_3<\frac{\alpha+\sqrt{\alpha^2+4}}{2}<\lambda_4<\lambda_2$.

{\bf Case-ii}: $\lambda_1<\lambda_3<\frac{\alpha-\sqrt{\alpha^2+4}}{2}<\lambda_4<\lambda_2$.

{\bf Case-iii}: $\lambda_1<\frac{\alpha-\sqrt{\alpha^2+4}}{2}<\lambda_2
<\lambda_3<\frac{\alpha+\sqrt{\alpha^2+4}}{2}<\lambda_4$.

In {\bf Case-i}, we take $(X,Y)=(X_2,X_4)$ into \eqref{eqn:call}, where $X_2\in V_{\lambda_2}, X_4\in V_{\lambda_4}$, it follows
\begin{equation}\label{eqn:6.1}
\begin{aligned}
&2\frac{g((\nabla_{X_1} S)X_2,X_4)^2}{(\lambda_2-\lambda_1)(\lambda_4-\lambda_1)}
+2\frac{g((\nabla_{X_3} S)X_2,X_4)^2}{(\lambda_2-\lambda_3)(\lambda_4-\lambda_3)}\\
&=(1+\lambda_2\lambda_4)\Big\{1+2g(\phi X_2,X_4)^2-2g(AX_2,X_4)^2-2g(AX_2,JX_4)^2\Big\}\\
&=-(1+\lambda_2\lambda_4)<0.
\end{aligned}
\end{equation}
By $0<\lambda_1<\lambda_3<\frac{\alpha+\sqrt{\alpha^2+4}}{2}<\lambda_4<\lambda_2$, the left hand side of \eqref{eqn:6.1} is non-negative. However, the right hand side of \eqref{eqn:6.1} is negative, which is a contradiction.

In {\bf Case-ii}, we take $(X,Y)=(X_1,X_3)$ into \eqref{eqn:call}, where $X_1\in V_{\lambda_1}, X_3\in V_{\lambda_3}$, it follows
\begin{equation}\label{eqn:6.2}
\begin{aligned}
&2\frac{g((\nabla_{X_2} S)X_1,X_3)^2}{(\lambda_1-\lambda_2)(\lambda_3-\lambda_2)}
+2\frac{g((\nabla_{X_4} S)X_1,X_3)^2}{(\lambda_1-\lambda_4)(\lambda_3-\lambda_4)}\\
&=(1+\lambda_1\lambda_3)\Big\{1+2g(\phi X_1,X_3)^2-2g(AX_1,X_3)^2-2g(AX_1,JX_3)^2\Big\}\\
&=-(1+\lambda_1\lambda_3)<0.
\end{aligned}
\end{equation}
By $\lambda_1<\lambda_3<\frac{\alpha-\sqrt{\alpha^2+4}}{2}<\lambda_4<\lambda_2$, the left hand side of \eqref{eqn:6.2} is non-negative. However, the right hand side of \eqref{eqn:6.2} is negative, which is a contradiction.

In {\bf Case-iii}, we have $|\lambda_2\lambda_3|<1$, $|\lambda_1\lambda_4|>1$ and $\lambda_1\lambda_4<0$.
Taking $(X,Y)=(X_1,X_4)$ into \eqref{eqn:call}, where $X_1\in V_{\lambda_1}, X_4\in V_{\lambda_4}$, it follows
\begin{equation}\label{eqn:6.3}
\begin{aligned}
&2\frac{g((\nabla_{X_2} S)X_1,X_4)^2}{(\lambda_1-\lambda_2)(\lambda_4-\lambda_2)}
+2\frac{g((\nabla_{X_3} S)X_1,X_4)^2}{(\lambda_1-\lambda_3)(\lambda_4-\lambda_3)}\\
&=(1+\lambda_1\lambda_4)\Big\{1+2g(\phi X_1,X_4)^2-2g(AX_1,X_4)^2-2g(AX_1,JX_4)^2\Big\}\\
&=-(1+\lambda_1\lambda_4)>0.
\end{aligned}
\end{equation}
By $\lambda_1<\frac{\alpha-\sqrt{\alpha^2+4}}{2}<\lambda_2
<\lambda_3<\frac{\alpha+\sqrt{\alpha^2+4}}{2}<\lambda_4$, the left hand side of \eqref{eqn:6.3} is non-positive. However, the right hand side of \eqref{eqn:6.3} is positive, which is a contradiction.

Thus, $M$ cannot have six distinct constant principal curvatures.
So, $M$ is a Hopf hypersurface of $Q^4$ with at most five distinct constant principal curvatures
and $\mathfrak{A}$-isotropic unit normal vector field. It follows from Theorem \ref{thm:1.1a} that $M$ is an open part of a tube over a totally geodesic $\mathbb{C}P^2\hookrightarrow Q^{4}$.
\end{proof}

\begin{theorem}[=Theorem \ref{thm:1.4}]\label{thm:6.3}
Let $M$ be a Hopf hypersurface of $Q^5$ with constant principal curvatures. Then,
$M$ is an open part of a tube over a totally geodesic $Q^{4}\hookrightarrow Q^{5}$.
\end{theorem}
\begin{proof}
If $M$ has $\mathfrak{A}$-principal unit normal vector field $N$, then $M$ is an open part of a tube over a totally geodesic $Q^{4}\hookrightarrow Q^{5}$.

If $M$ has $\mathfrak{A}$-isotropic unit normal vector field $N$, then by Lemma \ref{lemma:2.6}, we have $S\xi=\alpha\xi$, and $SAN=SA\xi=0$. So $\alpha$ and $0$ are principal curvatures. If there is some $\lambda\in \sigma(\mathcal{Q})$ such that ${\rm dim}(V_{\lambda}\oplus JV_{\lambda})\geq4$, then
${\rm dim}(A(V_{\lambda}\oplus JV_{\lambda}))\geq4$.
By (3) of Lemma \ref{lemma:5.1} and ${\rm dim}\mathcal{Q}=6$, we have $A(V_{\lambda}\oplus JV_{\lambda})\subset (\mathcal{Q}\ominus(V_{\lambda}\oplus JV_{\lambda}))$, then ${\rm dim}(A(V_{\lambda}\oplus JV_{\lambda}))\leq2$. We get a contradiction. So, for any $\lambda\in \sigma(\mathcal{Q})$,
it holds that ${\rm dim}(V_{\lambda}\oplus JV_{\lambda})=2$.
According to this fact, we have
$$
\begin{aligned}
&\mathcal{Q}=V_{\lambda_1}\oplus V_{\lambda_2}\oplus V_{\lambda_3}\oplus V_{\lambda_4}\oplus V_{\lambda_5}\oplus V_{\lambda_6},\
\lambda_2=\tfrac{\alpha\lambda_1+2}{2\lambda_1-\alpha},\ \lambda_4=\tfrac{\alpha\lambda_3+2}{2\lambda_3-\alpha},\
\lambda_6=\tfrac{\alpha\lambda_5+2}{2\lambda_5-\alpha},\\
&\{\lambda_1,\lambda_2\}\neq\{\lambda_3,\lambda_4\},
\{\lambda_1,\lambda_2\}\neq\{\lambda_5,\lambda_6\},
\{\lambda_3,\lambda_4\}\neq\{\lambda_5,\lambda_6\},\
JV_{\lambda_1}=V_{\lambda_2},\ JV_{\lambda_3}=V_{\lambda_4},\ JV_{\lambda_5}=V_{\lambda_6}.
\end{aligned}
$$
Note that, $\lambda_1$ and $\lambda_2$ (or $\lambda_3$ and $\lambda_4$, or $\lambda_5$ and $\lambda_6$)
may be the same constant. Thus, without loss of generality, we only need to consider
the following four cases on $\mathcal{Q}$:

{\bf 1}: $\frac{\alpha-\sqrt{\alpha^2+4}}{2}<\lambda_4<\lambda_2<\lambda_6\leq\frac{\alpha+\sqrt{\alpha^2+4}}{2}\leq\lambda_5<\lambda_1<\lambda_3$;

{\bf 2}: $\lambda_4<\lambda_2<\lambda_6\leq\frac{\alpha-\sqrt{\alpha^2+4}}{2}\leq\lambda_5<\lambda_1<\lambda_3
<\frac{\alpha+\sqrt{\alpha^2+4}}{2}$;

{\bf 3}: $\lambda_6\leq\frac{\alpha-\sqrt{\alpha^2+4}}{2}\leq\lambda_5<\lambda_1<\lambda_3
\leq\frac{\alpha+\sqrt{\alpha^2+4}}{2}\leq\lambda_4<\lambda_2$;

{\bf 4}: $\lambda_2<\lambda_6\leq\frac{\alpha-\sqrt{\alpha^2+4}}{2}\leq\lambda_5<\lambda_1
<\lambda_3\leq\frac{\alpha+\sqrt{\alpha^2+4}}{2}\leq\lambda_4$.

We deal with the above four cases together. By (3) of Lemma \ref{lemma:5.1}, for a unit vector field $X_1\in V_{\lambda_1}$, we can assume that $AX_1=a Y_3+b Y_5$,
where $Y_3\in (V_{\lambda_3}\oplus JV_{\lambda_3})$, $Y_5\in (V_{\lambda_5}\oplus JV_{\lambda_5})$
are unit vector fields, and $a,b$ are functions satisfying $a^2+b^2=1$.
Choosing unit vector fields $X_3\in V_{\lambda_3}$ and $X_5\in V_{\lambda_5}$, by
${\rm dim}(V_{\lambda_3}\oplus JV_{\lambda_3})=2$ and ${\rm dim}(V_{\lambda_5}\oplus JV_{\lambda_5})=2$,
we have $g(AX_1,X_3)^2+g(AX_1,JX_3)^2=a^2$ and $g(AX_1,X_5)^2+g(AX_1,JX_5)^2=b^2$.

Now, taking $(X,Y)=(X_1,X_3)$ and $(X,Y)=(X_1,X_5)$ into \eqref{eqn:call}, it follows
\begin{equation}\label{eqn:6.4}
\begin{aligned}
&2\sum_{\lambda_i\neq\lambda_1,\lambda_3}\frac{g((\nabla_{e_i} S)X_1,X_3)^2}{(\lambda_1-\lambda_i)(\lambda_3-\lambda_i)}\\
&=(1+\lambda_1\lambda_3)\Big\{1+2g(\phi X_1,X_3)^2-2g(AX_1,X_3)^2-2g(AX_1,JX_3)^2\Big\}\\
&=(1+\lambda_1\lambda_3)(1-2a^2).
\end{aligned}
\end{equation}
\begin{equation}\label{eqn:6.5}
\begin{aligned}
&2\sum_{\lambda_i\neq\lambda_1,\lambda_5}\frac{g((\nabla_{e_i} S)X_1,X_5)^2}{(\lambda_1-\lambda_i)(\lambda_5-\lambda_i)}\\
&=(1+\lambda_1\lambda_5)\Big\{1+2g(\phi X_1,X_5)^2-2g(AX_1,X_5)^2-2g(AX_1,JX_5)^2\Big\}\\
&=(1+\lambda_1\lambda_5)(1-2b^2).
\end{aligned}
\end{equation}
For each one of above four cases, we always have that the left hand side of \eqref{eqn:6.4}
and \eqref{eqn:6.5} are non-negative. Moreover, we also have $1+\lambda_1\lambda_3>0$ and $1+\lambda_1\lambda_5>0$. It follows that $a^2\leq\tfrac{1}{2}$ and
$b^2\leq\tfrac{1}{2}$. Then, without loss of generality, we can have $a=b=\tfrac{\sqrt{2}}{2}$.
So $AX_1=\frac{1}{\sqrt{2}}(Y_3+Y_5)$ and
\begin{equation}\label{eqn:6.6}
X_1=\frac{1}{\sqrt{2}}(AY_3+AY_5).
\end{equation}
Taking the inner product of \eqref{eqn:6.6} with $Y_3$ and $JY_5$, with the use of (3) of Lemma \ref{lemma:5.1}, $V_{\lambda_3}\oplus JV_{\lambda_3}={\rm Span}\{Y_3,JY_3\}$ and
$V_{\lambda_5}\oplus JV_{\lambda_5}={\rm Span}\{Y_5,JY_5\}$,
we get $g(AY_3,Y_5)=g(AY_3,JY_5)=0$. Further, it follows that $AY_3\in (V_{\lambda_1}\oplus JV_{\lambda_1})$.
By ${\rm Span}\{Y_3,JY_3\}=V_{\lambda_3}\oplus JV_{\lambda_3}$, so
$A(V_{\lambda_3}\oplus JV_{\lambda_3})\subset (V_{\lambda_1}\oplus JV_{\lambda_1})$.
Due that ${\rm dim}(V_{\lambda_1}\oplus JV_{\lambda_1})={\rm dim}(V_{\lambda_3}\oplus JV_{\lambda_3})=2$, we get $A(V_{\lambda_1}\oplus JV_{\lambda_1})=V_{\lambda_3}\oplus JV_{\lambda_3}$.
Thus, it deduces that
$A(V_{\lambda_5}\oplus JV_{\lambda_5})=(V_{\lambda_5}\oplus JV_{\lambda_5})$, which contradicts
with (3) of Lemma \ref{lemma:5.1}. So $M$ cannot have $\mathfrak{A}$-isotropic unit normal vector field $N$.

We have completed the proof of the Theorem \ref{thm:6.3}.
\end{proof}

Finally, for the complex quadric $Q^6$, we give the following result.

\begin{theorem}[=Theorem \ref{thm:1.5}]\label{thm:6.4}
Let $M$ be a Hopf hypersurface of $Q^6$ with constant principal curvatures. Then,
\begin{enumerate}
\item[(1)]
$M$ is an open part of a tube over a totally geodesic $Q^{5}\hookrightarrow Q^{6}$; or

\item[(2)]
$M$ is an open part of a tube over a totally geodesic $\mathbb{C}P^3\hookrightarrow Q^{6}$; or

\item[(3)]
$M$ has six distinct constant principal curvatures. Their values and multiplicities
are given by
$$
\begin{tabular}{|c|c|c|c|c|c|c|}
  \hline
  {\rm value} & $2\cot(2r)$ & $0$ & $\cot(r)$ & $-\tan(r)$ & $-\frac{\cos(r)+\sin(r)}{\cos(r)-\sin(r)}$ & $\frac{\cos(r)-\sin(r)}{\cos(r)+\sin(r)}$\\
  \hline
  {\rm multiplicity} & $1$ & $2$ & $2$ & $2$ & $2$ & $2$\\
  \hline
\end{tabular}
$$
where $0<r<\frac{\pi}{4}$. %It holds $\mathcal{Q}=V_{\lambda_1}\oplus V_{\lambda_2}\oplus V_{\lambda_3}\oplus V_{\lambda_4}$. The almost product structure $A\in\mathfrak{A}$ maps
%$V_{\lambda_1}\oplus V_{\lambda_2}$ into $V_{\lambda_3}\oplus V_{\lambda_4}$, and vice versa.
The Reeb function is $\alpha=2\tan(2r)$. In particular, Example \ref{E3}, for $k=2$, is contained in this case.
\end{enumerate}
\end{theorem}
\begin{proof}
If $M$ has $\mathfrak{A}$-principal unit normal vector field $N$, then $M$ is an open part of a tube over a totally geodesic $Q^{5}\hookrightarrow Q^{6}$.

If $M$ has $\mathfrak{A}$-isotropic unit normal vector field $N$,
we always assume that $\alpha\geq0$. By (2) of Lemma \ref{lemma:5.5}, there are $2k$ principal curvatures (counted with multiplicities) greater than $\frac{\alpha}{2}$ in $\sigma(\mathcal{Q})$, $0\leq k\leq4$.

When $k=4$, if $\lambda_+=\frac{\alpha+\sqrt{\alpha^2+4}}{2}$, then
$\sigma(\mathcal{Q})=\frac{\alpha+\sqrt{\alpha^2+4}}{2}$. By (3) of Lemma \ref{lemma:5.1},
we have $A\mathcal{Q}\perp\mathcal{Q}$, which is contradiction. So $\lambda_+>\frac{\alpha+\sqrt{\alpha^2+4}}{2}$. Then by Proposition \ref{prop:4.3w}, the principal
curvatures are all non-negative (at least five positive principal
curvatures) on focal submanifold $M_+$, which also contradicts with that $M_+$ is austere.

When $k=3$, if $\lambda_+=\frac{\alpha+\sqrt{\alpha^2+4}}{2}$, then
${\rm dim}V_{\frac{\alpha+\sqrt{\alpha^2+4}}{2}}={\rm dim}(AV_{\frac{\alpha+\sqrt{\alpha^2+4}}{2}})=6$ and ${\rm dim}(\mathcal{Q}\ominus V_{\frac{\alpha+\sqrt{\alpha^2+4}}{2}})=2$. By (3) of Lemma \ref{lemma:5.1},
we have $AV_{\frac{\alpha+\sqrt{\alpha^2+4}}{2}}\perp V_{\frac{\alpha+\sqrt{\alpha^2+4}}{2}}$,
i.e. $AV_{\frac{\alpha+\sqrt{\alpha^2+4}}{2}}\subset (\mathcal{Q}\ominus V_{\frac{\alpha+\sqrt{\alpha^2+4}}{2}})$. It is a contradiction.
So $\lambda_+>\frac{\alpha+\sqrt{\alpha^2+4}}{2}$ and
${\rm dim}V_{\lambda_+}\in\{1,2,3\}$. Note that there are at least two principal curvatures
being $0$ on $M_+$.
If ${\rm dim}V_{\lambda_+}=1$, then ${\rm dim}M_+=10$, and there are at least $6$ positive
principal curvatures (counted with multiplicities) on $M_+$, so $M_+$ can not be austere, it is a contradiction. If ${\rm dim}V_{\lambda_+}=2$, then ${\rm dim}M_+=9$, and there are at least $5$ positive principal curvatures (counted with multiplicities) on $M_+$, so $M_+$ can not be austere, it is a contradiction.
If ${\rm dim}V_{\lambda_+}=3$, then ${\rm dim}M_+=8$, and there are at least $4$ positive principal curvatures (counted with multiplicities) on $M_+$, so $M_+$ can not be austere, it is a contradiction.

Now, there are $2k$ principal curvatures (counted with multiplicities) greater than $\frac{\alpha}{2}$ in $\sigma(\mathcal{Q})$, $0\leq k\leq2$. And, there are $8-2k$ principal curvatures (counted with multiplicities) smaller that $\frac{\alpha}{2}$ in $\sigma(\mathcal{Q})$, $0\leq k\leq2$.
Further, if $\lambda_-=\frac{\alpha-\sqrt{\alpha^2+4}}{2}$ and $\lambda_+>\frac{\alpha+\sqrt{\alpha^2+4}}{2}$, then ${\rm dim}M_+\in\{9,10\}$,
and there is only one negative principal curvature
$\frac{1+\lambda_+\lambda_-}{\lambda_+-\lambda_-}$ with multiplicity at least $4$ on
$M_+$. But, it can be checked that for any positive principal curvature on $M_+$,
its multiplicity is at most $3$. So we get a contradiction. Thus, if $\lambda_-=\frac{\alpha-\sqrt{\alpha^2+4}}{2}$, then
$\lambda_+=\frac{\alpha+\sqrt{\alpha^2+4}}{2}$ and $\sigma(\mathcal{Q})=
\{\frac{\alpha-\sqrt{\alpha^2+4}}{2},\frac{\alpha+\sqrt{\alpha^2+4}}{2}\}$.
By (3) of Lemma \ref{lemma:5.1} and Lemma \ref{lemma:5.5}, it holds that
${\rm dim}V_{\frac{\alpha-\sqrt{\alpha^2+4}}{2}}={\rm dim}V_{\frac{\alpha+\sqrt{\alpha^2+4}}{2}}=4$, $JV_{\frac{\alpha-\sqrt{\alpha^2+4}}{2}}=V_{\frac{\alpha-\sqrt{\alpha^2+4}}{2}}$ and
$JV_{\frac{\alpha+\sqrt{\alpha^2+4}}{2}}=V_{\frac{\alpha+\sqrt{\alpha^2+4}}{2}}$.
So $M$ satisfies $S\phi=\phi S$, it follows that $M$ is an open part of a tube over a totally geodesic $\mathbb{C}P^3\hookrightarrow Q^{6}$.

In the following, we only need to consider $\lambda_-<\frac{\alpha-\sqrt{\alpha^2+4}}{2}$.
Denote the principal curvatures in $\sigma(\mathcal{Q})$ by $\lambda_i,\mu_i$,
$1\leq i\leq4$, where $\mu_i=\tfrac{\alpha\lambda_i+2}{2\lambda_i-\alpha}$.
Here, some $\lambda_i,\mu_i$ may be the same constants.
We always set $\lambda_1$ to be the smallest in $\sigma(\mathcal{Q})$.
So $\lambda_-=\lambda_1$.
By Proposition \ref{prop:4.4w}, on $M_-$, the principal curvatures are given by
\begin{equation}\label{eqn:6.7as1}
0,\ \tilde{\alpha}=\alpha+\frac{(4+\alpha^2){\lambda_1}}{\lambda_1^2-\alpha{\lambda_1}-1},\
\tilde{\lambda}_i=\frac{1+{\lambda_1}\lambda_i}{{\lambda_1}-\lambda_i},\
\tilde{\mu}_i=\frac{1+{\lambda_1}\mu_i}{{\lambda_1}-\mu_i},\ \lambda_i\neq\lambda_1,\ \mu_i\neq\lambda_1.
\end{equation}

According to that there are $2k$ principal curvatures (counted with multiplicities) greater than $\frac{\alpha}{2}$ in $\sigma(\mathcal{Q})$, $0\leq k\leq2$,
we need to consider the following three cases:

{\bf Case-1}: There are $4$ principal curvatures
(counted with multiplicities) greater than $\frac{\alpha}{2}$ in $\sigma(\mathcal{Q})$;

{\bf Case-2}: There are $2$ principal curvatures
(counted with multiplicities) greater than $\frac{\alpha}{2}$ in $\sigma(\mathcal{Q})$;

{\bf Case-3}: There are $0$ principal curvatures
(counted with multiplicities) greater than $\frac{\alpha}{2}$ in $\sigma(\mathcal{Q})$.

In {\bf Case-1}, by $\lambda_-<\frac{\alpha-\sqrt{\alpha^2+4}}{2}$, it holds ${\rm dim}V_{\lambda_-}\leq2$.
If ${\rm dim}V_{\lambda_-}=1$, without loss of generality, we assume that $\lambda_1<\lambda_2\leq\frac{\alpha-\sqrt{\alpha^2+4}}{2}\leq
\mu_2<\mu_1<\frac{\alpha}{2}<\lambda_3\leq\lambda_4
\leq\frac{\alpha+\sqrt{\alpha^2+4}}{2}\leq\mu_4\leq\mu_3$.
Then, by equation \eqref{eqn:6.7as1}, it holds $\tilde{\lambda}_2\leq\tilde{\mu}_2<\tilde{\mu}_1
<\tilde{\lambda}_3\leq\tilde{\lambda}_4\leq\tilde{\mu}_4\leq\tilde{\mu}_3$ on $M_-$.

When $\alpha>0$, if $\lambda_1\leq\frac{-2-\sqrt{\alpha^2+4}}{\alpha}$, then
$\tilde{\alpha}\geq0$, $\tilde{\lambda}_3,\tilde{\lambda}_4,\tilde{\mu}_4,\tilde{\mu}_3>0$
on $M_-$. It contradicts with that
$M_-$ is austere. So $\lambda_1>\frac{-2-\sqrt{\alpha^2+4}}{\alpha}$, then $\tilde{\alpha}<0$
on $M_-$. If $\alpha=0$ on $M$, then it must hold $\tilde{\alpha}<0$ on $M_-$. Due that $\tilde{\lambda}_2\leq\tilde{\mu}_2<\tilde{\mu}_1
<\tilde{\lambda}_3\leq\tilde{\lambda}_4\leq\tilde{\mu}_4\leq\tilde{\mu}_3$ on $M_-$,
thus from the fact that $M_-$ is austere, we need to consider the following four subcases on $M_-$:

{\bf Case-1-1}: $\tilde{\alpha}\leq\tilde{\lambda}_2\leq\tilde{\mu}_2<\tilde{\mu}_1
<0<\tilde{\lambda}_3\leq\tilde{\lambda}_4\leq\tilde{\mu}_4\leq\tilde{\mu}_3$;

{\bf Case-1-2}: $\tilde{\lambda}_2\leq\tilde{\alpha}\leq\tilde{\mu}_2<\tilde{\mu}_1
<0<\tilde{\lambda}_3\leq\tilde{\lambda}_4\leq\tilde{\mu}_4\leq\tilde{\mu}_3$;

{\bf Case-1-3}: $\tilde{\lambda}_2\leq\tilde{\mu}_2\leq\tilde{\alpha}\leq\tilde{\mu}_1
<0<\tilde{\lambda}_3\leq\tilde{\lambda}_4\leq\tilde{\mu}_4\leq\tilde{\mu}_3$;

{\bf Case-1-4}: $\tilde{\lambda}_2\leq\tilde{\mu}_2\leq\tilde{\mu}_1\leq\tilde{\alpha}
<0<\tilde{\lambda}_3\leq\tilde{\lambda}_4\leq\tilde{\mu}_4\leq\tilde{\mu}_3$.

Since $\tilde{\alpha}<0$ on $M_-$, then by Proposition \ref{prop:4.1w}, there is a parallel hypersurface $\Phi_{r_0}(M)$ ($r_2<r_0\leq0$) of $M$ such that the Reeb function being $0$ on $\Phi_{r_0}(M)$.
For the convenience of calculation, without loss of generality,
we just regard the parallel hypersurface $\Phi_{r_0}(M)$ as $M$. Then, $\alpha=0$
and $\lambda_1<\lambda_2\leq-1\leq
\mu_2<\mu_1<0<\lambda_3\leq\lambda_4
\leq1\leq\mu_4\leq\mu_3$ holds on $M$.

In {\bf Case-1-1}, according to that $M_-$ is austere, it holds
$$
\tilde{\alpha}+\tilde{\mu}_3=\tilde{\lambda}_2+\tilde{\mu}_4=\tilde{\mu}_2+\tilde{\lambda}_4
=\tilde{\mu}_1+\tilde{\lambda}_3=0.
$$
Then by using \eqref{eqn:6.7as1} and $\alpha=0$, we get
\begin{equation}\label{eqn:6.4.1}
\frac{2\lambda_1}{\lambda_1^2-1}+\frac{1+\lambda_1\lambda_3}{\lambda_1-\lambda_3}=0,
\end{equation}
\begin{equation}\label{eqn:6.4.2}
\frac{\lambda_1+\lambda_2}{\lambda_1\lambda_2-1}+\frac{1+\lambda_1\lambda_4}{\lambda_1-\lambda_4}=0,
\end{equation}
\begin{equation}\label{eqn:6.4.3}
\frac{\lambda_1+\lambda_4}{\lambda_1\lambda_4-1}+\frac{1+\lambda_1\lambda_2}{\lambda_1-\lambda_2}=0,
\end{equation}
\begin{equation}\label{eqn:6.4.4}
\frac{4\lambda_1}{\lambda_1^2-1}+\frac{\lambda_1+\lambda_3}{\lambda_1\lambda_3-1}=0.
\end{equation}
By \eqref{eqn:6.4.2} and \eqref{eqn:6.4.4}, we have $\lambda_3=\frac{-\lambda_1(\lambda_1^2-5)}{5\lambda_1^2-1}$
and $((\lambda_1^2-1)\lambda_2-2\lambda_1)\lambda_4=1-\lambda_1(\lambda_1+2\lambda_2)$.
If $(\lambda_1^2-1)\lambda_2-2\lambda_1=0$, then $1-\lambda_1(\lambda_1+2\lambda_2)=-\frac{(1+\lambda_1^2)^2}{\lambda_1-1}\neq0$. So
$\lambda_4=\frac{1-\lambda_1(\lambda_1+2\lambda_2)}{(\lambda_1^2-1)\lambda_2-2\lambda_1}$. Then
\eqref{eqn:6.4.3} becomes to $\frac{4\lambda_1(\lambda_1^2-1)(\lambda_2^2+1)}
{(\lambda_1-\lambda_2)(-\lambda_2+\lambda_1(-3+\lambda_1^2+3\lambda_1\lambda_2))}=0$, which
deduces that $\lambda_1=-1$, it is a contradiction.

In {\bf Case-1-2}, according to that $M_-$ is austere, it holds
$$
\tilde{\lambda}_2+\tilde{\mu}_3=\tilde{\alpha}+\tilde{\mu}_4=\tilde{\mu}_2+\tilde{\lambda}_4
=\tilde{\mu}_1+\tilde{\lambda}_3=0.
$$
Then by using \eqref{eqn:6.7as1} and $\alpha=0$, it still holds \eqref{eqn:6.4.1},
\eqref{eqn:6.4.2} and
\begin{equation}\label{eqn:6.4.5}
\frac{4\lambda_1}{\lambda_1^2-1}+\frac{\lambda_1+\lambda_4}{\lambda_1\lambda_4-1}=0,
\end{equation}
\begin{equation}\label{eqn:6.4.6}
\frac{1+\lambda_1\lambda_2}{\lambda_1-\lambda_2}+\frac{\lambda_1+\lambda_3}{\lambda_1\lambda_3-1}=0.
\end{equation}
By \eqref{eqn:6.4.5} and \eqref{eqn:6.4.6}, we have $\lambda_4=-\frac{\lambda_1(\lambda_1^2-5)}{5\lambda_1^2-1}$ and $(2\lambda_1+(\lambda_1^2-1)\lambda_2)\lambda_3=1-\lambda_1^2+2\lambda_1\lambda_2$.
If $2\lambda_1+(\lambda_1^2-1)\lambda_2=0$, then $1-\lambda_1^2+2\lambda_1\lambda_2=-\frac{(1+\lambda_1^2)^2}{\lambda_1^2-1}\neq0$, which is a contradiction.
So $\lambda_3=\frac{1-\lambda_1^2+2\lambda_1\lambda_2}{2\lambda_1+(\lambda_1^2-1)\lambda_2}$.
Then \eqref{eqn:6.4.2} becomes to $(7\lambda_1-16\lambda_1^3+\lambda_1^5)\lambda_2=1-16\lambda_1^2+7\lambda_1^4$. If
$7\lambda_1-16\lambda_1^3+\lambda_1^5=0$, then by $\lambda_1<-1$, we have $\lambda_1=-\sqrt{8+\sqrt{57}}$,
which shows that $1-16\lambda_1^2+7\lambda_1^4=720+96\sqrt{57}\neq0$. So
$\lambda_2=\frac{1-16\lambda_1^2+7\lambda_1^4}{7\lambda_1-16\lambda_1^3+\lambda_1^5}$.
Then by using \eqref{eqn:6.4.1}, we can solve that $\lambda_1=-\sqrt{15-4\sqrt{10}+2\sqrt{96-30\sqrt{10}}}$
or $\lambda_1=-\sqrt{15+4\sqrt{10}+2\sqrt{96+30\sqrt{10}}}$. We substitute $\lambda_1=-\sqrt{15-4\sqrt{10}+2\sqrt{96-30\sqrt{10}}}$ into $\lambda_2$, then $\lambda_2\approx0.7389>0$, which is a contradiction. We substitute $\lambda_1=-\sqrt{15+4\sqrt{10}+2\sqrt{96+30\sqrt{10}}}$ into $\lambda_4$,
then $\lambda_4\approx1.3574$ and $\mu_4\approx0.7366$, which contradicts with $\lambda_4\leq\mu_4$.

In {\bf Case-1-3}, according to that $M_-$ is austere, it holds
$$
\tilde{\lambda}_2+\tilde{\mu}_3=\tilde{\mu}_2+\tilde{\mu}_4=\tilde{\alpha}+\tilde{\lambda}_4
=\tilde{\mu}_1+\tilde{\lambda}_3=0.
$$
Then by using \eqref{eqn:6.7as1} and $\alpha=0$, it still holds \eqref{eqn:6.4.1}, \eqref{eqn:6.4.6} and
\begin{equation}\label{eqn:6.4.7}
\frac{4\lambda_1}{\lambda_1^2-1}+\frac{\lambda_1\lambda_4+1}{\lambda_1-\lambda_4}=0,
\end{equation}
\begin{equation}\label{eqn:6.4.8}
\frac{\lambda_1+\lambda_2}{\lambda_1\lambda_2-1}+\frac{\lambda_1+\lambda_4}{\lambda_1\lambda_4-1}=0.
\end{equation}
By \eqref{eqn:6.4.1} and \eqref{eqn:6.4.7}, we have
$(\lambda_1(\lambda_1^2-3))\lambda_3=1-3\lambda_1^2$ and
$(\lambda_1(\lambda_1^2-5))\lambda_4=1-5\lambda_1^2$. If $\lambda_1=-\sqrt{3}$ or
$\lambda_1=-\sqrt{5}$, then $1-3\lambda_1^2\neq0$ and $1-5\lambda_1^2\neq0$.
So $\lambda_3=\frac{1-3\lambda_1^2}{\lambda_1(\lambda_1^2-3)}$ and
$\lambda_4=\frac{1-5\lambda_1^2}{\lambda_1(\lambda_1^2-5)}$. Then by \eqref{eqn:6.4.8}, we have
$(7\lambda_1-16\lambda_1^3+\lambda_1^5)\lambda_2=1-16\lambda_1^2+7\lambda_1^4$. If
$7\lambda_1-16\lambda_1^3+\lambda_1^5=0$, then $1-16\lambda_1^2+7\lambda_1^4\neq0$.
So $\lambda_2=\frac{1-16\lambda_1^2+7\lambda_1^4}{7\lambda_1-16\lambda_1^3+\lambda_1^5}$.
Now, from \eqref{eqn:6.4.6}, we can also solve that $\lambda_1=-\sqrt{15-4\sqrt{10}+2\sqrt{96-30\sqrt{10}}}$
or $\lambda_1=-\sqrt{15+4\sqrt{10}+2\sqrt{96+30\sqrt{10}}}$. We substitute $\lambda_1=-\sqrt{15-4\sqrt{10}+2\sqrt{96-30\sqrt{10}}}$ into $\lambda_2$, it holds
that $\lambda_2\approx0.7389>0$, which is a contradiction. If $\lambda_1=-\sqrt{15+4\sqrt{10}+2\sqrt{96+30\sqrt{10}}}$, then we can solve the values of
$\lambda_i$ and $\mu_i$, $1\leq i\leq4$. Now, we can further check the principal curvatures of $M_+$ by using
Proposition \ref{prop:4.3w}. It follows that the minimum principal curvature is approximately equal to $-1.6881$ and the maximum principal curvature is approximately equal to $4.1981$
on $M_+$, which contradicts with the fact that $M_+$ is austere.

In {\bf Case-1-4}, according to that $M_-$ is austere, it holds
$$
\tilde{\lambda}_2+\tilde{\mu}_3=\tilde{\mu}_2+\tilde{\mu}_4=\tilde{\mu}_1+\tilde{\lambda}_4
=\tilde{\alpha}+\tilde{\lambda}_3=0.
$$
Then by using \eqref{eqn:6.7as1} and $\alpha=0$, it holds \eqref{eqn:6.4.6}, \eqref{eqn:6.4.8} and
\begin{equation}\label{eqn:6.4.9}
\frac{4\lambda_1}{\lambda_1^2-1}+\frac{1+\lambda_1\lambda_3}{\lambda_1-\lambda_3}=0,
\end{equation}
\begin{equation}\label{eqn:6.4.10}
\frac{2\lambda_1}{\lambda_1^2-1}+\frac{1+\lambda_1\lambda_4}{\lambda_1-\lambda_4}=0.
\end{equation}
By \eqref{eqn:6.4.9} and \eqref{eqn:6.4.10}, we have $\lambda_3=\frac{1-5\lambda_1^2}{\lambda_1(\lambda_1^2-5)}$ and
$\lambda_4=\frac{1-3\lambda_1^2}{\lambda_1(\lambda_1^2-3)}$. Then by \eqref{eqn:6.4.8}, we have
$\lambda_2=\frac{1-10\lambda_1^2+5\lambda_1^4}{5\lambda_1-10\lambda_1^3+\lambda_1^5}$.
Now, from \eqref{eqn:6.4.6}, we can solve that $\lambda_1=-\sqrt{15-4\sqrt{10}+2\sqrt{96-30\sqrt{10}}}$
or $\lambda_1=-\sqrt{15+4\sqrt{10}+2\sqrt{96+30\sqrt{10}}}$. We substitute $\lambda_1=-\sqrt{15-4\sqrt{10}+2\sqrt{96-30\sqrt{10}}}$ into $\lambda_2$, it holds
that $\lambda_2\approx1.3533>0$, which is a contradiction. We substitute $\lambda_1=-\sqrt{15+4\sqrt{10}+2\sqrt{96+30\sqrt{10}}}$ into $\lambda_2$ and $\mu_2$, it holds
that the $\lambda_2\approx-0.7897$ and $\mu_2\approx-1.2661$, which contradicts with $\lambda_2\leq\mu_2$.

In {\bf Case-1}, if ${\rm dim}V_{\lambda_-}=2$, we assume that $\lambda_1=\lambda_2<\frac{\alpha-\sqrt{\alpha^2+4}}{2}<\mu_2=\mu_1
<\frac{\alpha}{2}<\lambda_3\leq\lambda_4\leq\frac{\alpha+\sqrt{\alpha^2+4}}{2}\leq\mu_4\leq\mu_3$.
By \eqref{eqn:6.7as1}, it holds $\tilde{\mu}_2=\tilde{\mu}_1
<\tilde{\lambda}_3\leq\tilde{\lambda}_4\leq\tilde{\mu}_4\leq\tilde{\mu}_3$ on $M_-$.

When $\alpha>0$, if $\lambda_1\leq\frac{-2-\sqrt{\alpha^2+4}}{\alpha}$, then
$\tilde{\alpha}\geq0$, $\tilde{\lambda}_3,\tilde{\lambda}_4,\tilde{\mu}_4,\tilde{\mu}_3>0$
on $M_-$. It contradicts with the fact that
$M_-$ is austere. So $\lambda_1>\frac{-2-\sqrt{\alpha^2+4}}{\alpha}$, then $\tilde{\alpha}<0$
on $M_-$. If $\alpha=0$, then it also holds $\tilde{\alpha}<0$ on $M_-$. Due that $\tilde{\mu}_2=\tilde{\mu}_1
<\tilde{\lambda}_3\leq\tilde{\lambda}_4\leq\tilde{\mu}_4\leq\tilde{\mu}_3$ and
$\tilde{\alpha}<0$ on $M_-$,
thus from the fact that $M_-$ is austere, we need to consider the following two subcases on $M_-$:

{\bf Case-1-5}: $\tilde{\alpha}\leq\tilde{\mu}_2=\tilde{\mu}_1
<\tilde{\lambda}_3\leq\tilde{\lambda}_4\leq\tilde{\mu}_4\leq\tilde{\mu}_3$;

{\bf Case-1-6}: $\tilde{\mu}_2=\tilde{\mu}_1\leq\tilde{\alpha}
<\tilde{\lambda}_3\leq\tilde{\lambda}_4\leq\tilde{\mu}_4\leq\tilde{\mu}_3$.

Since $\tilde{\alpha}<0$ on $M_-$, then by Proposition \ref{prop:4.1w}, there is a parallel hypersurface $\Phi_{r_0}(M)$ ($r_2<r_0\leq0$) of $M$ such that the Reeb function being $0$ on $\Phi_{r_0}(M)$.
For the convenience of calculation, without loss of generality,
we just regard the parallel hypersurface $\Phi_{r_0}(M)$ as $M$. Then, $\alpha=0$
and $\lambda_1=\lambda_2<-1<\mu_2=\mu_1<0<\lambda_3\leq\lambda_4
\leq1\leq\mu_4\leq\mu_3$ holds on $M$.

In {\bf Case-1-5}, according to that $M_-$ is austere, it holds
$$
\tilde{\alpha}+\tilde{\mu}_3=\tilde{\mu}_2+\tilde{\mu}_4=\tilde{\mu}_1+\tilde{\lambda}_4
=\tilde{\lambda}_3=0.
$$
Then by using \eqref{eqn:6.7as1}, $\alpha=0$ and $\lambda_2=\lambda_1$, it holds
$$
\frac{1+\lambda_1\lambda_3}{\lambda_1-\lambda_3}
=\frac{2\lambda_1}{\lambda_1^2-1}+\frac{1+\lambda_1\lambda_4}{\lambda_1-\lambda_4}
=\frac{2\lambda_1}{\lambda_1^2-1}+\frac{\lambda_1+\lambda_4}{-1+\lambda_1\lambda_4}
=\frac{4\lambda_1}{\lambda_1^2-1}+\frac{\lambda_1+\lambda_3}{-1+\lambda_1\lambda_3}=0.
$$
By direct calculation, with the use of $\lambda_1<-1$, above equations have no solution
to $\lambda_1,\lambda_3$ and $\lambda_4$. We get a contradiction.

In {\bf Case-1-6}, according to that $M_-$ is austere, it holds
$$
\tilde{\mu}_2+\tilde{\mu}_3=\tilde{\mu}_1+\tilde{\mu}_4=\tilde{\alpha}+\tilde{\lambda}_4
=\tilde{\lambda}_3=0.
$$
Then by using \eqref{eqn:6.7as1}, $\alpha=0$ and $\lambda_2=\lambda_1$, it holds
$$
\frac{1+\lambda_1\lambda_3}{\lambda_1-\lambda_3}
=\frac{4\lambda_1}{\lambda_1^2-1}+\frac{1+\lambda_1\lambda_4}{\lambda_1-\lambda_4}
=\frac{2\lambda_1}{\lambda_1^2-1}+\frac{\lambda_1+\lambda_4}{-1+\lambda_1\lambda_4}
=\frac{2\lambda_1}{\lambda_1^2-1}+\frac{\lambda_1+\lambda_3}{-1+\lambda_1\lambda_3}=0.
$$
By direct calculation, with the use of $\lambda_1<-1$, above equations have no solution
to $\lambda_1,\lambda_3$ and $\lambda_4$. We still get a contradiction.

In {\bf Case-2}, if $\alpha=0$ on $M$, then by
$\lambda_-<\frac{\alpha-\sqrt{\alpha^2+4}}{2}=-1$ and
${\rm dim}V_{\lambda_-}\leq3$, we know that the negative principal curvatures (counted with multiplicities) are at least four on $M_-$, and the positive principal curvatures (counted with multiplicities) are at most two on $M_-$. It contradicts with the fact that $M_-$ is austere. On the other hand, if there is a parallel hypersurface
between $M$ and $M_-$ such that its Reeb function being $0$,
then we can also get a contradiction in a similar way.
Thus, we always assume that the Reeb functions on $M$ and the parallel hypersurfaces
between $M$ and $M_-$ are positive in {\bf Case-2}. It follows that
$\tilde{\alpha}\geq0$ on $M_-$. We set $\mu_4$ to be the biggest principal curvature in $\sigma(\mathcal{Q})$
on $M$ in {\bf Case-2}.
So $\lambda_+=\mu_4$. By Proposition \ref{prop:4.3w}, on $M_+$, the principal curvatures are given by
\begin{equation}\label{eqn:6.7as1+}
0,\ \bar{\alpha}=\alpha+\frac{(4+\alpha^2){\mu_4}}{\mu_4^2-\alpha{\mu_4}-1},\
\bar{\lambda}_i=\frac{1+{\mu_4}\lambda_i}{{\mu_4}-\lambda_i},\
\bar{\mu}_i=\frac{1+{\mu_4}\mu_i}{{\mu_4}-\mu_i},\ \lambda_i\neq\mu_4,\ \mu_i\neq\mu_4.
\end{equation}
Now, we need to consider the following four subcases on $M$:

{\bf Case-2-1}: There are $0$ principal curvatures
(counted with multiplicities) belonging to $(0,\frac{\alpha}{2})$ in $\sigma(\mathcal{Q})$;

{\bf Case-2-2}: There are $1$ principal curvatures
(counted with multiplicities) belonging to $(0,\frac{\alpha}{2})$ in $\sigma(\mathcal{Q})$;

{\bf Case-2-3}: There are $2$ principal curvatures
(counted with multiplicities) belonging to $(0,\frac{\alpha}{2})$ in $\sigma(\mathcal{Q})$;

{\bf Case-2-4}: There are $3$ principal curvatures
(counted with multiplicities) belonging to $(0,\frac{\alpha}{2})$ in $\sigma(\mathcal{Q})$.

In {\bf Case-2-1}, we have ${\rm dim}V_{\lambda_1}\leq3$.
If ${\rm dim}V_{\lambda_1}=1$, then there are at least five negative principal curvatures (counted with multiplicities) and
at most three positive principal curvatures (counted with multiplicities) on $M_-$, which contradicts with that $M_-$ is austere.
If ${\rm dim}V_{\lambda_1}=2$, then there are at least four negative principal curvatures (counted with multiplicities) and
at most three positive principal curvatures (counted with multiplicities) on $M_-$, which contradicts with that $M_-$ is austere.
If ${\rm dim}V_{\lambda_1}=3$, then we assume
$\lambda_1=\lambda_2=\lambda_3<\frac{\alpha-\sqrt{\alpha^2+4}}{2}<\mu_3=\mu_2=\mu_1\leq0
<\lambda_4\leq\mu_4$. By \eqref{eqn:6.7as1}, $M_-$ is austere if and only if
$\tilde{\mu}_3=\tilde{\mu}_2=\tilde{\mu}_1<0<\tilde{\alpha}=\tilde{\lambda}_4=\tilde{\mu}_4$ and
$\tilde{\mu}_1+\tilde{\alpha}=0$. So $\lambda_4=\mu_4=\frac{\alpha+\sqrt{\alpha^2+4}}{2}$.
From $\tilde{\mu}_1+\tilde{\alpha}=0$, we have $\alpha=\frac{-4\lambda_1}{\lambda_1^2-1}$.
Substitute it into $\tilde{\alpha}-\tilde{\lambda}_4=0$, it becomes to $1=0$,
which is a contradiction.

In {\bf Case-2-2}, we need consider the following subcases:

{\bf Case-2-2-1}: $\lambda_1<\lambda_2<\lambda_3<\frac{\alpha-\sqrt{\alpha^2+4}}{2}<\mu_3<\mu_2\leq0<\mu_1
<\frac{\alpha}{2}<\lambda_4<\frac{\alpha+\sqrt{\alpha^2+4}}{2}<\mu_4$;

{\bf Case-2-2-2}: $\lambda_1<\lambda_2<\lambda_3=\frac{\alpha-\sqrt{\alpha^2+4}}{2}=\mu_3<\mu_2\leq0<\mu_1
<\frac{\alpha}{2}<\lambda_4<\frac{\alpha+\sqrt{\alpha^2+4}}{2}<\mu_4$;

{\bf Case-2-2-3}: $\lambda_1<\lambda_2=\lambda_3<\frac{\alpha-\sqrt{\alpha^2+4}}{2}<\mu_3=\mu_2\leq0
<\mu_1<\frac{\alpha}{2}<\lambda_4
<\frac{\alpha+\sqrt{\alpha^2+4}}{2}<\mu_4$;

{\bf Case-2-2-4}: $\lambda_1<\lambda_2=\lambda_3=\frac{\alpha-\sqrt{\alpha^2+4}}{2}=\mu_3=\mu_2\leq0
<\mu_1<\frac{\alpha}{2}<\lambda_4
<\frac{\alpha+\sqrt{\alpha^2+4}}{2}<\mu_4$;

{\bf Case-2-2-5}: $\lambda_1<\lambda_2=\lambda_3=\frac{\alpha-\sqrt{\alpha^2+4}}{2}=\mu_3=\mu_2\leq0
<\mu_1<\frac{\alpha}{2}
<\lambda_4=\frac{\alpha+\sqrt{\alpha^2+4}}{2}=\mu_4$.

{\bf Case-2-2-6}: $\lambda_1<\lambda_2<\lambda_3<\frac{\alpha-\sqrt{\alpha^2+4}}{2}<\mu_3<\mu_2\leq0
<\mu_1<\frac{\alpha}{2}
<\lambda_4=\frac{\alpha+\sqrt{\alpha^2+4}}{2}=\mu_4$;

{\bf Case-2-2-7}: $\lambda_1<\lambda_2<\lambda_3=\frac{\alpha-\sqrt{\alpha^2+4}}{2}=\mu_3<\mu_2\leq0<\mu_1
<\frac{\alpha}{2}<\lambda_4=\frac{\alpha+\sqrt{\alpha^2+4}}{2}=\mu_4$;

{\bf Case-2-2-8}: $\lambda_1<\lambda_2=\lambda_3<\frac{\alpha-\sqrt{\alpha^2+4}}{2}<\mu_3=\mu_2\leq0
<\mu_1<\frac{\alpha}{2}
<\lambda_4=\frac{\alpha+\sqrt{\alpha^2+4}}{2}=\mu_4$.

In {\bf Case-2-2-1}, due that $M_+$ is austere, we must have $\bar{\alpha}>0$ and $\bar{\mu}_2>0$ on $M_+$. Then by Proposition \ref{prop:4.1w}, there is a parallel hypersurface $\Phi_{r_0}(M)$
$(0\leq r_0<r_1)$ of $M$ such that the corresponding principal curvature $\mu_2(r_0)=\frac{\sin(r_0)+\mu_2\cos(r_0)}{\cos(r_0)-\mu_2\sin(r_0)}=0$.
For the convenience of calculation, without loss of generality,
we just regard the parallel hypersurface $\Phi_{r_0}(M)$ as $M$. Then, it holds $\lambda_2=-\frac{2}{\alpha}$, $\mu_2=0$, $\alpha>0$,
and $\lambda_1<\lambda_2=-\frac{2}{\alpha}<\lambda_3<\frac{\alpha-\sqrt{\alpha^2+4}}{2}<\mu_3
<\mu_2=0<\mu_1
<\frac{\alpha}{2}<\lambda_4<\frac{\alpha+\sqrt{\alpha^2+4}}{2}<\mu_4$ on $M$.
Now, we further have four subcases on $M_+$:

{\bf Case-2-2-1-1}: $\bar{\lambda}_1<\bar{\lambda}_2<\bar{\lambda}_3<\bar{\mu}_3<\bar{\alpha}
<\bar{\mu}_2<\bar{\mu}_1<\bar{\lambda}_4$;

{\bf Case-2-2-1-2}: $\bar{\lambda}_1<\bar{\lambda}_2<\bar{\lambda}_3<\bar{\mu}_3
<\bar{\mu}_2<\bar{\alpha}<\bar{\mu}_1<\bar{\lambda}_4$;

{\bf Case-2-2-1-3}: $\bar{\lambda}_1<\bar{\lambda}_2<\bar{\lambda}_3<\bar{\mu}_3
<\bar{\mu}_2<\bar{\mu}_1<\bar{\alpha}<\bar{\lambda}_4$;

{\bf Case-2-2-1-4}: $\bar{\lambda}_1<\bar{\lambda}_2<\bar{\lambda}_3<\bar{\mu}_3
<\bar{\mu}_2<\bar{\mu}_1<\bar{\lambda}_4<\bar{\alpha}$.

In {\bf Case-2-2-1-1}, from $\bar{\lambda}_3+\bar{\mu}_2=\bar{\lambda}_2+\bar{\mu}_1=0$
and $\lambda_2=-\frac{2}{\alpha}$, we get
$\lambda_1=-\frac{2 \lambda_4}{-1+\lambda_4^2}$ and
$\lambda_3=\frac{2(\alpha-2\lambda_4)(2+\alpha\lambda_4)}
{4+8\alpha\lambda_4-4\lambda_4^2+\alpha^2(-1+\lambda_4^2)}$.
If $\lambda_4=\sqrt{5\pm2\sqrt{5}}$, then it follows from $\bar{\lambda}_1+\bar{\lambda}_4=0$ that $\alpha=0$, which contradicts with $\alpha>0$ on $M$.
Then, we substitute the expressions of $\lambda_1$ and $\lambda_3$ into
$\bar{\lambda}_1+\bar{\lambda}_4=0$, we have
$\alpha=-\frac{2(1-10\lambda_4^2+5\lambda_4^2)\pm2(1+\lambda_4^2)^2
\sqrt{1+\lambda_4^2}}{\lambda_4 (5-10\lambda_4^2+\lambda_4^4)}$.
If $\alpha=-\frac{2(1-10\lambda_4^2+5\lambda_4^2)+2(1+\lambda_4^2)^2
\sqrt{1+\lambda_4^2}}{\lambda_4 (5-10\lambda_4^2+\lambda_4^4)}$,
by $\bar{\mu}_3+\bar{\alpha}=0$, we have $5-\lambda_4^2 + 4 \sqrt{1 +\lambda_4^2}=0$,
which deduces that $\lambda_4=2\sqrt{2}+\sqrt{5}$. But in this situation, we have
$\lambda_2\approx0.3071$, which contradicts with $\lambda_2<0$.
If $\alpha=-\frac{2(1-10\lambda_4^2+5\lambda_4^2)-2(1+\lambda_4^2)^2
\sqrt{1+\lambda_4^2}}{\lambda_4 (5-10\lambda_4^2+\lambda_4^4)}$,
by $\bar{\mu}_3+\bar{\alpha}=0$, we have $-5+\lambda_4^2 + 4 \sqrt{1 +\lambda_4^2}=0$,
which deduces that $\lambda_4=\sqrt{13-4\sqrt{10}}$. But in this situation, we have
$\lambda_1\approx1.8251$, which contradicts with $\lambda_1<0$.

In {\bf Case-2-2-1-2}, from $\bar{\mu}_3+\bar{\mu}_2=\bar{\lambda}_2+\bar{\mu}_1=0$
and $\lambda_2=-\frac{2}{\alpha}$, we get
$\lambda_1=-\frac{2 \lambda_4}{-1 + \lambda_4^2}$ and
$\alpha=\frac{-2 + 2\lambda_4 (-2\lambda_3 +\lambda_4)}{2 \lambda_4 + \lambda_3 (-1 +\lambda_4^2)}$. Substitute them into
$\bar{\lambda}_1+\bar{\lambda}_4=0$, we have
$\lambda_4=\frac{2\lambda_3}{\lambda_3^2-1}$. So $\bar{\lambda}_3+\bar{\alpha}=0$
becomes to $9-14\lambda_3^2 +\lambda_3^4=0$, which deduces that $\lambda_3=\sqrt{2}-\sqrt{5}$
or $\lambda_3=-\sqrt{7+2\sqrt{10}}$.
If $\lambda_3=\sqrt{2}-\sqrt{5}$, then $\lambda_2\approx0.3071$, which contradicts with $\lambda_2<0$.
If $\lambda_3=-\sqrt{7+2\sqrt{10}}$, then $\lambda_2\approx0.2381$, which contradicts with $\lambda_2<0$.

In {\bf Case-2-2-1-3}, from $\bar{\mu}_3+\bar{\mu}_2=\bar{\lambda}_2+\bar{\alpha}=0$
and $\lambda_2=-\frac{2}{\alpha}$, we get
$\lambda_3=\frac{-2 + 2\lambda_4 (-\alpha+ \lambda_4)}{4\lambda_4 + \alpha(-1 +\lambda_4^2)}$ and
$\alpha=\frac{1 - 5 \lambda_4^2}{\lambda_4 (-2 +\lambda_4^2)}$. Substitute them into
$\bar{\lambda}_3+\bar{\mu}_1=0$, we have
$\lambda_1=\frac{4\lambda_4}{4\lambda_4^2-1}$. So $\bar{\lambda}_1+\bar{\lambda}_4=0$
becomes to $9-56\lambda_4^2 +16\lambda_4^4=0$, which deduces that
$\lambda_4=\sqrt{\frac{7}{4}+\sqrt{\frac{5}{2}}}$
or $\lambda_4=\frac{-\sqrt{2}+\sqrt{5}}{2}$. If $\lambda_4=\sqrt{\frac{7}{4}+\sqrt{\frac{5}{2}}}$,
then $\lambda_1\approx0.5923$, which contradicts with $\lambda_1<0$.
If $\lambda_4=\frac{-\sqrt{2}+\sqrt{5}}{2}$,
then $\lambda_2\approx9.6659$, which contradicts with $\lambda_2<0$.

In {\bf Case-2-2-1-4}, from $\bar{\mu}_3+\bar{\mu}_2=\bar{\lambda}_2+\bar{\lambda}_4=0$
and $\lambda_2=-\frac{2}{\alpha}$, we get
$\lambda_3=\frac{-2 + 2\lambda_4 (-\alpha+ \lambda_4)}{4\lambda_4 + \alpha(-1 +\lambda_4^2)}$ and
$\alpha=\frac{2-6\lambda_4^2}{\lambda_4 (-3 +\lambda_4^2)}$. Substitute them into
$\bar{\lambda}_1+\bar{\alpha}=0$, we have
$\lambda_1=\frac{\lambda_4(\lambda_4^2-5)}{4\lambda_4^2-2}$. So $\bar{\lambda}_3+\bar{\mu}_1=0$
becomes to $9-14\lambda_4^2 +\lambda_4^4=0$, which deduces that
$\lambda_4=\pm\sqrt{2}+\sqrt{5}$. If $\lambda_4=-\sqrt{2}+\sqrt{5}$,
then by Proposition \ref{prop:4.4w}, on $M_-$, the minimum principal curvature is approximately equal to $-3.2553$,
the maximum principal curvature is approximately equal to
$2.4335$, which contradicts with the fact that $M_-$ is austere.
If $\lambda_4=\sqrt{2}+\sqrt{5}$,
then $\lambda_1\approx0.5923$, which contradicts with $\lambda_1<0$.

In {\bf Case-2-2-2}, due that $M_+$ is austere, we must have $\bar{\alpha}>0$ and $\bar{\mu}_2>0$ on $M_+$. Then by Proposition \ref{prop:4.1w}, there is a parallel hypersurface $\Phi_{r_0}(M)$
($0\leq r_0<r_1$) of $M$ such that the corresponding principal curvature $\mu_2(r_0)=\frac{\sin(r_0)+\mu_2\cos(r_0)}{\cos(r_0)-\mu_2\sin(r_0)}=0$.
For the convenience of calculation, without loss of generality,
we just regard the parallel hypersurface $\Phi_{r_0}(M)$ as $M$. Then, it holds $\lambda_2=-\frac{2}{\alpha}$, $\mu_2=0$, $\alpha>0$,
and $\lambda_1<\lambda_2=-\frac{2}{\alpha}<\lambda_3=\frac{\alpha-\sqrt{\alpha^2+4}}{2}=\mu_3<\mu_2=0<\mu_1
<\frac{\alpha}{2}<\lambda_4<\frac{\alpha+\sqrt{\alpha^2+4}}{2}<\mu_4$
on $M$. Due that $M_+$ is austere and $\bar{\lambda}_3=\bar{\mu}_3$, we have
$\bar{\mu}_3+\bar{\mu}_2=\bar{\lambda}_2+\bar{\mu}_1=\bar{\lambda}_1+\bar{\lambda}_4
=\bar{\lambda}_3+\bar{\alpha}=0$ on $M_+$.
Then, following the same calculation with {\bf Case-2-2-1-2}, we can get
a contradiction.

In {\bf Case-2-2-3}, {\bf Case-2-2-4} and {\bf Case-2-2-5}, according to the multiplicities of principal curvature $\tilde{\lambda}_2$ and
$\tilde{\mu}_2$ on $M_-$, then for any relationship between the values of $\tilde{\alpha}$ and
$\tilde{\lambda}_i,\tilde{\mu}_i$, we know that $M_-$ can not be austere.
We get a contradiction.

In {\bf Case-2-2-6}, according to $\tilde{\mu}_2<0$ and $\tilde{\lambda}_4=\tilde{\mu}_4$ on $M_-$,
then for any relationship between the values of $\tilde{\alpha}$ and
$\tilde{\lambda}_i,\tilde{\mu}_i$, $M_-$ can not be austere. We get a contradiction.

In {\bf Case-2-2-7}, due that $M_+$ is austere, we must have $\bar{\alpha}>0$ and $\bar{\mu}_2>0$ on $M_+$. Then by Proposition \ref{prop:4.1w}, there is a parallel hypersurface $\Phi_{r_0}(M)$ ($0\leq r_0<r_1$) of $M$ such that the corresponding principal curvature $\mu_2(r_0)=\frac{\sin(r_0)+\mu_2\cos(r_0)}{\cos(r_0)-\mu_2\sin(r_0)}=0$.
For the convenience of calculation, without loss of generality,
we just regard the parallel hypersurface $\Phi_{r_0}(M)$ as $M$. Then, it holds $\lambda_2=-\frac{2}{\alpha}$, $\mu_2=0$, $\alpha>0$,
and $\lambda_1<\lambda_2=-\frac{2}{\alpha}<\lambda_3=\frac{\alpha-\sqrt{\alpha^2+4}}{2}=\mu_3<\mu_2=0
<\mu_1<\frac{\alpha}{2}<\lambda_4=\frac{\alpha+\sqrt{\alpha^2+4}}{2}=\mu_4$
on $M$.
Now, due that $M_-$ is also austere, it must hold   $\tilde{\lambda}_2<\tilde{\lambda}_3=\tilde{\mu}_3<\tilde{\mu}_2
<\tilde{\mu}_1<\tilde{\lambda}_4=\tilde{\mu}_4<\tilde{\alpha}$ on $M_-$. From $\tilde{\mu}_2+\tilde{\mu}_1=0$,
we have $\alpha=\frac{2-6\lambda_1^2}{\lambda_1(\lambda_1^2-3)}$. Substitute it into
$\tilde{\lambda}_3+\tilde{\lambda}_4=0$, we have $\frac{2}{\lambda_1}=0$, which is a
contradiction.

In {\bf Case-2-2-8}, due that $M_+$ is austere, we must have $0<\bar{\mu}_3=\bar{\mu}_2$ on $M_+$. Then by Proposition \ref{prop:4.1w}, there is a parallel hypersurface $\Phi_{r_0}(M)$ ($0\leq r_0<r_1$) of $M$ such that the corresponding principal curvature $\mu_2(r_0)=\frac{\sin(r_0)+\mu_2\cos(r_0)}{\cos(r_0)-\mu_2\sin(r_0)}=0$.
For the convenience of calculation, without loss of generality,
we just regard the parallel hypersurface $\Phi_{r_0}(M)$ as $M$. Then, it holds $\lambda_2=-\frac{2}{\alpha}$, $\mu_2=0$, $\alpha>0$,
and $\lambda_1<\lambda_2=-\frac{2}{\alpha}=\lambda_3<\frac{\alpha-\sqrt{\alpha^2+4}}{2}<\mu_3=\mu_2=0
<\mu_1<\frac{\alpha}{2}<\lambda_4=\frac{\alpha+\sqrt{\alpha^2+4}}{2}=\mu_4$
on $M$. Now, due that $M_-$ is also austere, it holds $\tilde{\lambda}_2=\tilde{\lambda}_3<\tilde{\mu}_3=\tilde{\mu}_2
<\tilde{\mu}_1=\tilde{\alpha}<\tilde{\lambda}_4=\tilde{\mu}_4$. From $\tilde{\alpha}-\tilde{\mu}_1=0$,
we have $\alpha=\frac{-4\lambda_1}{\lambda_1^2-1}$. Substitute it into
$\tilde{\mu}_2+\tilde{\alpha}=0$, we have $\frac{1}{\lambda_1}=0$, which is a
contradiction.

In {\bf Case-2-3}, if $\lambda_+=\frac{\alpha+\sqrt{\alpha^2+4}}{2}$, we have the following four cases:

{\bf Case-2-3-1}: $\lambda_1<\lambda_2<\lambda_3<\mu_3\leq0<\mu_2<\mu_1<\frac{\alpha}{2}<\lambda_4
=\frac{\alpha+\sqrt{\alpha^2+4}}{2}=\mu_4$;

{\bf Case-2-3-2}: $\lambda_1<\lambda_2<\lambda_3=\frac{\alpha-\sqrt{\alpha^2+4}}{2}
=\mu_3<0<\mu_2<\mu_1<\frac{\alpha}{2}<\lambda_4
=\frac{\alpha+\sqrt{\alpha^2+4}}{2}=\mu_4$;

{\bf Case-2-3-3}: $\lambda_1=\lambda_2<\lambda_3<\mu_3\leq0<\mu_2=\mu_1<\frac{\alpha}{2}<\lambda_4
=\frac{\alpha+\sqrt{\alpha^2+4}}{2}=\mu_4$;

{\bf Case-2-3-4}: $\lambda_1=\lambda_2<\lambda_3=\frac{\alpha-\sqrt{\alpha^2+4}}{2}
=\mu_3<0<\mu_2=\mu_1<\frac{\alpha}{2}<\lambda_4
=\frac{\alpha+\sqrt{\alpha^2+4}}{2}=\mu_4$.

In {\bf Case-2-3-1}, due that $M_-$ is austere, we have
$\tilde{\alpha}=\tilde{\lambda}_2<\tilde{\lambda}_3<\tilde{\mu}_3<\tilde{\mu}_2<\tilde{\mu}_1
<\tilde{\lambda}_4=\tilde{\mu}_4$. From $\tilde{\alpha}-\tilde{\lambda}_2=\tilde{\lambda}_3+\tilde{\mu}_1=0$, we have
$\lambda_2=\frac{1+\lambda_1^2 (3 +\alpha\lambda_1)}{-\alpha+ 3\lambda_1 +\lambda_1^3}$
and $\lambda_3=\frac{2 -\lambda_1 (6\lambda_1 +\alpha(-3 +\lambda_1^2))}{\alpha- 3\alpha\lambda_1^2 + 2\lambda_1 (-3 +\lambda_1^2)}$. Substitute them into $\tilde{\alpha}+\tilde{\mu}_4=0$,
we have $\lambda_1=\frac{-6 \pm\sqrt{32 + 8 \alpha^2}}{3\alpha+\sqrt{4 +\alpha^2}}$.
For $\lambda_1=\frac{-6\pm\sqrt{32 + 8 \alpha^2}}{3\alpha+\sqrt{4 +\alpha^2}}$, then
$\tilde{\mu}_2+\tilde{\mu}_3=0$ becomes to $\frac{\pm13}{4\sqrt{2}}=0$,
which is a contradiction.

In {\bf Case-2-3-2}, due that $M_-$ is austere, we have
$\tilde{\lambda}_2<\tilde{\lambda}_3=\tilde{\mu}_3<\tilde{\mu}_2<\tilde{\mu}_1
<\tilde{\lambda}_4=\tilde{\mu}_4<\tilde{\alpha}$. From $\tilde{\lambda}_3+\tilde{\lambda}_4=0$,
we have $\alpha=\frac{-4\lambda_1}{\lambda_1^2-1}$. Substitute it into $\tilde{\mu}_1+\tilde{\mu}_2=0$,
we have $\frac{\lambda_1-\lambda_2}{1+\lambda_1\lambda_2}=0$, which means $\lambda_1=\lambda_2$.
It is a contradiction.

In {\bf Case-2-3-3}, according to that the multiplicities of principal curvatures $\tilde{\mu}_1$ and
$\tilde{\mu}_4$ are all $2$ on $M_-$, then for any relationship between the values of $\tilde{\alpha}$ and
$\tilde{\lambda}_i,\tilde{\mu}_i$, $M_-$ can not be austere. It is a contradiction.

In {\bf Case-2-3-4}, due that $M_-$ is austere, we have
$\tilde{\lambda}_3=\tilde{\mu}_3<\tilde{\alpha}=\tilde{\mu}_2=\tilde{\mu}_1
<\tilde{\lambda}_4=\tilde{\mu}_4$. From $\tilde{\alpha}=\tilde{\lambda}_3+\tilde{\lambda}_4=0$,
by direct calculation, we can have
$\lambda_1=\lambda_2=\frac{-2-\sqrt{\alpha^2+4}}{\alpha}$,
$\mu_1=\mu_2=\frac{-2+\sqrt{\alpha^2+4}}{\alpha}$,
$\lambda_3=\mu_3=\frac{\alpha-\sqrt{\alpha^2+4}}{2}$,
$\lambda_4=\mu_4=\frac{\alpha+\sqrt{\alpha^2+4}}{2}$.
It follows from (3) of Lemma \ref{lemma:5.1}, we have that
$A(V_{\frac{-2-\sqrt{\alpha^2+4}}{\alpha}}\oplus V_{\frac{-2+\sqrt{\alpha^2+4}}{\alpha}})
=V_{\frac{\alpha-\sqrt{\alpha^2+4}}{2}}\oplus V_{\frac{\alpha+\sqrt{\alpha^2+4}}{2}}$.
By checking the relationship between the principal curvatures, and using
Proposition \ref{prop:E3P}, we find that the tube around the equivariant embedding of the $9$-dimensional homogeneous space $Sp_2Sp_1/SO_2Sp_1Sp_1$ in $Q^6$ is contained in this {\bf Case-2-3-4}.

In {\bf Case-2-3}, if $\lambda_+>\frac{\alpha+\sqrt{\alpha^2+4}}{2}$, we consider the following four subcases:

{\bf Case-2-3-5}: $\lambda_1<\lambda_2<\lambda_3<\frac{\alpha-\sqrt{\alpha^2+4}}{2}
<\mu_3\leq0<\mu_2<\mu_1<\frac{\alpha}{2}
<\lambda_4<\frac{\alpha+\sqrt{\alpha^2+4}}{2}<\mu_4$;

{\bf Case-2-3-6}: $\lambda_1<\lambda_2<\lambda_3=\frac{\alpha-\sqrt{\alpha^2+4}}{2}
=\mu_3<0<\mu_2<\mu_1<\frac{\alpha}{2}
<\lambda_4<\frac{\alpha+\sqrt{\alpha^2+4}}{2}<\mu_4$;

{\bf Case-2-3-7}: $\lambda_1=\lambda_2<\lambda_3<\frac{\alpha-\sqrt{\alpha^2+4}}{2}
<\mu_3\leq0<\mu_2=\mu_1<\frac{\alpha}{2}
<\lambda_4<\frac{\alpha+\sqrt{\alpha^2+4}}{2}<\mu_4$;

{\bf Case-2-3-8}: $\lambda_1=\lambda_2<\lambda_3=\frac{\alpha-\sqrt{\alpha^2+4}}{2}
=\mu_3<0<\mu_2=\mu_1<\frac{\alpha}{2}
<\lambda_4<\frac{\alpha+\sqrt{\alpha^2+4}}{2}<\mu_4$.

In {\bf Case-2-3-5}, we have $\tilde{\alpha}>0$ on $M_-$. In fact, if
$\tilde{\alpha}=0$ on $M_-$, then $\lambda_1=\frac{-2-\sqrt{\alpha^2+4}}{\alpha}$
and $\mu_1=\frac{-2+\sqrt{\alpha^2+4}}{\alpha}$. It follows that
there are four negative principal curvatures and two positive principal curvatures on $M_-$, which
contradicts with that $M_-$ is austere. Thus, according to that $M_-$ is austere,
we need to consider the following subcases on $M_-$:

%{\bf XCase-2-3-5-1}: $\tilde{\lambda}_2<\tilde{\lambda}_3<\tilde{\mu}_3<\tilde{\alpha}=\tilde{\mu}_2
%<\tilde{\mu}_1<\tilde{\lambda}_4<\tilde{\mu}_4$

{\bf Case-2-3-5-1}: $\tilde{\lambda}_2<\tilde{\lambda}_3<\tilde{\mu}_3<\tilde{\mu}_2
<\tilde{\alpha}<\tilde{\mu}_1<\tilde{\lambda}_4<\tilde{\mu}_4$;

{\bf Case-2-3-5-2}: $\tilde{\lambda}_2<\tilde{\lambda}_3<\tilde{\mu}_3<\tilde{\mu}_2
<\tilde{\mu}_1<\tilde{\alpha}<\tilde{\lambda}_4<\tilde{\mu}_4$;

{\bf Case-2-3-5-3}: $\tilde{\lambda}_2<\tilde{\lambda}_3<\tilde{\mu}_3<\tilde{\mu}_2
<\tilde{\mu}_1<\tilde{\lambda}_4<\tilde{\alpha}<\tilde{\mu}_4$;

{\bf Case-2-3-5-4}: $\tilde{\lambda}_2<\tilde{\lambda}_3<\tilde{\mu}_3<\tilde{\mu}_2
<\tilde{\mu}_1<\tilde{\lambda}_4<\tilde{\mu}_4<\tilde{\alpha}$.

%For {\bf XCase-2-3-5-1}, by $\tilde{\alpha}=\tilde{\lambda}_3+\tilde{\lambda}_4=0$, we have
%$\alpha=-\frac{4\lambda_1}{\lambda_1^2-1}$ and
%$\lambda_4=\frac{\lambda_3-\lambda_1(\lambda_1\lambda_3+2)}{\lambda_1^2-1-2\lambda_1\lambda_3}$.
%Substitute them into $\tilde{\mu}_2=0$, we have $\frac{\lambda_1-\lambda_2}{1+\lambda_1\lambda_2}=0$.
%It is a contradiction.

In {\bf Case-2-3-5-1}--{\bf Case-2-3-5-4}, we all have $\tilde{\mu}_2<0$ on $M_-$.
Then by Proposition \ref{prop:4.1w}, there is a parallel hypersurface $\Phi_{r_0}(M)$
($r_2<r_0<0$) of $M$ such that the corresponding principal curvature $\mu_2(r_0)=\frac{\sin(r_0)+\mu_2\cos(r_0)}{\cos(r_0)-\mu_2\sin(r_0)}=0$.
Then, it holds $\lambda_2(r_0)=-\frac{2}{\alpha(r_0)}$, $\mu_2(r_0)=0$, $\alpha(r_0)>0$,
and $\lambda_1(r_0)<\lambda_2(r_0)=-\frac{2}{\alpha(r_0)}<\lambda_3(r_0)
<\mu_3(r_0)<0=\mu_2(r_0)<\mu_1(r_0)<\frac{\alpha(r_0)}{2}
<\lambda_4(r_0)<\frac{\alpha(r_0)+\sqrt{\alpha^2(r_0)+4}}{2}<\mu_4(r_0)$ on $\Phi_{r_0}(M)$.
It follows that $\Phi_{r_0}(M)$ belongs to the {\bf Case-2-2-1}, which does not occur.

In {\bf Case-2-3-6}, due that $M_-$ is austere, it must hold
$\tilde{\lambda}_2<\tilde{\lambda}_3=\tilde{\mu}_3<\tilde{\mu}_2<0<\tilde{\mu}_1<\tilde{\alpha}
=\tilde{\lambda}_4<\tilde{\mu}_4$. Then by Proposition \ref{prop:4.1w}, there is a parallel hypersurface $\Phi_{r_0}(M)$ ($r_2<r_0<0$) of $M$ such that the corresponding principal curvature $\mu_2(r_0)=\frac{\sin(r_0)+\mu_2\cos(r_0)}{\cos(r_0)-\mu_2\sin(r_0)}=0$.
Then, it holds $\alpha(r_0)>0$
and $\lambda_1(r_0)<\lambda_2(r_0)=-\frac{2}{\alpha(r_0)}<\lambda_3(r_0)
=\tfrac{\alpha(r_0)-\sqrt{\alpha^2(r_0)+4}}{2}=\mu_3(r_0)<0=\mu_2(r_0)<\mu_1(r_0)<\frac{\alpha(r_0)}{2}
<\lambda_4(r_0)<\tfrac{\alpha(r_0)+\sqrt{\alpha^2(r_0)+4}}{2}<\mu_4(r_0)$ on $\Phi_{r_0}(M)$.
It follows that $\Phi_{r_0}(M)$ belongs to the {\bf Case-2-2-2}, which does not occur.

%So by Proposition \ref{prop:4.1w}, there is a parallel hypersurface $\Phi_r(M)$ such that the corresponding principal curvature $\mu_2(r)=\frac{\sin(r)+\mu_2\cos(r)}{\cos(r)-\mu_2\sin(r)}=0$.
%So without loss of generality, we assume that the initial hypersurface is $\Phi_r(M)$.
%Still denote $\lambda_i,\mu_i$ and $\alpha$ the corresponding principal curvatures on $\Phi_r(M)$.
%Then on $\Phi_r(M)$, it holds $\lambda_2=-\frac{2}{\alpha}$, $\mu_2=0$ and $\alpha>0$.
%Now, {\bf XCase-2-3-6} on $M$ is equivalent to consider
%$\lambda_1<\lambda_2=-\frac{2}{\alpha}<\lambda_3<\mu_3<0=\mu_2<\mu_1<\frac{\alpha}{2}
%<\lambda_4<\frac{\alpha+\sqrt{\alpha^2+4}}{2}<\mu_4$
%on $\Phi_r(M)$.

%In the following, by $\tilde{\mu}_2+\tilde{\mu}_1=0$, we have $\alpha=\frac{2-6\lambda_1^2}{\lambda_1(\lambda_1^2-3)}$. Substitute it in $\tilde{\alpha}-\tilde{\lambda}_4=0$, we have $\lambda_4=\frac{-3\lambda_1}{\lambda_1^2-2}$.
%Further, substitute them into $\tilde{\lambda}_2+\tilde{\mu}_4=0$, we have
%$\frac{\lambda_1^2+5}{2\lambda_1}=0$, which is a contradiction.

In {\bf Case-2-3-7}, according to the multiplicities of principal curvature $\bar{\lambda}_1$ and
$\bar{\mu}_1$ on $M_+$, then for any relationship between the values of $\bar{\alpha}$ and
$\bar{\lambda}_i,\bar{\mu}_i$, $M_+$ can not be austere.

In {\bf Case-2-3-8}, according to the multiplicities of principal curvature $\tilde{\lambda}_3$ and
$\tilde{\mu}_1$ on $M_-$, then for any relationship between the values of
$\tilde{\alpha}$ and
$\tilde{\lambda}_i,\tilde{\mu}_i$, $M_-$ can not be austere.

In {\bf Case-2-4}, if $\lambda_+>\frac{\alpha+\sqrt{\alpha^2+4}}{2}$, then there will be at least five positive principal curvatures (counted with multiplicities)
and at most three negative principal curvatures (counted with multiplicities) on $M_+$, which contradicts with that $M_+$ is austere.
So it must hold $\lambda_+=\frac{\alpha+\sqrt{\alpha^2+4}}{2}$ on $M$.
%Assume that
%$\lambda_1\leq\lambda_2\leq\lambda_3\leq\mu_3\leq\mu_2\leq\mu_1<\lambda_4=\mu_4
%=\frac{\alpha+\sqrt{\alpha^2+4}}{2}$. Then by the fact that $M_-$ is austere,
%we have $\tilde{\alpha}\lambda_1\leq\lambda_2\leq\lambda_3\leq\mu_3\leq\mu_2\leq\mu_1<\lambda_4=\mu_4$
Now, we need to consider the following four subcases:

{\bf Case-2-4-1}: $\lambda_1<\lambda_2=\lambda_3<0<\mu_3=\mu_2<\mu_1<\lambda_4
=\frac{\alpha+\sqrt{\alpha^2+4}}{2}=\mu_4$;

{\bf Case-2-4-2}: $\lambda_1=\lambda_2<\lambda_3<0<\mu_3<\mu_2=\mu_1<\lambda_4
=\frac{\alpha+\sqrt{\alpha^2+4}}{2}=\mu_4$;

{\bf Case-2-4-3}: $\lambda_1<\lambda_2<\lambda_3<0<\mu_3<\mu_2<\mu_1<\lambda_4
=\frac{\alpha+\sqrt{\alpha^2+4}}{2}=\mu_4$;

{\bf Case-2-4-4}: $\lambda_1=\lambda_2=\lambda_3<0<\mu_3=\mu_2=\mu_1<\lambda_4
=\frac{\alpha+\sqrt{\alpha^2+4}}{2}=\mu_4$.

In {\bf Case-2-4-1} and {\bf Case-2-4-4}, since $\tilde{\alpha}\geq0$ on $M_-$, and $M_-$ is austere, it must hold $\tilde{\mu}_3<0$
on $M_-$. Then by Proposition \ref{prop:4.1w}, there is a parallel hypersurface $\Phi_{r_0}(M)$
$(-r_2<r_0<0)$ of $M$ such that the corresponding principal curvature $\mu_3(r_0)=\frac{\sin(r_0)+\mu_3\cos(r_0)}{\cos(r_0)-\mu_3\sin(r_0)}=0$.
Then, it holds $\lambda_3(r_0)=-\frac{2}{\alpha(r_0)}$, $\mu_3(r_0)=0$ on $\Phi_{r_0}(M)$.
Now, {\bf Case-2-4-1} and {\bf Case-2-4-4} on $M$ are equivalent to consider
the following two subcases on $\Phi_{r_0}(M)$:

{\bf Case-2-4-1*}: $\lambda_1(r_0)<\lambda_2(r_0)=\lambda_3(r_0)<0=\mu_3(r_0)
=\mu_2(r_0)<\mu_1(r_0)<\lambda_4(r_0)
=\frac{\alpha(r_0)+\sqrt{\alpha^2(r_0)+4}}{2}=\mu_4(r_0)$;

%{\bf Case-2-4-2*}: $\lambda_1(r_0)=\lambda_2(r_0)<\lambda_3(r_0)<0=\mu_3(r_0)
%<\mu_2(r_0)=\mu_1(r_0)<\lambda_4(r_0)
%=\frac{\alpha(r_0)+\sqrt{\alpha^2(r_0)+4}}{2}=\mu_4(r_0)$;

%{\bf Case-2-4-3*}: $\lambda_1(r_0)<\lambda_2(r_0)<\lambda_3(r_0)<0=\mu_3(r_0)
%<\mu_2(r_0)<\mu_1(r_0)<\lambda_4(r_0)
%=\frac{\alpha(r_0)+\sqrt{\alpha^2(r_0)+4}}{2}=\mu_4(r_0)$;

{\bf Case-2-4-4*}: $\lambda_1(r_0)=\lambda_2(r_0)=\lambda_3(r_0)<0=\mu_3(r_0)
=\mu_2(r_0)=\mu_1(r_0)<\lambda_4(r_0)
=\frac{\alpha(r_0)+\sqrt{\alpha^2(r_0)+4}}{2}=\mu_4(r_0)$.

In {\bf Case-2-4-1*}, $\Phi_{r_0}(M)$ belongs to {\bf Case-2-2-8}, which does not occur.

%In {\bf Case-2-4-2*}, $\Phi_{r_0}(M)$ belongs to {\bf Case-2-3-3}, which does not occur.

%In {\bf Case-2-4-3*}, $\Phi_{r_0}(M)$ belongs to {\bf Case-2-3-1}, which does not occur.

In {\bf Case-2-4-4*}, $\Phi_{r_0}(M)$ belongs to {\bf Case-2-1}, which does not occur.

In {\bf Case-2-4-2}, according to the multiplicities of principal curvature $\tilde{\mu}_2$ and
$\tilde{\lambda}_4$ on $M_-$, then for any relationship between the values of
$\tilde{\alpha}$ and
$\tilde{\lambda}_i,\tilde{\mu}_i$, $M_-$ can not be austere.

In {\bf Case-2-4-3}, according to the multiplicity of principal curvature $\tilde{\lambda}_4$
and $\tilde{\alpha}\geq0$ on $M_-$,
then for any relationship between the values of
$\tilde{\alpha}$ and
$\tilde{\lambda}_i,\tilde{\mu}_i$, $M_-$ can not be austere.

In {\bf Case-3}, by $\lambda_-<\frac{\alpha-\sqrt{\alpha^2+4}}{2}$, we have ${\rm dim}V_{\lambda_-}\leq4$. If ${\rm dim}V_{\lambda_-}=4$, then ${\rm dim}V_{\mu_1}=4$
and $\mathcal{Q}=V_{\lambda_-}\oplus V_{\mu_1}$. It contradicts with (3) of Lemma \ref{lemma:5.1}.
So ${\rm dim}V_{\lambda_-}\leq3$ on $M$.

If $\alpha=0$ on $M$, then there are
at least five negative principal curvatures (counted with multiplicities) and no positive principal curvatures on $M_-$,
which contradicts with that $M_-$ is austere.
So the Reeb functions of $M$ and any parallel hypersurface $\Phi_r(M)$ are positive, then  $\tilde{\alpha}\geq0$ on $M_-$.

If $\lambda_->\frac{-2-\sqrt{\alpha^2+4}}{\alpha}$, then $\mu_1<\frac{-2+\sqrt{\alpha^2+4}}{\alpha}$.
We choose a unit vector field $X_1\in V_{\lambda_-}$
and take $(X,Y)=(X_1,JX_1)$ into \eqref{eqn:call}, we get
$$
0\geq2\sum_{\mu_i\neq\lambda_1,\mu_1}\frac{g((\nabla_{e_i} S)X_1,JX_1)^2}{(\lambda_1-\mu_i)(\mu_1-\mu_i)}
=3(1+\lambda_1\mu_1)>0.
$$
It is a contradiction.

If $\lambda_-=\frac{-2-\sqrt{\alpha^2+4}}{\alpha}$, then $\tilde{\alpha}=0$ on $M_-$.
By ${\rm dim}V_{\lambda_-}\leq3$ on $M$, then there are at least two negative principal curvatures
(counted with multiplicities) and no positive principal curvatures on $M_-$,
which contradicts with that $M_-$ is austere.

Thus, for {\bf Case-3}, it must hold $\lambda_-<\frac{-2-\sqrt{\alpha^2+4}}{\alpha}$,
$\mu_1\in(\frac{-2+\sqrt{\alpha^2+4}}{\alpha},\frac{\alpha}{2})$ on $M$.
Note that $\lambda_-<\frac{-2-\sqrt{\alpha^2+4}}{\alpha}$, so $\tilde{\alpha}>0$ holds on $M_-$.
Then by Proposition \ref{prop:4.1w}, on any parallel hypersurface $\Phi_r(M)$, the corresponding principal curvatures satisfy $\lambda_i(r)=\frac{\sin(r)+\lambda_i\cos(r)}{\cos(r)-\lambda_i\sin(r)}
<\frac{\alpha(r)}{2}<\alpha(r)$ and $\mu_i(r)=\frac{\sin(r)+\mu_i\cos(r)}{\cos(r)-\mu_i\sin(r)}<\frac{\alpha(r)}{2}<\alpha(r)$. Thus,
it follows that $\tilde{\lambda}_i<\tilde{\alpha}$ for $\lambda_i\neq\lambda_-$ and $\tilde{\mu}_i<\tilde{\alpha}$ for $\mu_i\neq\lambda_-$ on $M_-$.
In the following, we only need to consider
the following four subcases on $M$:

{\bf Case-3-1}: There are $4$ principal curvatures
(counted with multiplicities) belonging to $(0,\frac{\alpha}{2})$ in $\sigma(\mathcal{Q})$;

{\bf Case-3-2}: There are $3$ principal curvatures
(counted with multiplicities) belonging to $(0,\frac{\alpha}{2})$ in $\sigma(\mathcal{Q})$;

{\bf Case-3-3}: There are $2$ principal curvatures
(counted with multiplicities) belonging to $(0,\frac{\alpha}{2})$ in $\sigma(\mathcal{Q})$;

{\bf Case-3-4}: There are $1$ principal curvatures
(counted with multiplicities) belonging to $(0,\frac{\alpha}{2})$ in $\sigma(\mathcal{Q})$;

{\bf Case-3-5}: There are $0$ principal curvatures
(counted with multiplicities) belonging to $(0,\frac{\alpha}{2})$ in $\sigma(\mathcal{Q})$;

In {\bf Case-3-1}, we have $\lambda_1\leq\lambda_2\leq\lambda_3\leq\lambda_4<0<\mu_4\leq\mu_3\leq\mu_2\leq\mu_1<\alpha$
on $M$. Due that $\tilde{\lambda}_i<\tilde{\alpha}$ for $\lambda_i\neq\lambda_-$ and $\tilde{\mu}_i<\tilde{\alpha}$ on $M_-$, then one can check that $M_-$ is austere if and only if
$\tilde{\lambda}_2<\tilde{\lambda}_3<\tilde{\lambda}_4
<\tilde{\mu}_4<0<\tilde{\mu}_3<\tilde{\mu}_2<\tilde{\mu}_1<\tilde{\alpha}$ and
$\tilde{\lambda}_2+\tilde{\alpha}=\tilde{\lambda}_3+\tilde{\mu}_1
=\tilde{\lambda}_4+\tilde{\mu}_2=\tilde{\mu}_4+\tilde{\mu}_3=0$ on $M_-$.
Then by Proposition \ref{prop:4.1w}, there is a parallel hypersurface $\Phi_{r_0}(M)$
$(-r_2<r_0<0)$ of $M$ such that the corresponding principal curvature $\mu_4(r_0)=\frac{\sin(r_0)+\mu_4\cos(r_0)}{\cos(r_0)-\mu_4\sin(r_0)}=0$.
For the convenience of calculation, without loss of generality,
we just regard the parallel hypersurface $\Phi_{r_0}(M)$ as $M$. Then, it holds $\alpha>0$
and $\lambda_1<\lambda_2<\lambda_3<\lambda_4=-\frac{2}{\alpha}<\mu_4=0
<\mu_3<\mu_2<\mu_1<\alpha$ on $M$.
Then from $\tilde{\lambda}_3+\tilde{\mu}_1=0$, we get
$\alpha=\frac{2 - 2 \lambda_1 (3 \lambda_1 + (-3 + \lambda_1^2) \lambda_3)}{\lambda_3 + \lambda_1 (-3 + \lambda_1^2 - 3 \lambda_1 \lambda_3)}$. Substitute it into $\tilde{\mu}_3+\tilde{\mu}_4=0$, we get
$\lambda_1=\frac{-2\lambda_3}{\lambda_3^2-1}$. Now from $\tilde{\lambda}_2+\tilde{\alpha}=0$, it follows
that $\lambda_2=\frac{\lambda_3(\lambda_3^2+3)}{2}$. So $\tilde{\lambda}_4+\tilde{\mu}_2=0$ implies
that $9-14\lambda_3^2+\lambda_3^4=0$. So $\lambda_3=\pm\sqrt{2}-\sqrt{5}$.
If $\lambda_3=-\sqrt{2}-\sqrt{5}$, then $\lambda_1\approx0.5923$, which contradicts with
$\lambda_1<0$. If $\lambda_3=\sqrt{2}-\sqrt{5}$, then by Proposition \ref{prop:4.3w} and
Proposition \ref{prop:4.4w}, $M_-$ and $M_+$ are all austere.
In this case, we have
\begin{equation}\label{eqn:c61}
\begin{aligned}
&\lambda_1=-2\sqrt{2}-\sqrt{5},\ \lambda_2=10\sqrt{2}-7\sqrt{5},\
\lambda_3=\sqrt{2}-\sqrt{5},\ \lambda_4=\frac{3}{\sqrt{2}-5\sqrt{5}},\\
&\mu_4=0,\ \mu_3=\frac{\sqrt{5}-\sqrt{2}}{2},\ \mu_2=\sqrt{5}-\sqrt{2},\
\mu_1=\sqrt{2}+\frac{1}{\sqrt{5}},\ \alpha=\frac{2(5\sqrt{5}-\sqrt{2})}{3}.
\end{aligned}
\end{equation}

Now, we will use Lemma \ref{lemma:5.3} to prove this case does not occur.
Choose a unit vector field $X_1\in V_{\lambda_1}$. By (3) of Lemma \ref{lemma:5.1}, we have
$AX_1\in (V_{\lambda_2}\oplus V_{\mu_2}\oplus V_{\lambda_3}\oplus V_{\mu_3}\oplus
V_{\lambda_4}\oplus V_{\mu_4})$. Let $a=\|AX_1|_{V_{\lambda_2}\oplus V_{\mu_2}}\|^2$,
$b=\|AX_1|_{V_{\lambda_3}\oplus V_{\mu_3}}\|^2$ and $c=\|AX_1|_{V_{\lambda_4}\oplus V_{\mu_4}}\|^2$, where
$AX_1|_{V_{\lambda_2}\oplus V_{\mu_2}}$, $AX_1|_{V_{\lambda_3}\oplus V_{\mu_3}}$ and $AX_1|_{V_{\lambda_4}\oplus V_{\mu_4}}$ are the projections of $AX_1$ on
$V_{\lambda_2}\oplus V_{\mu_2}$, $V_{\lambda_3}\oplus V_{\mu_3}$ and
$V_{\lambda_4}\oplus V_{\mu_4}$ respectively. So $a+b+c=1$ and $a,b,c\geq0$.

Now, we take $X=X_1$ and $X=JX_1$ into \eqref{eqn:ca2} respectively,
with the use of \eqref{eqn:c61}, we obtain
$$
\begin{aligned}
&(191+61\sqrt{10})a+3(34+11\sqrt{10})b+(76+23\sqrt{10})c=0,\\
&(118-73\sqrt{10})a+(186-78\sqrt{10})b+(-322+88\sqrt{10})c=0.
\end{aligned}
$$
Then by above two equations, we get
$b=\frac{-2-\sqrt{10}-(5+\sqrt{10})a}{2}$, $c=\frac{4+\sqrt{10}+(3+\sqrt{10})a}{2}$,
which contradicts with $a,b,c\geq0$.

In {\bf Case-3-2}, we have $\lambda_1\leq\lambda_2\leq\lambda_3\leq\lambda_4\leq\mu_4\leq0<\mu_3\leq\mu_2\leq\mu_1<\alpha$
on $M$. Due that $\tilde{\lambda}_i<\tilde{\alpha}$ for $\lambda_i\neq\lambda_-$ and $\tilde{\mu}_i<\tilde{\alpha}$ on $M_-$, and according to the fact that $M_-$ is austere,
we only need to consider the following two subcases on $M$:

{\bf Case-3-2-1}: $\lambda_1<\lambda_2<\lambda_3<\lambda_4<\frac{\alpha-\sqrt{\alpha^2+4}}{2}
<\mu_4\leq0<\mu_3<\mu_2<\mu_1<\alpha$;

{\bf Case-3-2-2}: $\lambda_1=\lambda_2<\lambda_3<\lambda_4
=\frac{\alpha-\sqrt{\alpha^2+4}}{2}=\mu_4<0<\mu_3<\mu_2=\mu_1<\alpha$.

In {\bf Case-3-2-1}, due that $\lambda_+=\frac{\alpha+\sqrt{\alpha^2+4}}{2}$, we have $\bar{\mu}_4>0$ on $M_+$.
So by Proposition \ref{prop:4.1w}, there is a parallel hypersurface $\Phi_{r_0}(M)$
$(0<r_0<r_1)$ of $M$ such that the corresponding principal curvature $\mu_4(r_0)=\frac{\sin(r_0)+\mu_4\cos(r_0)}{\cos(r_0)-\mu_4\sin(r_0)}>0$.
Then on $\Phi_{r_0}(M)$, it holds $\lambda_1(r_0)<\lambda_2(r_0)<\lambda_3(r_0)<\lambda_4(r_0)
<\frac{\alpha(r_0)-\sqrt{\alpha^2(r_0)+4}}{2}<0<\mu_4(r_0)<\mu_3(r_0)<\mu_2(r_0)<\mu_1(r_0)
<\frac{\alpha(r_0)}{2}<\alpha(r_0)$. It means that $\Phi_{r_0}(M)$
satisfies {\bf Case-3-1}, which does not occur.

In {\bf Case-3-2-2}, due that $\tilde{\lambda}_i<\tilde{\alpha}$ for $\lambda_i\neq\lambda_-$ and $\tilde{\mu}_i<\tilde{\alpha}$ for $\mu_i\neq\lambda_-$, and
$M_-$ is austere, we have
$\tilde{\lambda}_3<\tilde{\lambda}_4=\tilde{\mu}_4<\tilde{\mu}_3<\tilde{\mu}_2=\tilde{\mu}_1<\tilde{\alpha}$ and $\tilde{\mu}_3=\tilde{\lambda}_4+\tilde{\mu}_1=\tilde{\lambda}_3+\tilde{\alpha}=0$ on $M_-$. From $\tilde{\mu}_3=0$, we have $\alpha=-\frac{2(\lambda_1+\lambda_3)}{\lambda_1\lambda_3-1}$.
By $\alpha>0$, then $\lambda_1\lambda_3-1>0$.
Substitute $\alpha=-\frac{2(\lambda_1+\lambda_3)}{\lambda_1\lambda_3-1}$ into $\tilde{\lambda}_4+\tilde{\mu}_1=0$ and $\tilde{\lambda}_3+\tilde{\alpha}=0$, we have
$$
1+6\lambda_1\lambda_3-2\lambda_3^2+\lambda_1^2(\lambda_3^2-2)=
(1+\lambda_1^2)(1+\lambda_3^2)-4(\lambda_3-\lambda_1)^2=0.
$$
By direct calculation, the above two equations have no real solution. Thus, {\bf Case-3-2-2}
does not occur.
%we can solve that
%$\lambda_1=\frac{4\lambda_3\pm\sqrt{3}(\lambda_3^2+1)}{3-\lambda_3^2}$.
%For $\lambda_1=\frac{4\lambda_3\pm\sqrt{3}(\lambda_3^2+1)}{3-\lambda_3^2}$, %$\tilde{\lambda}_3+\tilde{\alpha}=0$ becomes to $\frac{\pm1}{\sqrt{3}}=0$, which is a contradiction.

In {\bf Case-3-3}--{\bf Case-3-5}, then there are at least four negative principal curvatures (counted with multiplicities)
and at most three positive principal curvatures (counted with multiplicities) on $M_-$.
Thus $M_-$ can not be austere. We get a contradiction.

We have completed the proof of the Theorem \ref{thm:6.4}.
\end{proof}

%%%%%%%%%%%%%%%%%%%%%%%%%%%%%%%%%%%%%%%%%%%%%%%%%%%%%%%%%%%%%%%%%%%%

\vskip 10mm

\begin{flushleft}
Haizhong Li\\
{\sc Department of Mathematical Sciences, Tsinghua University,\\
Beijing, 100084, P.R. China}\\
E-mail: lihz@tsinghua.edu.cn

\vskip 1mm

Hiroshi Tamaru\\
{\sc Department of Mathematics, Graduate School of Science, Osaka Metropolitan University,
3-3-138, Sugimoto, Sumiyoshi-ku, Osaka, 558-8585, Japan}\\
E-mail: tamaru@omu.ac.jp

\vskip 1mm

Zeke Yao\\
{\sc School of Mathematical Sciences, South China Normal University,\\
Guangzhou 510631, P.R. China}\\
E-mail: yaozkleon@163.com

\end{flushleft}

\end{document}